\documentclass[a4paper,10pt]{article}

\usepackage{amsmath,amssymb,amsfonts,color,mathtools,amsthm}
\usepackage{algorithmic,cancel,graphics}
\usepackage{algorithm,booktabs}
\usepackage[mathscr]{eucal}
\usepackage[normalem]{ulem} 
\usepackage{comment}
\usepackage{todonotes}
\usepackage{hyperref}
\usepackage{algorithmic}
\newtheorem{theorem}{Theorem}[section]

\newtheorem{proposition}[theorem]{Proposition}
\newtheorem{example}{Example}
\newtheorem{remark}[theorem]{Remark}

\newcommand{\R}{\mathbb{R}}
\newcommand{\N}{\mathbb{N}}

\DeclareMathOperator{\dist}{dist}
\DeclareMathOperator{\argmin}{argmin}

\title{HJB-RBF based approach for the control of PDEs\thanks{AA was supported by the CNPq research grant 3008414/2019-1 and by a research grant from PUC-Rio. HO was supported by the Coordenação de Aperfeiçoamento de Pessoal de Nível Superior - Brasil (CAPES).}}

\author{Alessandro Alla\thanks{Department of Mathematics, PUC-Rio, Rio de Janeiro, Brasil (\url{alla@mat.puc-rio.br},\url{www.alessandroalla.com})}
\and Hugo Oliveira\thanks{Department of Mathematics, PUC-Rio, Rio de Janeiro, Brasil (\url{oliveira@mat.puc-rio.br})}, 
\and
Gabriele Santin\thanks{Digital Society Center, Bruno Kessler Foundation, Trento (Italy) (\url{gsantin@fbk.eu})} }

\begin{document}

\maketitle

\begin{abstract}
Semi-lagrangian schemes for discretization of the dynamic programming principle are based on a time discretization projected on a state-space grid. 
The use of a structured grid makes this approach not feasible for high-dimensional problems due to the curse of dimensionality. 
% Here, we present a new approach for infinite horizon optimal control problems where the value function is computed using Radial Basis Functions (RBF)\ale{?? possiamo levare RBF?} and Shepard’s moving least squares approximation method on scattered grids. 
Here, we present a new approach for infinite horizon optimal control problems where the value function is computed using Radial Basis Functions (RBF) by the Shepard’s moving least squares approximation method on scattered grids.
We propose a new method to generate a scattered mesh driven by the dynamics and the selection of the shape parameter in the RBF using an optimization routine. This mesh will help to localize the problem and approximate the dynamic programming principle in high dimension. Error estimates for the value function are also provided. Numerical tests for high dimensional problems will show the effectiveness of the proposed method.

% The classical Dynamic Programming (DP) approach to optimal control problems is based on the characterization of the value function as the unique viscosity solution of a Hamilton-Jacobi-Bellman (HJB) equation. The DP scheme for the numerical approximation of viscosity solutions of Bellman equations is typically based on a time discretization which is projected on a fixed state-space grid. The time discretization can be done by a one-step scheme for the dynamics and the projection on the grid typically uses a local interpolation. 

\end{abstract}



%%%%%%%%%%%%%%%%%%%%%%%%%%%%%%%%%%%%%%%%%%%%%%%%%%%%%%
\section{Introduction}
\label{Section1}
\setcounter{section}{1}
\setcounter{equation}{0}
\setcounter{theorem}{0}
\setcounter{algorithm}{0}
\renewcommand{\theequation}{\arabic{section}.\arabic{equation}}
%%%%%%%%%%%%%%%%%%%%%%%%%%%%%%%%%%%%%%%%%%%%%%%%%%%%%

The classical Dynamic Programming (DP) approach to optimal control problems is based on the characterization of the value function as the unique viscosity solution of a Hamilton-Jacobi-Bellman (HJB) equation. The DP scheme for the numerical approximation of viscosity solutions of Bellman equations is typically based on a time discretization which is projected on a fixed state-space grid. The time discretization can be done by a one-step scheme for the dynamics and the projection on the grid typically uses a local interpolation. Clearly, the use of a grid is a limitation with respect to possible applications in high-dimensional problems due to the curse of dimensionality. We refer to \cite{BCD97,FF13} for theoretical results and numerical methods. 

In \cite{JS15}, the HJB equation has been approximated by means of Shepard's moving least square approximation providing error estimates and convergence results. Numerical tests for a selected kernel and structured meshes have been provided for the control of low dimensional problems. A Kriging's interpolation was introduced in \cite{CC20} with Halton meshes up to dimension $10$.

Our work is based on \cite{JS15} with the aim to extend the method to high dimensional problems, e.g. control of PDEs which discretized dimension is, at least, of order $O(10^3)$. We have investigated this approach on scattered meshes since the use of structured meshes is impossible for such high dimensional problems, and thus we resort to a grid driven by the dynamics of the problem, similarly to what is done in \cite{AFS19} where a tree structure has been proposed to approximate the finite horizon optimal control problem and the time dependent HJB equation. The generation of the grid driven by the dynamics will help us to compute feedback controls for a class of initial conditions which is an important novelty in the field. 
% To the best of our knowledge this is the first paper able to reconstruct feedback control for more initial conditions when dealing with the control of a PDE, even if not in the entire domain. \todo{AA: Wanna check this}

Within the RBF approximant, the determination of the kernel's shape parameter is known to be a crucial choice. 
Several techniques exists in the literature to optimize this parameter (see e.g. Chapter 14 in \cite{F15} for a recent survey), but they are usually based on the minimization of the approximation error. 
We introduce instead a novel technique based on the residual of the HJB equation.
% \sout{In particular} Specifically, even if the reconstruction of the Value function requires the computation of several RBF \ale{?? RBF è corretto?} approximant, a unique parameter is globally optimized based on the residual of the Value iteration scheme. 
Specifically, even if the reconstruction of the Value function requires the computation of several Shepard approximant, a unique parameter is globally optimized based on the residual of the Value iteration scheme. 

% By numerical experiments we have observed that the residual is locally convex with respect to the shape parameter and we could apply e.g. a gradient method with respect the shape paremeter. 
The use of the residual in HJB has been introduced in \cite{G97} and also used in the context of RBF and HJB to obtain an adaptive grid \cite{FFJS17}. Here, the residual is used as quantity of interest to select the shape parameter.
We summarize the contributions of our paper and the novelty of our approach:
\begin{enumerate}
    \item generation of an unstructured mesh driven by the dynamics that helps us to localize the problem.
    \item An optimized way to select the shape parameter minimizing the residual of the Value Iteration method.
    \item The possibility to compute feedback control of PDEs for a class of initial conditions.
   \item An error estimate of the reconstruction process that is valid along trajectories of the dynamics.
\end{enumerate}

Let us briefly comment on related literature on the control of PDEs using a dynamic programming approach. In the last two decades there has been a tremendous effort to mitigate the curse of dimensionality with the goal to control PDEs. It is straightforward to see that the discretization of a PDE leads to a large system of ODEs which is hard to control using a DP approach. 

The first approach is related to the combination of model order reduction techniques and dynamic programming. The idea is to use a projection technique, e.g. Proper Orthogonal Decomposition (\cite{BGW15}) to reduce the dimension of the dynamical systems. Then, we can approximate with standard Value Iteration algorithm the correspondent reduced Hamilton-Jacobi-Bellman equation. This turns out to be efficient if the reduction is up to dimension $4$ or $5$. We refer to the pioneer work \cite{KVX04} and to \cite{AFV17} for error estimates of the method. 
A different way to reduce the dimension of the dynamical system is given by pseudo-spectral methods as shown in \cite{KK18}. 
Recently, the solution of Hamilton-Jacobi-Bellman on a tree structure has been proposed in \cite{AFS19} and its coupling with  model order reduction in \cite{AS20}.
% \ale{?? servono questi dettagli?}The tree structure is obtained by a discretization in time of the dynamics following all the possible directions \ale{?? of what?} for a given set of discrete controls. This will avoid the use of the interpolation operator and the possibility to control PDEs without any projection method. A pruning criteria has been introduced to keep the cardinality of the tree reasonable. The method has also been coupled to model order reduction to compute the tree and the value function more efficiently (see \cite{AS20}).

 Other methods concerned e.g. the use of sparse grids for time dependent Ha\-mil\-ton-Jacobi equation \cite{BGGK13} and for the control of wave equation with a DP approach \cite{GK17}, tensor train decomposition \cite{DKK19}, neural networks \cite{DO16a,DO16b} and max-plus algebra \cite{ME07,ME09}. All these approaches are important contributions on the mitigation of the curse of dimensionality.

%  with a local reconstruction of the feedback starting from a given initial condition which still limits applications in real-life problems.
 
 The outline of this paper is as follows. In Section \ref{Section2} we recall the background of the dynamic programming approach. Section \ref{Section3} concerns the radial basis functions interpolation and Shepard's approximation. In Section \ref{Section4}, we propose the novel method for the coupling between RBF and DP. We will discuss and comment the details of our approach. Numerical simulations for the control of two different PDEs are shown in Section \ref{Section5}. Finally, conclusions and future works are driven in Section \ref{Section6}.

In more general terms, the research on RBF approached to the solution of PDEs is very active, and a complete account of this field is beyond the scope of this paper. As an introduction, we refer to \cite{Chen2014,Fornberg2015}.

% \begin{enumerate}
%     \item Feedback control and DPP
%     \item Mitigation of the curse of dimensionality
%     \begin{itemize}
%         \item model reduction
%         \item sparse grids
%         \item polynomial approximation
%         \item neural networks
%         \item tensor decomposition
%         \item tree structure
%         \item max-plus algebra
%             \end{itemize}
%             \item radial basis and HJB
%             \item our contribution
%             \item outline
    
% \end{enumerate}

\section{Dynamic Programming Equations}
\label{Section2}
\setcounter{section}{2}
\setcounter{equation}{0}
\setcounter{theorem}{0}
\setcounter{algorithm}{0}
\renewcommand{\theequation}{\arabic{section}.\arabic{equation}}

This section summarizes the main results for infinite horizon control problem by means of dynamic programming equations. For a complete description we refer e.g. to the manuscripts \cite{BCD97,FF13}.

Let the dynamical system described by
\begin{equation}  
	\begin{cases}
		\dot{y} (t)=f(y(t),u(t)), \text{ } t \in (0,\infty],\\ \label{dyn}
		y(0)=x \text{}\in \mathbb{R}^d,
	\end{cases}
\end{equation}
where $y: (0,\infty] \rightarrow \mathbb{R}^d$ is the state variable, and the control is such that $u\in \mathcal{U} \coloneqq \{ u : (0,\infty] \rightarrow U, \text{ measurable}\}$ with $\mathcal{U}$ the set of admissible controls and $U\subset\R^m$ is a compact set. The function $f:\mathbb{R}^d \times \mathbb{R}^m \rightarrow \mathbb{R}^d $ is  Lipschitz continuous with respect to the first variable with constant $L_f>0$. 
% We also assume that $u(t)$ is a measurable function. 
Under such hypothesis, the existence of a unique solution to system \eqref{dyn} holds true (see \cite{FF13}). 

To select the unknown control $u$, we define the following cost functional $\mathcal{J}: \mathcal{U} \rightarrow \mathbb{R}$ 
\begin{equation}
 \label{costfunctional}
		\mathcal{J}_x(y, u) \coloneqq  \int_0^\infty g(y(s),u(s)) e^{-\lambda s}ds.
\end{equation}
The function $g: \mathbb{R}^d \times \mathbb{R}^m \rightarrow \mathbb{R}$ is the running cost and it is assumed to be bounded and Lipschitz continuous in the first variable. The constant $\lambda > 0$ is a discount factor and the term $e^{-\lambda s}$ guarantees the convergence of the integral for $g$ bounded.

Our optimal control problem reads:
\begin{equation}
	\min_{u \in \mathcal{U}} \mathcal{J}_x(y, u),
\end{equation}
with $y$ being a trajectory that solves \eqref{dyn} corresponding to the initial point $x$ and control $u$. We will use the index $x$ in our notations, to stress the dependence on the initial condition. 
We aim at obtaining the control in a feedback form and, for this reason, we define the Value function as follows:
\begin{equation}
	\label{vf}
	v(x) \coloneqq  \inf_{u \in \mathcal{U}} \mathcal{J}_x(y, u).
\end{equation}

One can characterize the value function in terms of the Dynamical Programming Principle (DPP):
\begin{equation}
	\label{DPP}
	v(x) = \inf_{u \in \mathcal{U}} \bigg\{\int_0^\tau g(y(s),u(s)) e^{-\lambda s}ds + e^{-\lambda \tau}\ v(y_{x}(\tau))\bigg\} \quad  \forall x \in \mathbb{R}^d,\tau>0.
\end{equation}
From \eqref{DPP}, we derive the Hamilton-Jacobi-Bellman (HJB) equations corresponding to the infinite horizon problem:
\begin{equation}
	\label{HJB}
	\lambda v(x) + \sup_{u \in U} \{ -f(x,u) \cdot \nabla v(x) - g(x,u)\}=0, \qquad x\in\mathbb{R}^d,
\end{equation}
where $\nabla v$ is the gradient of $v$. The HJB is a further characterization of the value function by means of a degenerate PDEs whose solution has to be thought in the viscous sense \cite{BCD97}.
Thus, if one is able to solve \eqref{HJB}, it is possible to obtain the optimal feedback control $u^{*}(t)$:
\begin{equation}
	\label{FDB}
	u^{*}(t) = \operatorname*{ arg\,max}_{u \in U} \{ -f(x,u) \cdot \nabla v(x) - g(x,u)\},\qquad x\in\mathbb{R}^d.
\end{equation}

% \todo[inline]{Continuous fixed point}
% \todo[inline]{Idea to discuss: to avoid to repeat the discretization of HJB equation we could introduce here a continuos fixed point and in Section 4 recall the structure of the discretization using a generic interpolation operator and mentioning we will focus on Shepard. If that is ok the section will stop here.}

\section{RBF and Shepard's method}
\label{Section3}
\setcounter{section}{3}
\setcounter{equation}{0}
\setcounter{theorem}{0}
\setcounter{algorithm}{0}
\renewcommand{\theequation}{\arabic{section}.\arabic{equation}}
This section presents a brief explanation of Radial Basis Functions (RBFs) and the Shepard's approximation method. Most of the following material is based on 
the 
book \cite{F07} for the general RBF theory, and on the paper \cite{JS15} for the Shepard method, to which we refer for 
further details.

An RBF $\varphi: \mathbb R^d \rightarrow \mathbb R$ is a radially invariant function $\varphi (x) = \varphi (\|x\|_2)$, which is usually required to be positive definite. Several instances of these $\varphi$ exists (see e.g.~\cite{F07}). In this paper we will consider only radially compactly supported RBFs. A significant example are the Wendland RBFs, which are a family of strictly positive definite RBFs of compact support and of different smoothness (depending on a parameter) (see \cite{W95}). The radial nature of these bases can be used to tune their spread by means of a shape parameter $\sigma>0$, i.e., usually one works with a basis $\varphi^\sigma(z)\coloneqq  \varphi(\sigma\|z\|_2)$. Figure \ref{fig-wend} shows the Wendland RBF
\begin{equation}\label{eq:wendland_rbf}
\varphi^{\sigma}(r) = \max \{0, (1-\sigma r)^6 (35 \sigma^2 r^2  + 18 \sigma r + 3)\}, \;\; r\coloneqq \|x\|,
\end{equation}
with $\sigma = 0.8$ on the left panel and $\sigma = 2$ on the right panel, and we can see how the parameters influence the {\em shape} of the basis functions and makes them flat (left) or spiky (right). The RBF is scaled so that $\varphi^\sigma(0)=1$

\begin{figure}[htbp]
\label{fig-wend}
	\centering 
	\includegraphics[scale = 0.4]{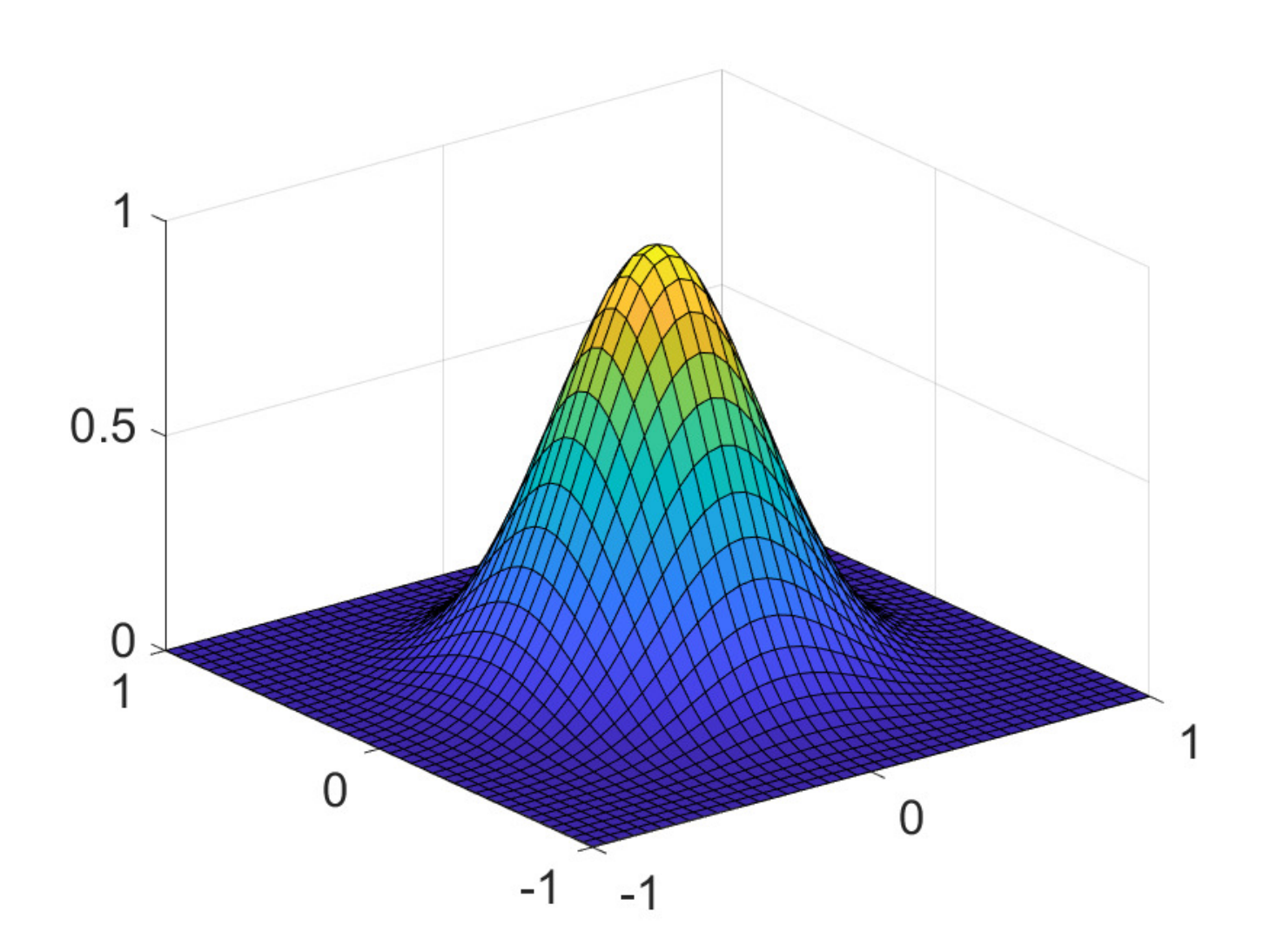}
	\includegraphics[scale = 0.4]{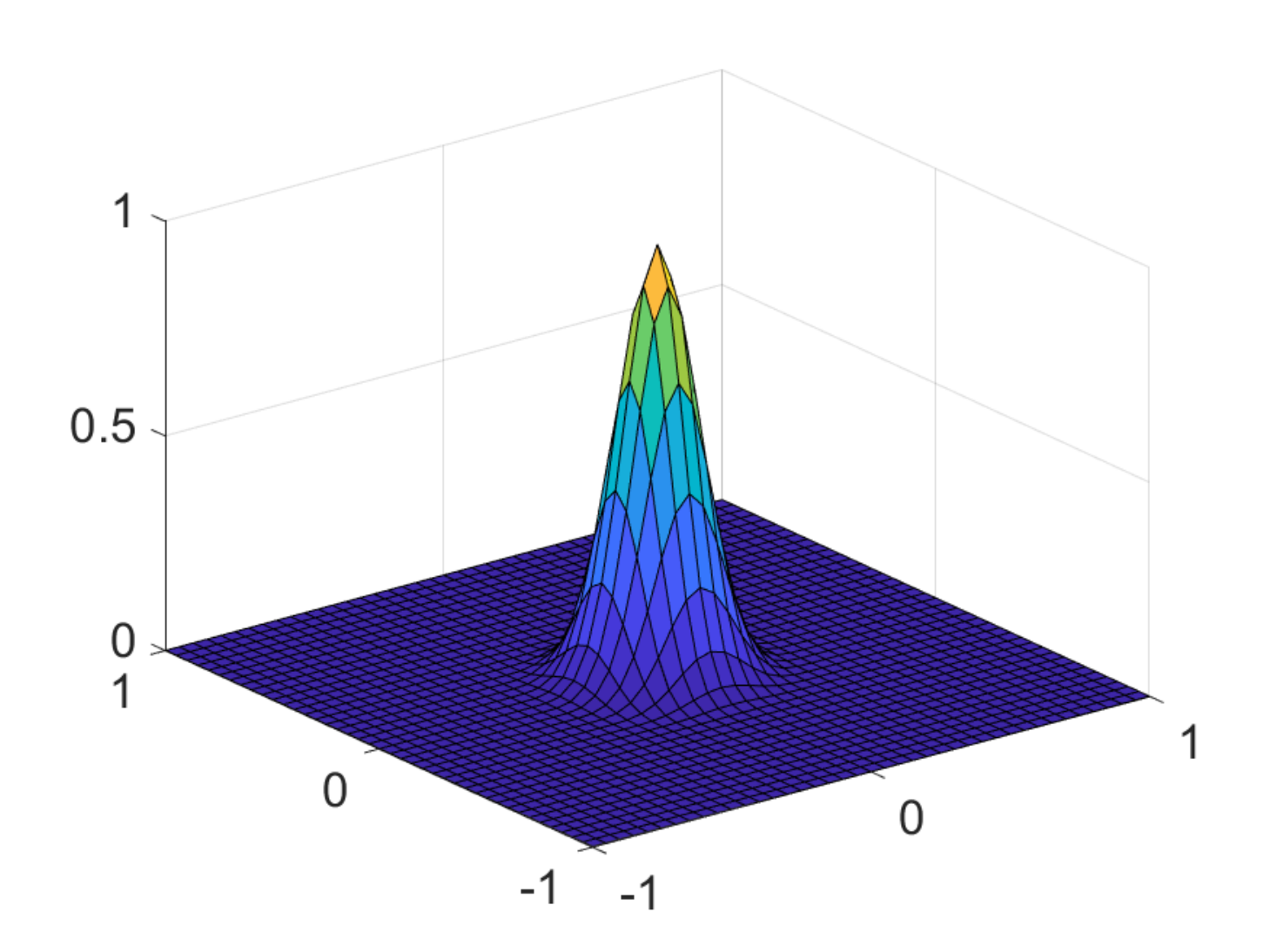}
	\caption{Wendland function with $\sigma = 0.8$ (left) and $\sigma = 2$ (right).}
\end{figure}

RBFs can be used as a tool in interpolation and approximations methods in a mesh-free environment. For this we consider a continuous function $f: \Omega \rightarrow \mathbb{R}$ with $\Omega \subset \mathbb{R}^d$ and a set of pairwise distinct interpolation points (or nodes) $X\coloneqq \{x_1, \dots, x_n\}\subset\Omega$ and the corresponding function evaluations. In this paper we use RBFs in a moving-least squares mode within a Shepard approximation scheme (see e.g. Chapter 23 
in \cite{F07}). 
In this case, the RBF bases are used to form $n$ weights 
\begin{equation}\label{eq:shepard_basis}
\psi^\sigma_i(x)\coloneqq  \frac{\varphi^\sigma(\|x-x_i\|_2)}{\sum_{j=1}^n \varphi^\sigma(\|x-x_j\|_2)}, \;\; 1\leq i\leq n,
\end{equation}
and the Shepard approximant $S^\sigma[f](x)$ is formed as 
\begin{equation}\label{eq:shepard_approx}
    S^\sigma[f](x) \coloneqq \sum_{i=1}^{n} f(x_i)\psi_i^\sigma (x).
\end{equation}
Observe that each $\psi_i(x)$ is compactly supported in $B(x_i, 1/\sigma)\subset\Omega$ and non negative, and the weights form a partition of unity, 
i.e., $\sum_{i=1}^n \psi_i^\sigma(x) = 1$ for all $x\in\Omega_{X, \sigma}$ with
\begin{equation}\label{eq:omega_z_sigma}
\Omega_{X, \sigma}\coloneqq \bigcup_{x\in X} B(x, 1/\sigma)\subset\R^d.
\end{equation} 
This implies that $S^\sigma[f](x)$ is actually a 
convex combination of the function values.
Moreover, the compact support of the weights leads to a computational advantage and a localization. In particular, the Shepard weights are evaluated by constructing a distance vector $D\in\R^{n}$ with $D_{i}\coloneqq \|x-x_i\|_2$ by computing only the entries $D_i$ such that $\|x-x_i\|\leq 1/\sigma$. This operation can be implemented by a range search.

An additional advantage of the Shepard method is that the the construction of the approximant \eqref{eq:shepard_approx} can be directly obtained from the function values and the evaluation of the weights, without solving any linear system.

As RBF-based methods work with unstructured meshes, to obtain error estimates in this context it is common to consider the \emph{fill distance} and the \emph{separation distance}
\begin{equation*}
h\coloneqq h_{X, \Omega}\coloneqq \sup_{x\in \Omega}\min_{x_i\in X}\|x-x_i\|,
\quad \quad q = q_{X} \coloneqq\min_{x_i \neq x_j\in X}\|x_i - x_j \|
\end{equation*}
The fill distance replaces the mesh size and it is the radius of the largest ball in $\Omega$ which does not contain any point from $X$, and it gives a quantification of the well spread of the approximation nodes in the domain. On the other hand, the separation distance
quantifies the minimal separation between different approximation points. We remark that for any sequence of points it holds $\frac{1}{2} q \leq h$, but the inverse inequality $h\leq c q$, $c>0$, does not hold unless the points are asymptotically uniformly distributed.

General error statements for Shepard approximation can be found in \cite{F07}. 
In this paper we will work with the result of \cite{JS15}, that we will recall in the following.
% which is a slight modification of these basic estimates, and that is tailored 
% on Lipschitz continuous functions, as is the case of the Value function, and on compactly supported RBFs. We remark that the error bound obtained in \cite{JS15} decays as  
% $ \frac{1}{\theta} h$
% , assuming that the shape parameter is scaled as $\sigma\coloneqq \theta/h$, with $\theta>0$ a fixed constant.

% \begin{lemma}[Lemma 3 in \cite{JS15}]
% \label{lemma_shep}
% Let $f: \Omega \subset \mathbb{R}^d \rightarrow \mathbb{R}$ be Lipschitz continuous with Lipschitz constant $L_{f}$. Let $\{Z_{n}\}_{n\in\N}\subset\Omega$ be a 
% sequence of sets of $n\in\N$ pairwise distinct nodes, and let $h_n$ denote the fill distance of $Z_n$. Furthermore, let $\theta>0$ be a fixed parameter, and 
% let $\rho>0$ be such that $\supp(\varphi)\subset B(0,\rho)$.
% 
% Then, if $S^{\sigma_n}[f]$ denotes the Shepard approximant of $f$ on $Z_n$ defined by using a shape parameter $\sigma_n\coloneqq  \theta / h_{n}$, for each $n\in\N$ it 
% holds
% \begin{equation*}
% \|f - S^{\sigma_n}[f] \|_{\infty} \leq L_{v} \frac{\rho}{\theta} h_{n}.
% \end{equation*} 
% \end{lemma}

% \newpage
\section{The coupling between DP and RBF}
\label{Section4}
\setcounter{section}{4}
\setcounter{equation}{0}
\setcounter{theorem}{0}
\setcounter{algorithm}{0}
\renewcommand{\theequation}{\arabic{section}.\arabic{equation}}
To overcome the difficulty of solving analitically equation \eqref{HJB}, in this section we introduce the core approach of our work. We first recall the 
semi-Lagrange discretization of \eqref{HJB} by means of Shepard interpolation, then we propose a method to generate unstructured meshes, we define a strategy to select the shape parameter, derive error estimates of the overall approximation method and, finally, we summarize everything in one algorithm.
% \ale{?? add something about error estimate}

We remark that for the purpose of numerical computation we consider only finite horizons $t\in[0, T]$ with a given $T>0$ large enough to simulate the infinite horizon problem. We also assume that the dynamics evolve for each initial value and control parameter within a compact set $\Omega\subset\R^d$. 

\subsection*{Semi-Lagrangian scheme for \eqref{HJB}}
We first chose a temporal step size $\Delta t>0$ and build a grid in time such that $t_k = k \Delta t$ with $k \in \mathbb{N}$. We will discuss in the 
following how to define a spatial discretization, and for now we just denote it as $X = \{x_{1},x_{2}, \dots, x_{n} \}\subset\Omega$.
Furthermore, the set $U$ is also discretized by setting the control $u_k = u(t)\in U$ for $t \in [t_k, t_{k+1})$ constant in the considered time interval.
To introduce the approximation of the value function, we represent the Shepard approximant as an operator 
\begin{equation}\label{eq:shepard_operator}
 S^\sigma: (L^{\infty},\|\cdot\|_{\infty}) \rightarrow (\mathcal{W},\|\cdot\|_{\infty}),
\end{equation}
where $\mathcal{W} = \mbox{span}\{\psi_1^\sigma, \psi_2^\sigma, \cdots \psi_n^\sigma\}$ as in \eqref{eq:shepard_basis}.
% \ale{?? chi sono $\psi_i$}
We remark that the Shepard approximation uses as interpolation nodes the same points $X$ 
defined above.
% GS: sentence removed because it is indeed useless: Moreover, we explicitly denote the dependency on the shape parameter $\sigma>0$. \ale{??}
For now we make no special assumption on its value, while its selection will be 
discussed later in the section.

We aim at the reconstruction of the vector $\{V_j\}_{j=1}^n\in\R^n$ where $V_j$ is the approximate value for $v(x_j)$ for each $x_j\in X$. The full discretization 
of equation \eqref{HJB} is obtained starting from a classical approach (see e.g. \cite{FF13}), but replacing as in \cite{JS15} the 
local linear interpolation operator on a structured grid with the Shepard approximation operator. 
This discretization reads
\begin{equation}
	\label{fully-discrete}
	V_j = [W_\sigma(V)]_{j}\coloneqq \min_{u \in U} \bigg\{  \Delta t\,g(x_j,u) + (1-\Delta t \lambda) S^\sigma[V](x_j + \Delta t\, f(x_j,u))\bigg\}.
\end{equation}
To compute $V_j$ the set $U$ is discretized in a finite number $M\in\N$ of points $U_M\coloneqq \{u_1, \dots, u_M\}$, and the minimum is computed by comparison. The full 
approximation scheme of the Value 
function is known as the Value Iteration (VI) method, and it is obtained by iteration of \eqref{fully-discrete}, i.e.,
\begin{equation}
\label{fixedpoint}
    V^{k+1} = W_\sigma(V^{k}),\quad k =0,1,\ldots.
\end{equation}
To ensure convergence of the scheme it is necessary that $W_\sigma$ is a contraction. In this context, the Shepard operator offers another striking 
benefit in comparison with plain RBF interpolation. Indeed, $S^\sigma$ in
\eqref{eq:shepard_operator} has unit norm, and this implies that the right hand side of \eqref{fully-discrete} is a contraction if $\Delta t \in \left(0,1/\lambda\right]$ (see \cite{JS15} for the details, and especially Lemma 2). Therefore, the convergence of 
the value iteration scheme is guaranteed.

As soon as we obtain an approximation of the value function, we can compute an approximation of the feedback control as 
\begin{equation}
	u_{n}^{*}(x) = \operatorname*{ arg\,min}_{u \in U} \{\Delta t g(x,u) + (1-\lambda \Delta t )S^\sigma[V](x+\Delta t f(x,u)) \},
\end{equation}
with $x=y(t_{n})$.  Thus, we are able to perform a reconstruction of an optimal trajectory $y^{*}$ and optimal control $u^{*}$.

% \gab{
Under the assumption that the fill distance $h$ decays to zero and that the shape parameter scales as $\sigma = \theta / h$, in \cite{JS15} it is proven 
that this approximation scheme converges. More precisely, under suitable assumptions on $f$ and $g$ Theorem 3 in \cite{JS15} guarantees that $\|v-V\|_{\infty}\leq (C/\theta) h$, where $C$ depends on the dynamics but not on the discretization.

% if $f :\Omega \times U\rightarrow \Omega$ in \eqref{dyn} is continuously differentiable and if the 
% associated running cost $g: \Omega \times U \rightarrow [0,\infty)$ is Lipschitz continuous in $x$, the theorem proves that
% \begin{equation}
%     \|v-V\|_{\infty} \le \dfrac{L_v}{\theta(1-e^{-\delta})}h_n\;\;\forall n\in\N,
% \end{equation}
% where $L_v>0$ is the Lipschitz constant of the value function, and $\delta>0$ is a lower bound ... \todo{AA: mettiamo una costante per h/theta dove la costante dipende da..}
% }

Despite these convincing theoretical guarantees, the requirement that $h=h_{X, \Omega}$ decays to zero is too restrictive in our setting, since a filling of the entire 
$\Omega$ may be out of reach for high dimensional problems. Moreover, as already mentioned in Section \ref{Section3}, Shepard 
method perform approximations in high dimensions and unstructured grids, while in \cite{JS15} the authors focused on a given configuration for the shape 
parameter and an equidistant grid. In the next paragraphs we will explain how to select the shape parameter and to generate unstructured meshes to solve high dimensional problems.

% We recall that this approximation method is known to be convergent, as proven in the following 
% theorem. 
% \begin{theorem}[Theorem 3 in \cite{JS15}]
% \label{theorem_1}
% For each $n\in\N$, let $X_n$, $h_n$, $\sigma_n$ $\theta$, $S_n$ be defined as in Lemma \ref{lemma_shep}.
% Assume that $f :\Omega \times U\rightarrow \Omega$ in \eqref{dyn} is continuously differentiable and that the associated running cost $g: \Omega \times U \rightarrow [0,\infty)$ is Lipschitz continuous in $x$. 
% Let $v$ be the solution of the continuous equation \eqref{HJB} and $V$ the fixed point of \eqref{fixedpoint}. Then\todo{AA: need to introduce $\delta$}
% \begin{equation}
%     \|v-V\|_{\infty} \le \dfrac{L_v\rho}{\theta(1-exp(-\delta))}h_n\;\;\forall n\in\N.
% \end{equation}
% \end{theorem}

\subsection*{On the Scattered Mesh}
Different possibilities are available for the definition of the discretization $X$ of the spatial domain. A standard choice is to use an equi-distributed grid, 
which covers the entire space and usually provides accurate results for interpolation problems. Unfortunately, for higher dimensional problems it is 
impossible to think to work on equi-distributed grid, as their size grows exponentially. 
This is a particular limitation in our case, since our goal is to control PDEs, whose discretization leads to high dimensional problems e.g. $d>10^3$. 
On the other hand, a random set of points is computationally efficient to generate and to use, but in this case additional care should be taken because the distribution of points can be irregular (some regions can be more densely populated than others) and the fill distance may  decrease only very slowly when increasing the number of points. 

In general terms, there is a tradeoff between keeping the grid at a reasonable size and the need to cover the relevant part of the computational domain. In 
particular, it is well known that the fill distance for any sequence of points $\{X_n\}_{n\in\N}$ can at most decrease as $h\leq c_\Omega n^{-1/d}$ in $\R^d$ 
for a suitable constant $c_\Omega>0$ depending only on the geometry of the domain. Observe that uniform points have precisely this asymptotic decay of the fill distance.
Thus, an exponentially growing number of points is required to obtain a good covering of $\Omega$ as $d$ increases.

The key point to overcome these limitations is to observe that the evolution of the system provides itself an indication of the regions of interest within the 
domain. Following this idea, we propose a discretization method driven by the dynamics of the control problem \eqref{dyn}. Observe that a similar idea has been 
used in \cite{SH18} to compute the value function along the trajectories of open loop control problems. Also in \cite{AFS19} the grid has been generated by points of the dynamics leading to the solution of the HJB equation on a tree structure.

To define our dynamics-dependent grid we fix a time step $\overline{\Delta t}>0$, a maximum number $\overline K\in\N$ of discrete times and, for $\overline 
L, \overline M >0$, some initial conditions of interest and a discretization of the control space, i.e., 
\begin{equation*}
\overline{X}\coloneqq \{\bar x_1, \bar x_2, \ldots, \bar x_{\overline L}\} \subset \Omega, \quad
\overline{U}\coloneqq \{\bar u_1,\bar u_2, \ldots, \bar u_{\overline{M}}\}\subset U.
\end{equation*}
Observe that all these parameters do not need to coincide with the ones used in the solution of the value iteration \eqref{fully-discrete}, but are rather 
only used to construct the grid. Moreover, in general we use $\overline{\Delta t} > \Delta t$ and $\overline{M}< M$, i.e., the discretization 
used to construct the mesh is coarser than the one use to solve the control problem.

For a given pair of initial condition $\bar x_i\in \overline{X}$ and control $\bar u_j\in \overline{U}$ we solve numerically \eqref{dyn} to obtain trajectories
\begin{align}\label{grid-dyn}
x_{i,j}^{k+1} & = x_{i,j}^k + \overline{\Delta t}\, f(x_{i,j}^k, \bar u_j), \quad k=1,\ldots,\bar K-1,\\
x_{i,j}^1& = \bar x_i,\notag
\end{align}
such that $x_{i,j}^k$
% $z_{i,j}^k$ 
% \todo{HO: Isn't it $x$ instead of $z$? AA: Think so too} \
is an approximation of the state variable with initial condition $\bar x_i$, constant control $u(t)=\bar u_j$, at time $k 
\overline{\Delta t}$. For each pair $\left(\bar x_i, \bar u_j\right)$ we obtain the set $X\left(\bar x_i, \bar u_j\right) 
\coloneqq \{x_{i,j}^1,\ldots,x_{i,j}^{\bar K}\}$ containing the discrete trajectory, and our mesh is defined as
\begin{equation}\label{eq:mesh}
X\coloneqq X(\overline X, \overline U, \overline \Delta t, \overline K)\coloneqq \bigcup\limits_{i=1}^{\bar L} \bigcup\limits_{j=1}^{\bar M} X\left(\bar x_i, \bar u_j\right).
\end{equation}
 
This choice of the grid is particularly well suited for the problem under consideration, as it does not aim at filling the space $\Omega$, but instead it 
provides points along trajectories of interest. In this view, the values of $\overline X, \overline U, \overline \Delta t, \overline K$ should be 
chosen so that $X$ contains points that are suitably close to the points of interest for the solution of the control problem. 
In the following proposition we provide a quantitative version of this idea, that will be the base of our error estimate in Theorem \ref{theo:us}.
% \ale{?? Ref.}

\begin{proposition}\label{prop:mesh}
Let $X\coloneqq X(\overline{X}, \overline{U}, \overline{\Delta t}, \overline{K})$ be the dynamics-dependent mesh of \eqref{eq:mesh}, and assume that $f$ is 
uniformly 
bounded i.e., there exist $M_f>0$ such that
\begin{align*}
\sup\limits_{x\in\Omega, u\in U} \|f(x, u)\| &\leq M_f.
\end{align*}
Then for each $x\in X$, $\Delta t>0$ and $u\in U$ it holds 
\begin{align}\label{eq:bound_distance_1}
\dist(x + \Delta t f(x, u), X) \leq & M_f \Delta t.
\end{align}
Assume furthermore that $f$ is uniformly Lipschitz continuous in both variables, i.e., there exist $L_x, L_u>0$ such that
\begin{align*}
\|f(x, u)-f(x', u)\| &\leq L_{x} \|x-x'\|\;\; \forall x, x'\in \Omega,\;\; u\in U,\\
\|f(x, u)-f(x, u')\| &\leq L_{u} \|u-u'\|\;\; \forall x\in \Omega,\;\; u, u'\in U.
\end{align*}
Then, if $x\coloneqq x^k(x_0, u, \Delta t)\in\Omega$ is a point on a discrete trajectory with initial point $x_0\in\Omega$, control $u\in U$, timestep $\Delta t>0$, 
and time 
instant $k\in \N$, $k\leq \bar K$, it holds
\begin{align}\label{eq:bound_distance_2}
\dist(x&, X)
\leq \left(|\Delta t - \overline{\Delta t}| \bar K M_f + \min_{\bar x\in \bar X} \|\bar x- x_0\| + \bar K \overline{\Delta t} L_u \min_{\bar u\in \bar U} 
\|\bar u- u\|\right) e^{\bar K \overline \Delta t L_x}.
\end{align}
\end{proposition}
\begin{proof}
If $x\in X$ we simply have
\begin{align*}
\dist(x + \Delta t f(x, u), X) 
&= \min\limits_{x'\in X} \|x + \Delta t f(x, u) - x'\|
\leq\|x + \Delta t f(x, u) -x\|\\
&=\Delta t \|f(x, u)\| \leq M_f \Delta t,
\end{align*}
which gives the bound \eqref{eq:bound_distance_1} using only the boundedness of $f$.
 
To prove \eqref{eq:bound_distance_2} we need to work explicitly with the initial points and the control values. By assumption we have 
\begin{equation*}
x = x^{k}(x_0, u, \Delta t) = x_0 + \Delta t\sum_{p=0}^{k-1} f(x^p(x_0, u, \Delta t), u),
\end{equation*}
and since $x'\in X$, using the definition \eqref{eq:mesh} we can choose it as
\begin{equation*}
x' = x_{\ell, m}^{k} = \bar x_\ell + \overline{\Delta t}\sum_{p=0}^{k-1} f(x_{\ell m}^p,  \bar u_m),
\end{equation*}
for some $\ell\in\{1, \dots, \bar L\}$, $m\in\{1,\dots, \bar M\}$. It follows that
\begin{align*}
x^{k}(x_0, u, \Delta t) -  \bar x_{\ell, m}^{k} 
=&\phantom{+} x_0-\bar x_\ell \\
&+ \Delta t\sum_{p=0}^{k-1} f(x^p(x_0, u, \Delta t), u) - \overline{\Delta t}\sum_{p=0}^{k-1} f(x_{\ell m}^p, \bar u_m),
\end{align*}
and thus adding and subtracting $\overline{\Delta t}\sum_{p=0}^{k-1} f(x^p(x_0, u, \Delta t), u)$ gives
\begin{align*}
\|x^{k}(x_0, u, \Delta t) -  x_{\ell, m}^{k}\| 
\leq&\phantom{+} \|x_0-\bar x_\ell\| 
+ |\Delta t - \overline{\Delta t}| \sum_{p=0}^{k-1} \|f(x^p(x_0, u, \Delta t), u)\|\\
&+  \overline{\Delta t} \sum_{p=0}^{k-1} \|f(x^p(x_0, u, \Delta t), u) - f(x_{\ell m}^p, \bar u_m)\|\\
\leq&\phantom{+} \|x_0-\bar x_\ell\| 
+ |\Delta t - \overline{\Delta t}| k M_f \\
&+  \overline{\Delta t} \sum_{p=0}^{k-1} \|f(x^p(x_0, u, \Delta t), u) - f(x_{\ell m}^p, \bar u_m)\|.
\end{align*}
Now adding and subtracting $f(x^p(x_0, u, \Delta t), \bar u_m)$ in the sum, and using the Lipschitz continuity of $f$, we get
\begin{align*}
\|x^{k}(x_0, u, \Delta t) -  x_{\ell, m}^{k}\| 
\leq&\phantom{+} \|x_0-\bar x_\ell\| 
+ |\Delta t - \overline{\Delta t}| k M_f \\
&+  \overline{\Delta t} \sum_{p=0}^{k-1} \left(L_u \|u - \bar u_m\| + L_x \|x^p(x_0, u, \Delta t) - x^p_{\ell,m}\|\right)\\
=&\phantom{+} \|x_0-\bar x_\ell\| 
+ |\Delta t - \overline{\Delta t}| k M_f + k \overline{\Delta t} L_u \|u - \bar u_m\|\\
&+ \overline{\Delta t} L_x\sum_{p=0}^{k-1} \|x^p(x_0, u, \Delta t) - x^p_{\ell,m}\|.
\end{align*}
Applying the discrete Gr{\"o}nwall lemma to this inequality gives 
\begin{align*}
\|x^{k}(x_0, u, \Delta t) -  x_{\ell, m}^{k}\|
&\leq \left(\|x_0-\bar x_\ell\| + |\Delta t - \overline{\Delta t}| k M_f + k \overline{\Delta t} L_u \|u - \bar u_m\|\right) e^{k \overline \Delta t L_x},
\end{align*}
and since $\ell$ and $m$ are free, we can choose them as $\bar u_m\coloneqq \argmin_{\bar u\in \bar U} \|u - \bar u\|$ and $x_\ell\coloneqq \argmin_{\bar x\in \bar X} \|x_0- 
\bar x\|$. Finally, bounding $k$ by $\bar K$ gives \eqref{eq:bound_distance_2}.
\end{proof}

\subsection*{Selection of the shape parameter} 
The quality of Shepard approximation strongly depends from the choice of shape parameter $\sigma$, both in general for RBF approximation \cite{F07} and in the 
special case of the solution of control problems \cite{JS15}.
As mentioned in the introduction, several techniques exists to tune the shape parameter in the RBF literature, such as cross validation and maximum likelihood estimation (see e.g. Chapter 14 in \cite{F15}), but they are designed to optimize the value of $\sigma$ in a fixed approximation setting. 
In our case, on the other hand, we need to construct an approximant at each iteration  $k$ within the value iteration \eqref{fully-discrete}. This makes the 
existing methods computationally expensive and difficult to adapt to the target of minimizing the error in the iterative method. 

For these reasons, we propose here a new method to select the shape parameter based on the minimization of a problem-specific indicator, namely the residual 
$R(\sigma)\coloneqq R(V_{\sigma})$. Assuming that the value iteration with parameter $\sigma$ has been stopped at iteration $k_{\mathrm{final}}$ giving the solution 
$V_{\sigma}\coloneqq V^{k_{\mathrm{final}}}$, we define the residual as 
\begin{equation}\label{residual}
        R(\sigma)\coloneqq  \|V_{\sigma} - W_\sigma(V_{\sigma}) \|_\infty,
\end{equation}
and we choose the shape parameter that minimizes this quantity with respect to $\sigma$. 

To get a suitable scale for the values of $\sigma$, we parametrize it in terms of the grid $X$, similarly to what is done in \cite{JS15} by setting $\sigma\coloneqq  
\theta/h_{\Omega, X}$ for a given $\theta >0$. Since $h_{\Omega, X}$ is difficult to compute or even to estimate in high dimensional problems, we resort to 
setting $\sigma = \theta/q_X$ and we optimize the value of $\theta > 0$. Observe that the separation distance $q_X$ is an easily computable quantity that 
depends only on $X$, and it is thus actually feasible to use this parametrization even in high dimension. 

Choosing an admissible set of parameters $\mathcal{P} \coloneqq [\theta_{\min}, \theta_{\max}]\subset \mathbb{R}^{+}$, the parameter is thus chosen by solving 
the optimization problem	
 \begin{equation}\label{opt-shape}
\bar\theta\coloneqq \argmin\limits_{\theta\in\mathcal{P}} R(\theta/q_X) = \argmin\limits_{\theta\in\mathcal{P}}\|V_{\theta/q_X} - W_{\theta/q_X}(V_{\theta/q_X}) 
\|_\infty.
\end{equation} 
\begin{remark}
This problem can be solved by using a comparison method or e.g. an inexact gradient method. The former means to discretize the set 
$\mathcal{P}$ as $\{\theta_1,\ldots,\theta_{N_p}\}\subset\mathcal{P}$ and to compute all the value functions for all $\theta_i$, $i=1,\ldots,N_p.$ The latter 
considers a projected gradient method where the parameter space $\mathcal{P}$ is continuous and the derivative is computed numerically as 
$$ R_\theta =\dfrac{R({\theta+\varepsilon}) - R({\theta})}{\varepsilon},$$
for some fixed $\varepsilon>0$. 
In the numerical tests, we will compare both minimization strategies in the low dimensional case, while we will concentrate on the comparison method in high dimensional one.
\end{remark}

\subsection*{Error Estimates}

We adapt the classical convergence theory that is used to prove rates of convergence for the value iteration when 
% polynomial \ale{?? linear} 
linear interpolation is used. In 
particular, the following argumentation follows the discussion in \cite[Section 8.4.1]{FF13}.

The idea is to estimate the time and space discretizations separately. Since the time discretization is independent of the interpolation scheme used in the 
space discretization, we just recall the following result from \cite[Section 8.4.1]{FF13}. Although the original result holds for the supremum norm 
over $\Omega$, we formulate it here for a general compact subset $\tilde \Omega\subset \Omega$ in order to deal with the space discretization later. We use the notation $\|f\|_{\infty, \tilde\Omega}\coloneqq  \sup_{x\in \tilde\Omega}|f(x)|$.
% \ale{?? abbiamo bisogno di tutto il teorema. vediamo quanto accorciamo.}
% \ale{Norma $\|\cdot\|_{\infty,\bar{\omega}}$ non definita}
\begin{theorem}\label{thm:error_estimate}
Let $v$ be the exact value function, and let $v^{\Delta t}$ be the solution of the value iteration \eqref{fully-discrete} without space discretization, i.e., 
\begin{align}\label{eq:semi_discrete}
v^{\Delta t}(x) = \min_{u \in U} \bigg\{  \Delta t\,g(x,u) + (1-\Delta t \lambda) v^{\Delta t} (x + \Delta t\, f(x,u))\bigg\}.
\end{align}
If $v$ is Lipschitz continuous, for each compact subset $\tilde \Omega\subset \Omega$ there exists a constant $C\coloneqq C(\tilde\Omega)>0$ such that
\begin{align}\label{eq:time_error}
\left\|v- v^{\Delta t}\right\|_{\infty,\tilde\Omega}\leq C \Delta t^{1/2}.
\end{align}
Assume additionally that the following hold:
\begin{enumerate}
\item $f$ is uniformly bounded, $U$ is convex, $f(x, u)$ is linear in $u$.
\item $g(\cdot, u)$ is Lipschitz continuous and $g(x, \cdot)$ is convex.
\item There exists an optimal control $u^\star \in \mathcal U$.
\end{enumerate}
Then there exists $C'\coloneqq C'(\tilde\Omega)>0$ such that
\begin{align}\label{eq:time_error_refined}
\left\|v- v^{\Delta t}\right\|_{\infty,\tilde\Omega}\leq C' \Delta t.
\end{align}
\end{theorem}

It remains now to quantify the error that is committed by introducing a space discretization, i.e., the error associated to the interpolation scheme in 
\eqref{fully-discrete}. This first requires a bound on the error of Shepard interpolation, and we report in the next proposition a result obtained in \cite{JS15}. 
In this case, the key idea is to scale the shape parameter of the RBF basis according to the fill distance of the mesh. According to Proposition 
\ref{prop:mesh}, we can control the fill distance $h_{X, \tilde \Omega}$ of the mesh $X$ within the set $\tilde \Omega$ obtained as the collection of 
the different trajectories of the discrete dynamics. In other words, we set
\begin{equation}\label{eq:set_omega_tilda}
\tilde \Omega\coloneqq  \tilde \Omega(\tilde X, \tilde U, \tilde T)\coloneqq  \left\{x\coloneqq  x^k(x_0, u, \Delta t): x_0\in \tilde X, u\in \tilde U, \Delta 
t\in \tilde T, k\leq \bar K\right\}.
\end{equation}
For this set, equation \eqref{eq:bound_distance_2} gives \begin{align}\label{eq:fill_distance_estimate}
&h_{X, \tilde\Omega} 
\coloneqq \sup\limits_{x\in\tilde\Omega}\dist(x, X)\nonumber\\
&\leq \left(\sup\limits_{\Delta t\in \tilde T}|\Delta t - \overline{\Delta t}| \bar K M_f + \sup\limits_{x_0\in \tilde X}\min\limits_{\bar x\in \bar X} \|\bar 
x- x_0\| + \bar K \overline{\Delta t} L_u \sup\limits_{u\in \tilde U}\min_{\bar u\in \bar U} 
\|\bar u- u\|\right) e^{\bar K \overline \Delta t L_x}
\end{align}

We can now recall from \cite{JS15}the error estimate for the Shepard approximation.
\begin{proposition}\label{prop:shepard_error}
Let $L_v>0$ be the Lipschitz constant of $v:\Omega\to \R$. Let $S^\sigma$ be the Shepard approximation of $v$ on $X$
obtained using a kernel with support contained in $B(0,1)$, and let $\sigma\coloneqq C / h_{X, \tilde\Omega}$ for a positive constant $C>0$. 
Then we have the bound
\begin{equation}\label{eq:error_shepard}
\left\|v - S^\sigma[v]\right\|_{\infty, \tilde\Omega} \leq C L_v h_{X, \tilde\Omega}.
\end{equation}
\end{proposition}
% Additionally, we have the following simple result that can be obtain by a simple application of the Cauchy-Schwartz inequality.
% \begin{proposition}
% Assume that each Shepard basis function $\psi_j$ is Lipschitz continuous with constant $L_{\psi_j}$. Then for any continuous function $v:\Omega\to\mathbb{R}$ 
% and for any $\sigma>0$, the Shepard interpolant $S^\sigma[v]$ is Lipschitz continuous with constant 
% \begin{align*}
% L_{S^\sigma[v]} \leq \sum_{j=1}^n |v(x_j)| L_{\psi_j}.
% \end{align*}
% \end{proposition}

With these tools, we can finally prove the following theorem.
\begin{theorem}\label{theo:us}
Let $\tilde \Omega$ be as in \eqref{eq:set_omega_tilda}, and assume that $\tilde U$ contains the two controls that are optimal for $v^{\Delta t}$ and for $V$. 
Then, under the assumptions of Proposition \ref{prop:shepard_error} it holds
\begin{align}\label{eq:thm}
&\|V - v^{\Delta t}\|_{\infty, \tilde \Omega} 
\leq \frac{C L_v}{\lambda} \frac{h_{X,\tilde \Omega}}{\Delta t}\\
&\leq \frac{C L_v}{\lambda} \frac{e^{\bar K \overline \Delta t L_x}}{\Delta t} \left(\sup\limits_{\Delta t\in \tilde T}|\Delta t - \overline{\Delta t}| \bar K 
M_f + \sup\limits_{x_0\in \tilde X}\min\limits_{\bar x\in \bar X} \|\bar 
x- x_0\| + \bar K \overline{\Delta t} L_u \sup\limits_{u\in \tilde U}\min_{\bar u\in \bar U} 
\|\bar u- u\|\right) 
\end{align}
\end{theorem}
\begin{proof}
We consider the control $u^\star$  that is optimal for $v^{\Delta t}$, and we define $z^\star\coloneqq  x_j+\Delta tf(x_j, u^\star)$ for $1\leq j\leq n$. 
This implies in particular that $u^\star$ is sub-optimal for $V$, and thus by the definition \eqref{fully-discrete} for each $1\leq j\leq n$ it holds
\begin{equation*}
V_j \leq  \Delta t\,g(x_j, u^\star) + (1-\Delta t \lambda) S^\sigma[V] (z^\star).
\end{equation*}
Combining this inequality with the definition \eqref{eq:semi_discrete} of $v^{\Delta t}(x_j)$, we have 
\begin{align*}
V_j - v^{\Delta t}(x_j)
&\leq  \Delta t\,g(x_j, u^\star) + (1-\Delta t \lambda) S^\sigma[V] (z^\star)
-\Delta t\,g(x_j,u^\star) - (1-\Delta t \lambda) v^{\Delta t} (z^\star)\\
&=(1-\Delta t \lambda)\left(S^\sigma[V] (z^\star)-v^{\Delta t} (z^\star)\right).
\end{align*}
Now we add and subtract the interpolation $S^\sigma[v^{\Delta t}]$ of $v^{\Delta t}$ evaluated at $z^\star$, and use the bound \eqref{eq:error_shepard} 
and the fact that $S^\sigma$ is a contraction to obtain
\begin{align*}
V_j - v^{\Delta t}(x_j)
&\leq(1-\Delta t \lambda)\left(S^\sigma[V]  (z^\star)-S^\sigma[v^{\Delta t}]  (z^\star) + S^\sigma[v^{\Delta t}]  (z^\star)-v^{\Delta t} (z^\star)\right)\\
&\leq (1-\Delta t \lambda)\left(\|V - v^{\Delta t}\|_{\infty, \tilde \Omega} + C L_v h_{X,\tilde \Omega}\right).
\end{align*}
We can now repeat the same reasoning using instead the control $u^*$ that is optimal for $V$, and in this way we can obtain a bound on the opposite quantity 
$v^{\Delta t}(x_j) - V_j$. These two inequalities, when combined, give the desired bound:
\begin{equation*}
\|V - v^{\Delta t}\|_{\infty, \tilde \Omega} 
\leq \frac{C L_v}{\lambda} \frac{h_{X,\tilde \Omega}}{\Delta t}.
\end{equation*}
\end{proof}

By triangular inequality from \eqref{eq:time_error_refined} and \eqref{eq:thm}, one can obtain
$$\|v - V\|_{\infty,\tilde\Omega}\leq C'\Delta t + \frac{C L_v}{\lambda} \frac{h_{X,\tilde \Omega}}{\Delta t}.$$

\subsection*{Algorithm}
The algorithm is summarized in Algorithm \ref{algo1}:
 \begin{algorithm}[H]
 	\caption{Value Iteration with shape parameter selection}
 	\label{algo1}
 	\begin{algorithmic}[1]
% %	\SetAlgoLined
 \STATE INPUT: $\Omega,\Delta t, U, \mathcal{P}$ parameter range, tolerance, RBF and system dynamics $f$, flag
 \STATE	initialization\;
 \STATE	Generate Mesh
 \IF{flag == Comparison}
 	\FOR{$\theta \in P$}
 	\STATE Compute $V_{\theta}$\; 
 		\STATE		$R(\theta) = || V_{\theta} - W(V_{\theta})||_\infty$
 	\ENDFOR
 \STATE	$\bar{\theta} = \underset{\theta \in P}{\arg \min } \text{ }  R(\theta)$\;
 	\ELSE
 \STATE $R_\theta = 1, \theta = \theta_0$, tol, $\varepsilon$
 \WHILE{$\|R_\theta\|>tol$}
 \STATE Compute $V_{\theta}$ and $V_{\theta+\varepsilon}$
 \STATE	Evaluate $R(V_\theta)$ and $R(V_{\theta+\varepsilon})$
  \STATE   $R_\theta =\dfrac{R(V_{\theta+\varepsilon})- R(V_{\theta})}{\varepsilon}$
 %\STATE err = $\|R_\theta\|$
 \STATE $\theta = \theta - R_\theta$
 \STATE $\theta=\max(\min(\theta_{\max},\theta),\theta_{\min})$ \qquad (projection into $\mathcal{P}$) 
\ENDWHILE
\STATE $\bar{\theta} = \theta, V_{\bar{\theta}} = V_\theta$
 	\ENDIF
 	\STATE OUTPUT: $\{\bar{\theta}, V_{\bar{\theta}}\}$ 
 \end{algorithmic}
 \end{algorithm}

% \begin{algorithm}[H]
% 	\SetAlgoLined
% 	\KwIn{$\Omega$,$\Delta t$, $U$, $P$ parameter range, tolerance, RBF and system dynamics $f$}
% 	initialization\;
% 		X = GenerateMesh($f$,$U$,$\Delta t$, points in $\Omega$)\;
% 	\ForEach{$\theta \in P$}{Value Iteration algorithm(VI) and save the result as $V_{\theta}$\; 
% 				$V_{\theta} = VI (X, \Delta t, U, tol)$
% 				$R(\theta) = || V_{\theta} - W(V_{\theta})||_1$
% 	}
% 	$\bar{\theta} = \underset{\theta \in P}{\arg \min } \text{ }  R(\theta)$\;
% 	\KwOut{ $\bar{\theta}$, $V_{\bar{\theta}}$ }
% 	\caption{Value Iteration with automatic parameter selection}
% 	\label{algo1}
% \end{algorithm}

%\newpage
\section{Numerical experiments}
\label{Section5}
\setcounter{section}{5}
\setcounter{equation}{0}
\setcounter{theorem}{0}
\setcounter{algorithm}{0}
\setcounter{example}{0}
\renewcommand{\theequation}{\arabic{section}.\arabic{equation}}

% \todo[inline]{Stress (i) it is a-posteriori, (ii) end time, (iii) we use the same value function with different initial conditions (iv) localize version (v) compatibilità dei parametri}
In this section we present three numerical tests to illustrate the proposed algorithm. The first test is a two dimensional minimum time problem with well-known analytical solution. In this test we analyse results using regular and scattered grids. The second and third test deal with a advection equation and a nonlinear heat equation, respectively. We discretize in space both PDEs using finite differences. The dimension of the semi-discrete problem will be $10201$ for the advection equation and $961$ for the parabolic problem. 
% To generate the grid in the PDE case we fix a class of initial conditions. 
For each example we provide examples of feedback and optimal controls reconstructed for different initial conditions that may not belong to the grid. In the parabolic case we present the effectiveness of feedback control under disturbances of the system. %In the hyperbolic case we present results of the control of a spiky initial condition that does not belong to the class of functions considered in this case.

In every experiment we define an admissible interval $\mathcal{P}$ to solve the minimization problem by comparison in Algorithm \ref{algo1} as follows: we start with a large interval $\mathcal{P}_1$ and we coarsely discretize it. Then, we run Algorithm \ref{algo1} and obtain $\bar\theta_1$. Later, we choose a set $\mathcal{P}_2 \subset \mathcal{P}_1$ such that $\bar\theta_1\in\mathcal{P}_2$. Using $\mathcal{P}_2$ and a finer refinement, the Algorithm provides $\bar\theta_2$. We iterate this procedure, and finally we set $\mathcal{P} = \mathcal{P}_n$. 
% This approach is rather intuitive since the residual is a function of one variable, and we usually select the set $\mathcal{P}$ such that it is convex to ensure the existence of a global minimizer.

In our test we use the Wendland RBF defined in \eqref{eq:wendland_rbf}. The numerical simulations reported in this paper are performed on a laptop with one CPU Intel Core i7-2.2 GHz and 16GB RAM. The codes are written in MATLAB.

\subsection*{Test 1: Eikonal equation in 2D}

We consider a two dimensional minimum time problem in $\Omega = [-1, 1]^2$ with the following dynamics and control space
\begin{equation}  
	f(x,u) = \begin{pmatrix}
		\cos(u) \\
		\sin(u)
	\end{pmatrix},\quad U = [0, 2 \pi]. 
\end{equation}

  The cost functional to be minimized is $\mathcal{J}_x(y, u) = \int_0^{t(x,u)} e^{-\lambda s}ds$ with \begin{equation}  
	t(x,u) \coloneqq
	\begin{cases}
		\inf_{s} \{s \in \mathbb{R}_+ : y_{x}(s,u) \in \mathcal{T}\}  & \text{ if } y_{x}(s,u) \in \mathcal{T} \text{ for some } s\\
		+ \infty & \text{ otherwise},
	\end{cases}
	\label{time_arrival}
\end{equation} being the time of arrival to the target $\mathcal{T}=(0,0)$ for each $x \in [-1, 1]^2$.

The analytical solution $V^{*}(x)$ for this problem is the Kruzkov transform of the distance to the target $\mathcal{T}$ (see e.g. \cite{BCD97}): \begin{equation*}
    V^{*}(x) \coloneqq 
    \begin{cases}
        1 & \text{ if } v^{*}(x)=\infty \\
        1 - \exp(- v^{*}(x)) & \text{ else} 
    \end{cases}
\end{equation*}
with $v^{*}(x) = ||x||_{2}$ for each $x \in [-1, 1]^2$.

We tested Algorithm \ref{algo1} for two different cases. The first one considers  an unstructured grid generated by random points, whereas the second case studies an unstructured grid generated by the problem dynamics as mentioned in Section \ref{Section4}. Due to the randomness of our grid we compute an average error. In the second case we will average over $10$ tests whereas on the third one over $5$. We do not discuss in detail the case with a regular grid since it has been extensively analyzed in \cite{JS15}. However, in Example \ref{case4}, we show optimal trajectories using also the regular grid with linear interpolation and the Shepard's approximant.
The \textit{relative error} reads:
\begin{equation}
\label{rel_error}
    \mathcal{E}(V_{\theta})=\frac{|| V_{\theta} - V^{*}||_{\infty}}{|| V^{*}||_{\infty}}
\end{equation}
where $V_{\theta}$ is the discrete value function obtained with the shape parameter $\sigma = \frac{\theta}{h}$ and $V^{*}$ is the exact solution.
%\todo[inline]{AA: Hugo please check next sentence. Something is not clear. What do you want to say?}
We will denote by $\theta^{*}$ the parameter selected in order to minimize the \textit{relative error} from \eqref{rel_error}: 
$$\theta^{*}\coloneqq \underset{\theta \in \mathcal{P}}{\arg \min } \text{ } \mathcal{E}(V_{\theta})$$  

Clearly, $\mathcal{E}(V_\theta^{*})$ will be a lower bound for $\mathcal{E}(V_\theta)$.
% Together with $\theta^{*}$ we also show the quantity \textit{minimum error} computed as  $\mathcal{E}(V_{\theta^{*}})$. As the result of a minimization process, $\mathcal{E}(V_{\theta^{*}})$ always assumes values less than $\mathcal{E}(V_{\theta})$.
In the following three cases we set $\lambda=1$ and $U$ is discretized with $16$ equidistant controls.

\begin{example}{ Random Unstructured Grid} \label{case2}

{\rm The first test with Eikonal equation is performed using an unstructured grid generated by random points. In order to obtain a grid which densely covers our numerical domain, a set of 40000 randomly distributed points in $[-1,1]^2$ is clustered using the \textit{k-means} algorithm, where $k$ is the number of desired points in the grid. Examples of this type of grid are shown in Figure \ref{unstructured_case2}.

\begin{figure}[htbp]
\label{unstructured_case2}
	\centering 
	\includegraphics[scale = 0.27]{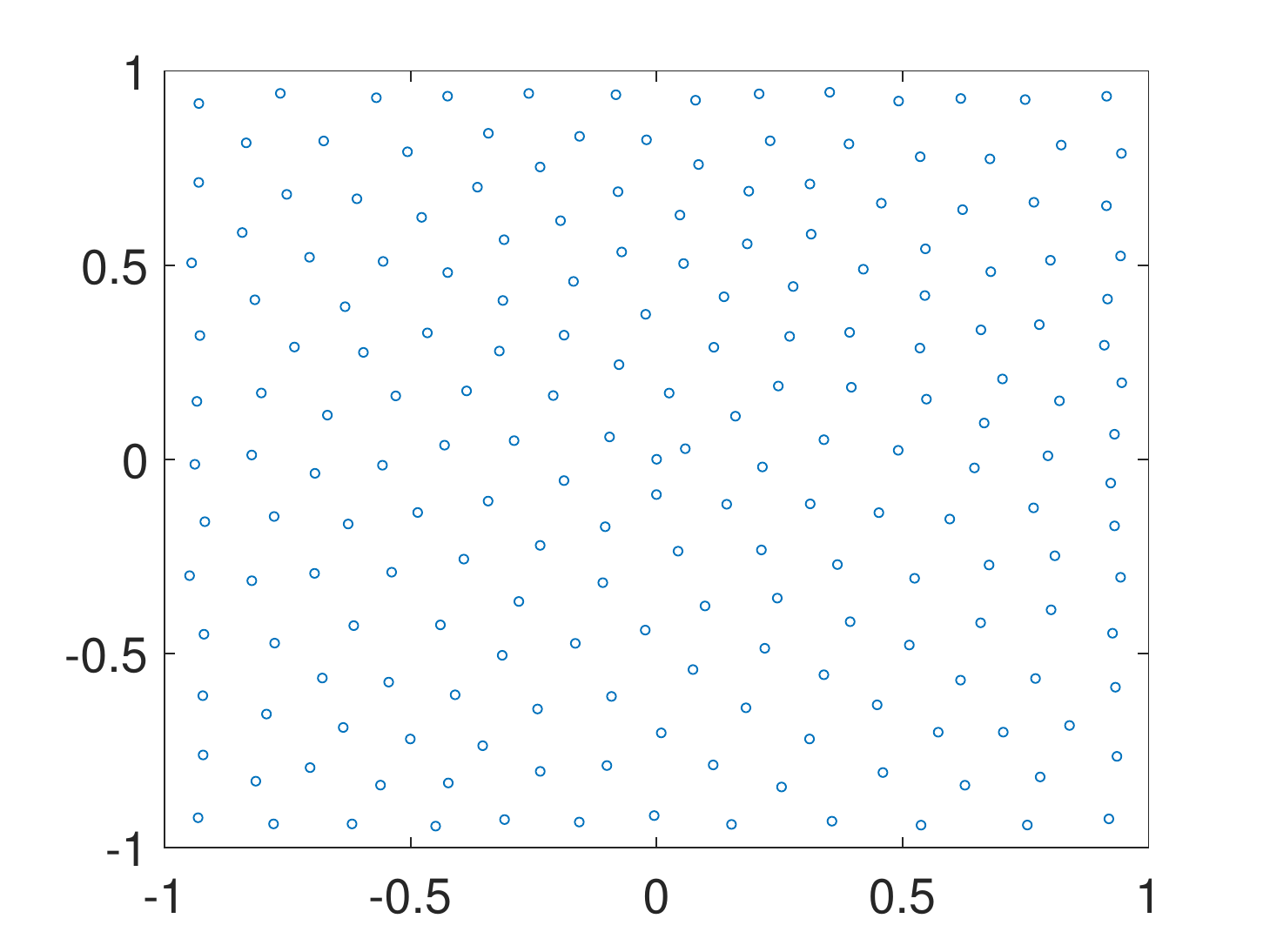}
	\includegraphics[scale = 0.27]{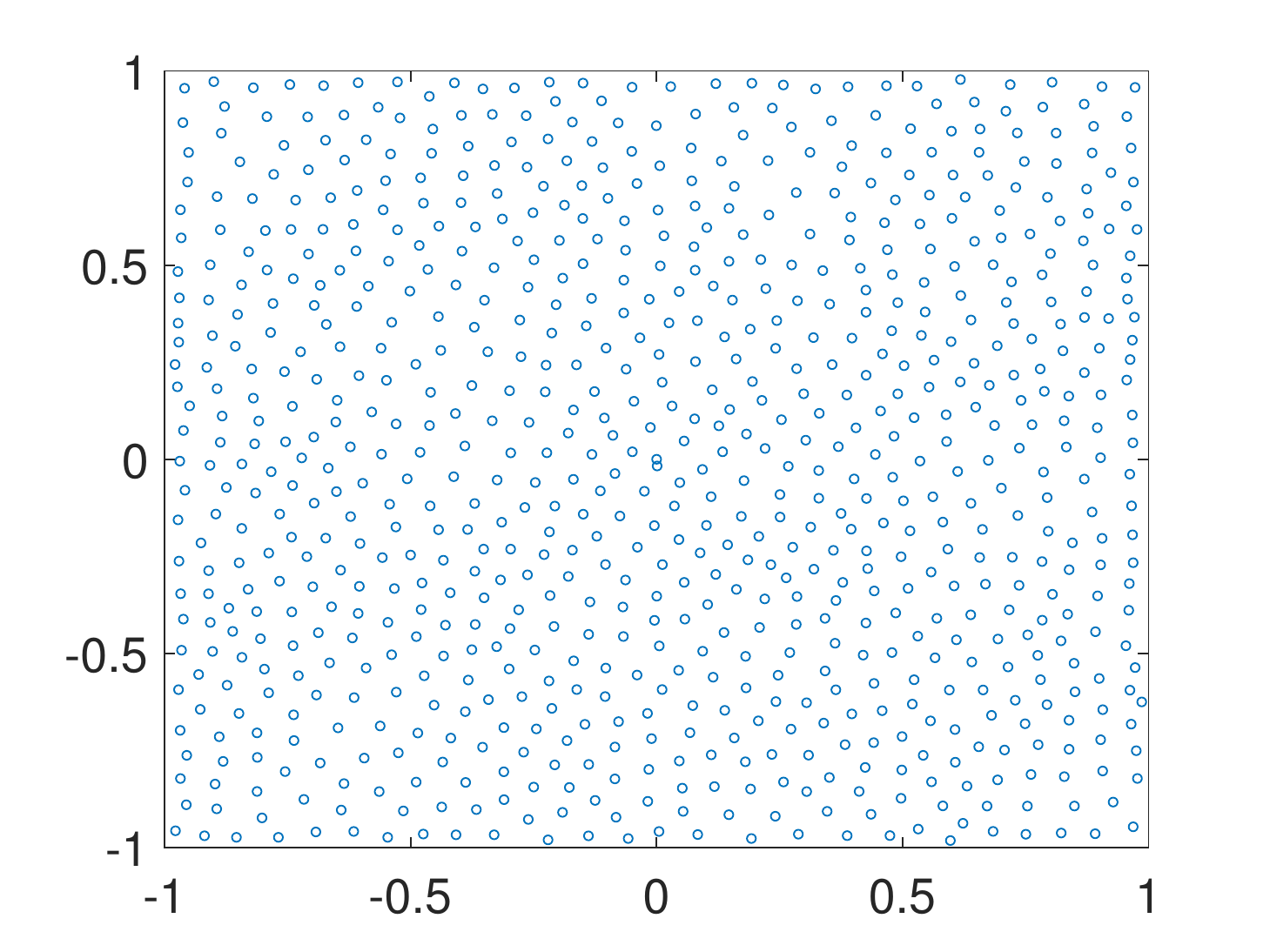}
	\includegraphics[scale = 0.27]{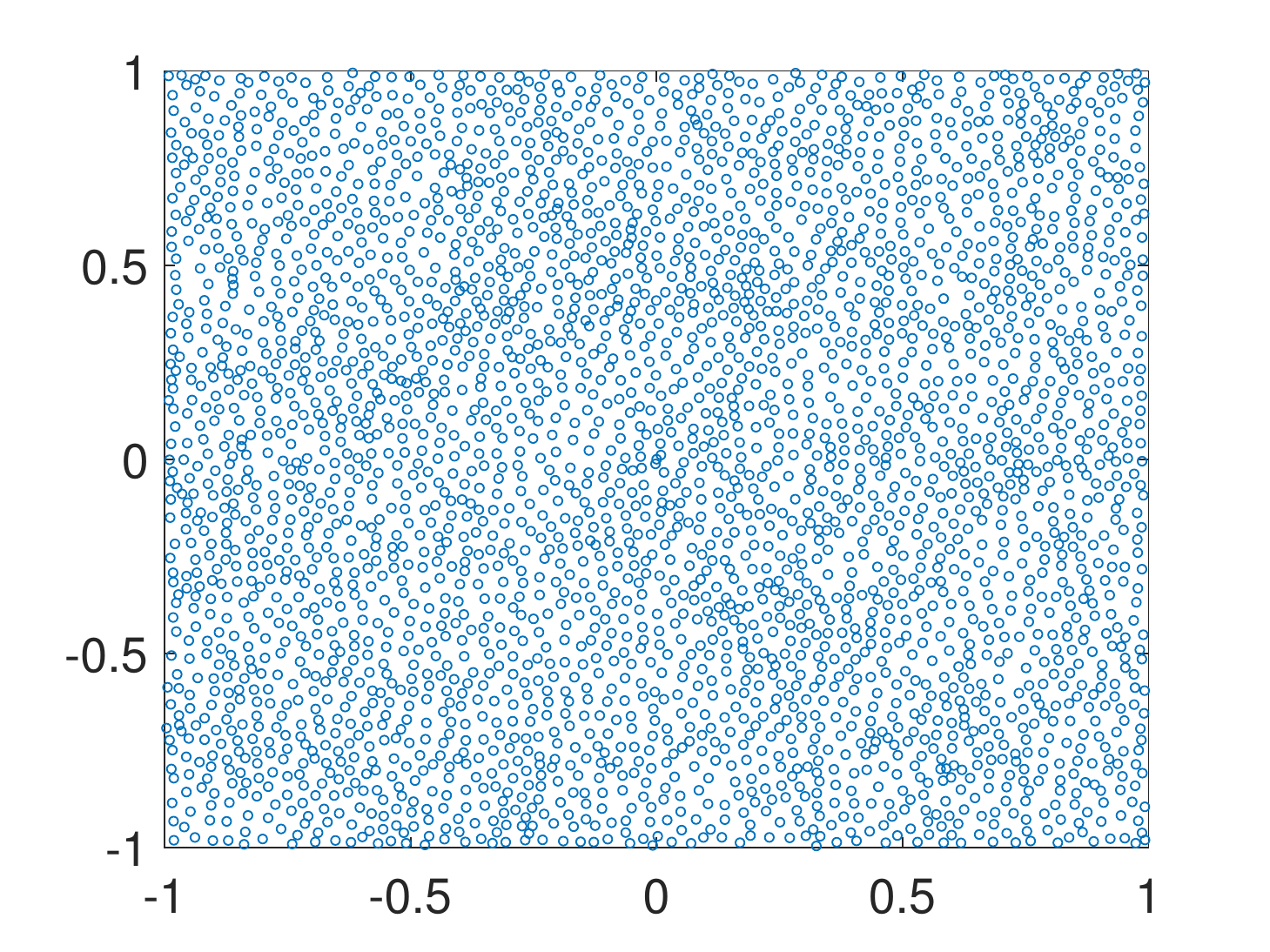}
	\caption{Example \ref{case2}. Examples of random generated grids. Left: $200$ points and fill distance $0.1618$.  Center: $800$ points and fill distance $0.0846$. Right: $3200$ points and fill distance $0.0461$.}
\end{figure}

We found $\mathcal{P} = [1,3]$ a suitable parameter space discretized with $0.1$ as step size. In the HJB equation we set $\Delta t = h$. 
% Figure \ref{unstruc1} presents similar pictures of Case 1, but now in the context of unstructured grid. 
%\todo[inline]{See comments/corrections Case 1 and adjust.}
In the left panel of Figure \ref{unstruc1} we see an example of residual $R(\theta)$ when the unstructured grid is formed by $3200$ nodes. The residual is minimized between $1.5$ and $2$. The average value after $10$ tests is $\bar{\theta}=1.76$ as can be seen in the fourth column of Table \ref{table2_eikonal}. In the middle panel of Figure \ref{unstruc1} we see a plot of $\mathcal{E}(V_{\theta})$ for different number of points in the grid; the values $\theta^{*}$ are in the fifth column of Table \ref{table2_eikonal}. The right panel of Figure $\ref{unstruc1}$ shows the behavior of $\mathcal{E}(V_{\bar{\theta}})$ and $\mathcal{E}(V_{\theta^{*}})$ decreasing the fill distance $h$. The error decays as $h$ does.
% \todo{AA: comments on the theorem?}

\begin{figure}[htbp]
\label{unstruc1}
	\centering 
	\includegraphics[scale = 0.27]{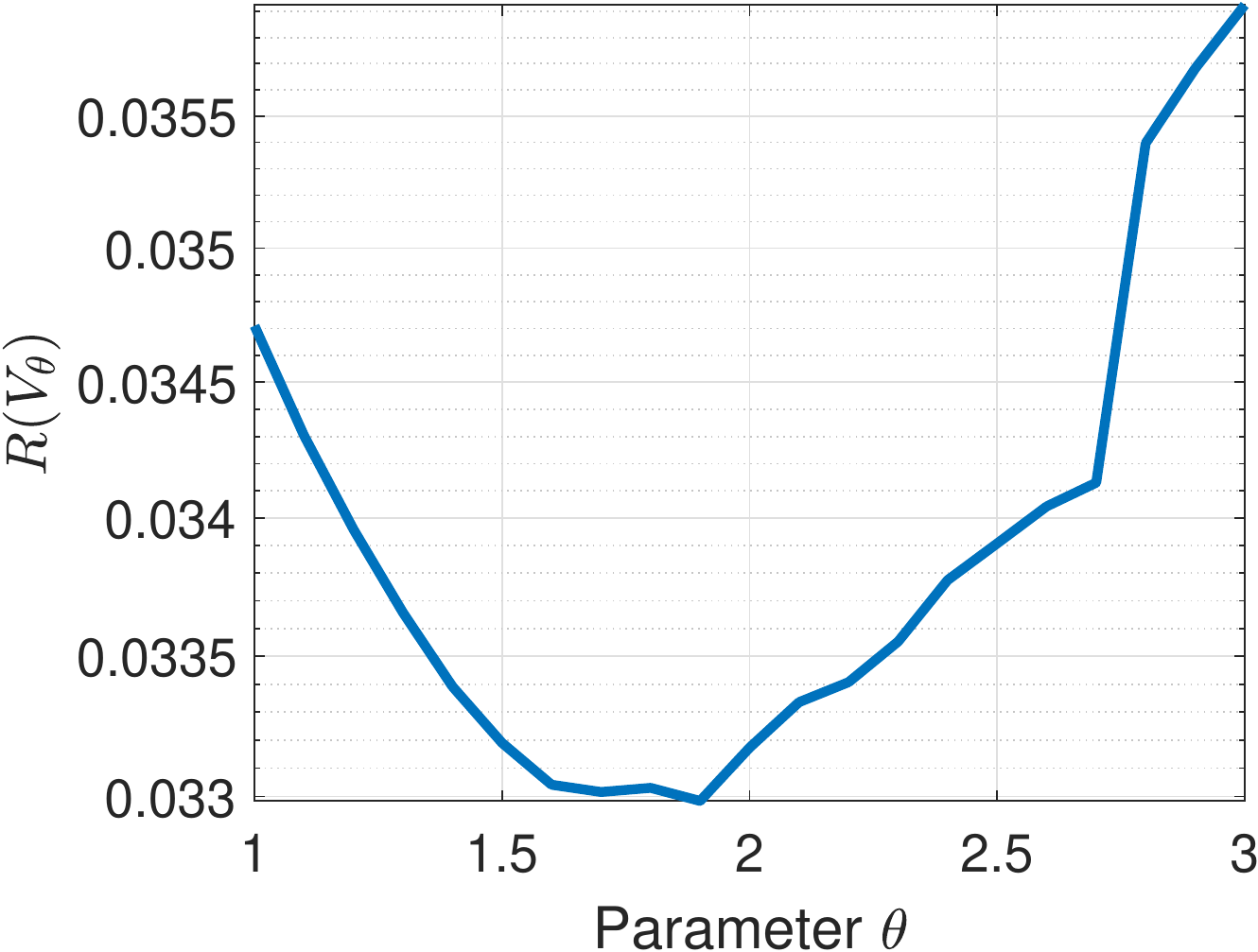}
	\includegraphics[scale = 0.27]{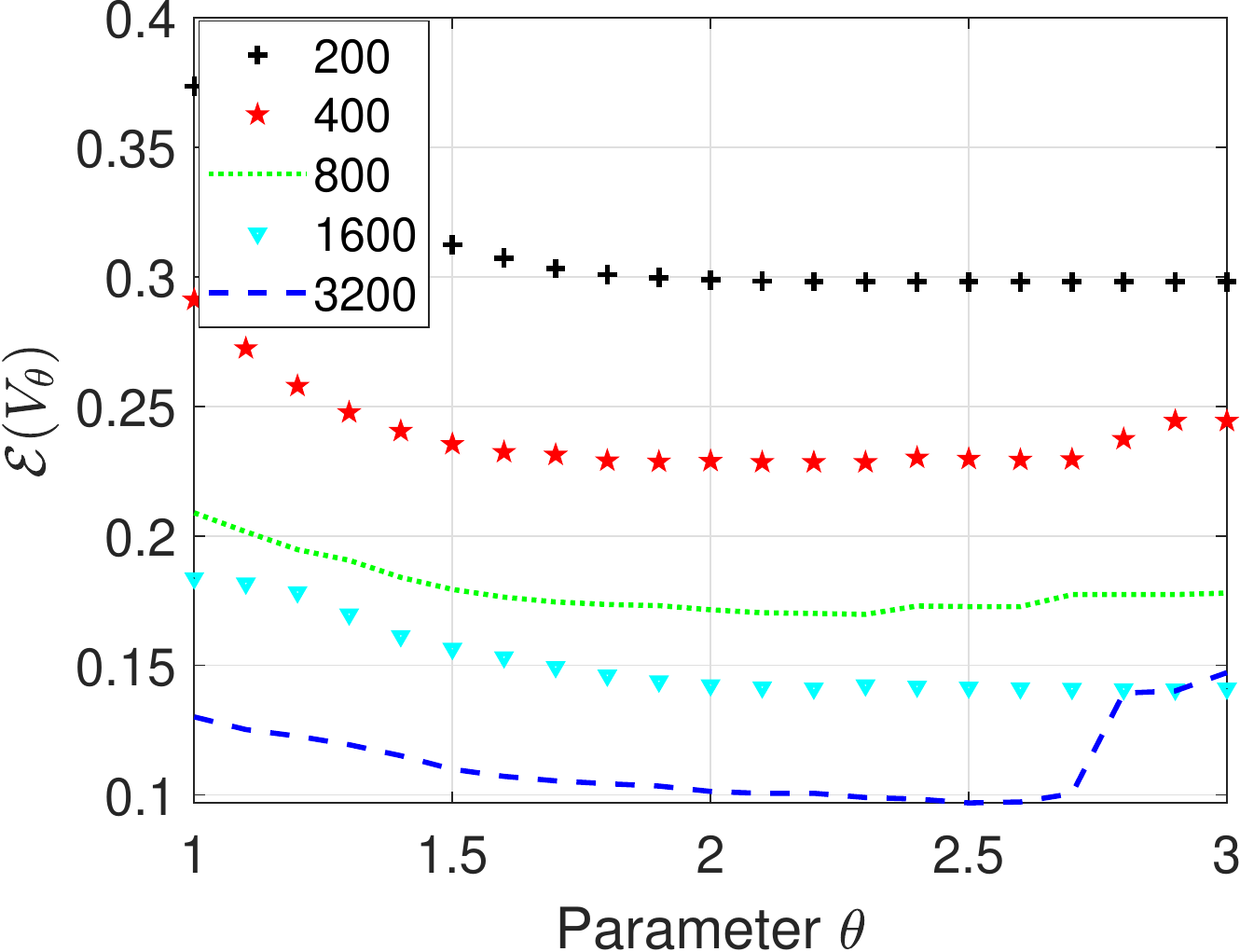}
	\includegraphics[scale = 0.27]{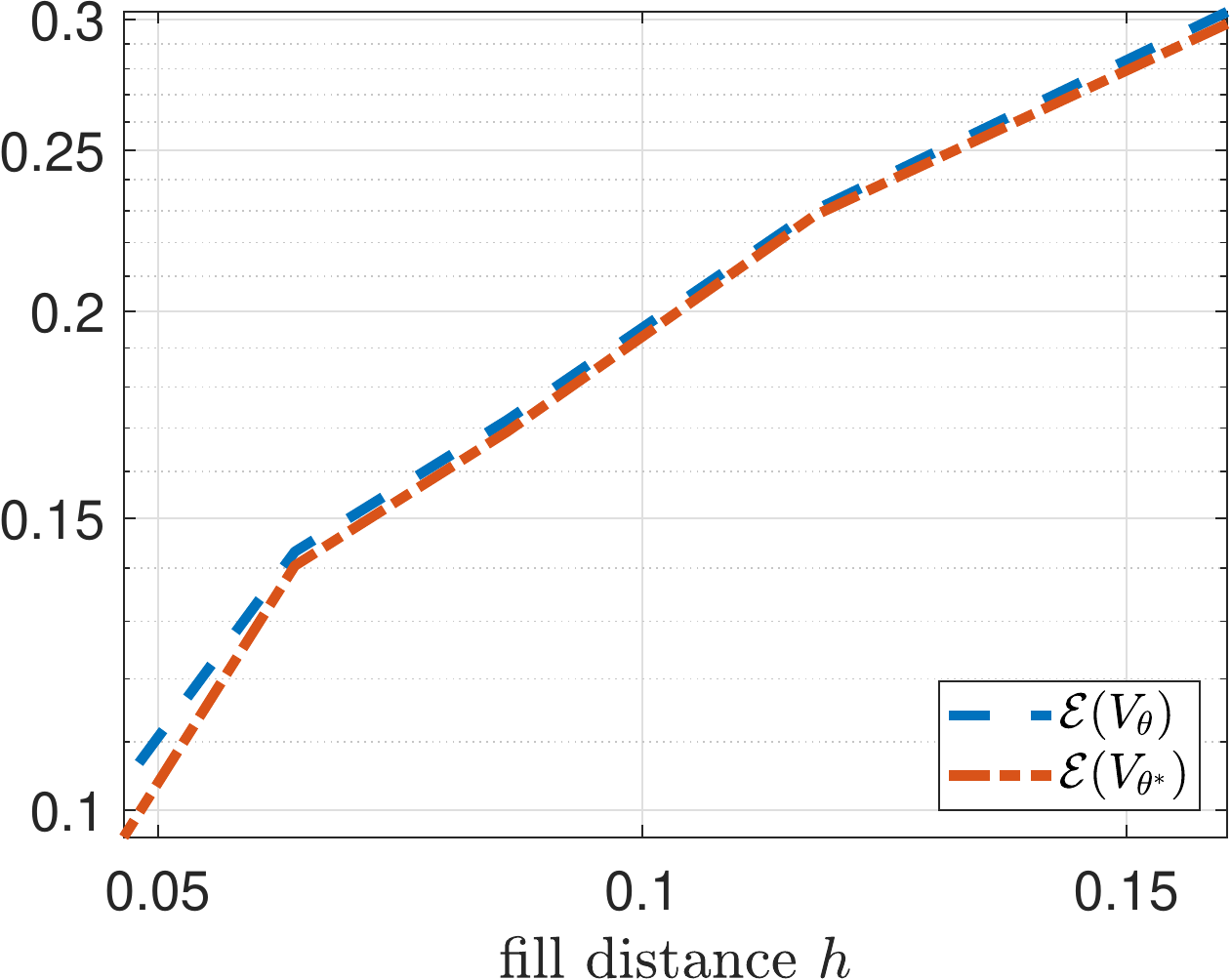}
	\caption{Example \ref{case2}. Left: Average residual for $3200$ points. Middle: $\mathcal{E}(V_{\theta})$. Right: $\mathcal{E}(V_{\bar{\theta}})$ and $\mathcal{E}(V_{\theta^{*}})$ variation with $h$.}
\end{figure}

Table \ref{table2_eikonal} shows the quality of our results. The first column presents the fill distance $h$ and the relative number of points is shown in the second column. The third column presents the CPU time (in seconds). The fourth column presents the values $\bar{\theta}$, outputs of Algorithm \ref{algo1} using the comparison method in the minimization procedure. The fifth column presents the values of $\theta^{*}$. The sixth and seventh columns present the values of $\mathcal{E}(V_{\bar{\theta}})$ and $\mathcal{E}(V_{\theta^{*}})$. 
% The last two columns of Table \ref{table2_eikonal} present the values of $\frac{h}{\bar{\theta}}$ and $\frac{h}{\theta^{*}}$. The observed reduction in $\frac{h}{\bar{\theta}}$ and $\frac{h}{\theta^{*}}$ together with a reduction in values of $\mathcal{E}(V_{\bar{\theta}})$ and $\mathcal{E}(V_{\theta^{*}})$ agrees the results of Theorem \ref{theo:us}

% Table \ref{table2_eikonal} has the same structure of Table \ref{table1_eikonal}.
We see the average decay of the fill distance $h$ when increasing the number of nodes. Accordingly, the average CPU time increases. The parameters $\bar{\theta}$ and $\theta^{*}$ assume values close to each other,  with $\theta^{*}>\bar{\theta}$. The errors $\mathcal{E}(V_{\bar{\theta}})$ and $\mathcal{E}(V_{\theta^{*}})$ reduce according to $h$.

% in fact as shown in the last two columns of the Table \ref{table2_eikonal} the values of $\frac{h}{\bar{\theta}}$ and $\frac{h}{\theta^{*}}$ decay as suggested in Theorem \ref{theo:us} \ale{?? is that true?}.

% \todo[inline]{AA: in the estimate we do not have $h/\theta$ anymore, shall we keep the ratio in the table?}
 \begin{table}[htbp]
 \centering
 \label{table2_eikonal}
 	\begin{tabular}{ccccccc}
 		\hline
 		$h$	& Points & CPU time & $\bar{\theta}$ & $\theta^{*}$ &$\mathcal{E}(V_{\bar{\theta}})$ & $\mathcal{E}(V_{\theta^{*}})$ \\
%  		&   $\frac{h}{\bar{\theta}}$ & $\frac{h}{\theta^{*}}$\\
\hline
 		$0.1603$	& $200$ & $9.8$ &$1.91$ & $2.16$ &$0.3031 $ & $0.2981$ \\
%  		&   $0.0839$ & $0.0742$\\
 		$0.1177$	& $400$ & $14.6$ & $1.86$ & $2.06$ &$0.23$ & $0.2284$ \\
%  		&   $0.0633 $ & $0.0572$\\
 		$0.0861$  & $800$ & $31.8$ &$1.92$ & $2.21$ &$0.172$ & $0.1697$ \\
%  		&   &$0.0448$ & $0.0389$ \\
 		$0.0641$ & $1600$ & $115$ & $2.04$ & $2.42$ &$0.1432 $ & $0.1407$ \\
%  		&  $0.0314  $ & $0.0265$\\
 		$0.0464$ & $3200$ & $504$ &  $1.76$ & $2.06$ &$0.1037$ & $0.0969$ \\
%  		&  $0.0264$ & $0.0225$\\ 
\hline
 	\end{tabular}  \\
 	\caption{Example \ref{case2}. Numerical Results with random unstructured grid.}
 \end{table}

Figure \ref{scattered_value} presents in the left panel the exact solution evaluated on the scattered grid with $3200$ points and $h=0.0464$. The middle picture is the solution obtained by value iteration algorithm with Shepard approximation and the last picture is the absolute error between the two solutions. The error has an erratic behavior always below $10^{-1}$.

\begin{figure}[htbp]
\label{scattered_value}
	\centering 
	\includegraphics[scale = 0.27]{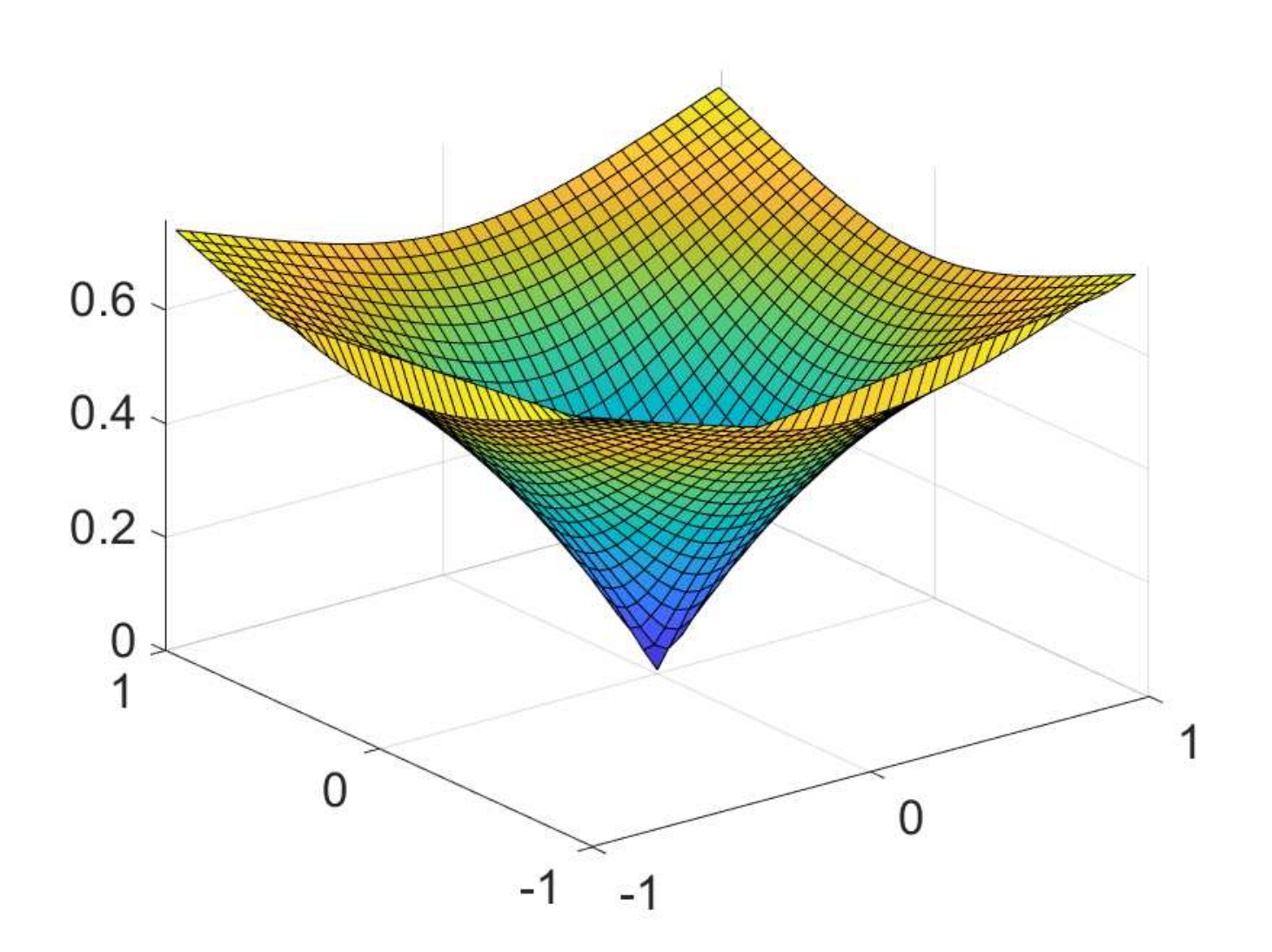}
	\includegraphics[scale = 0.27]{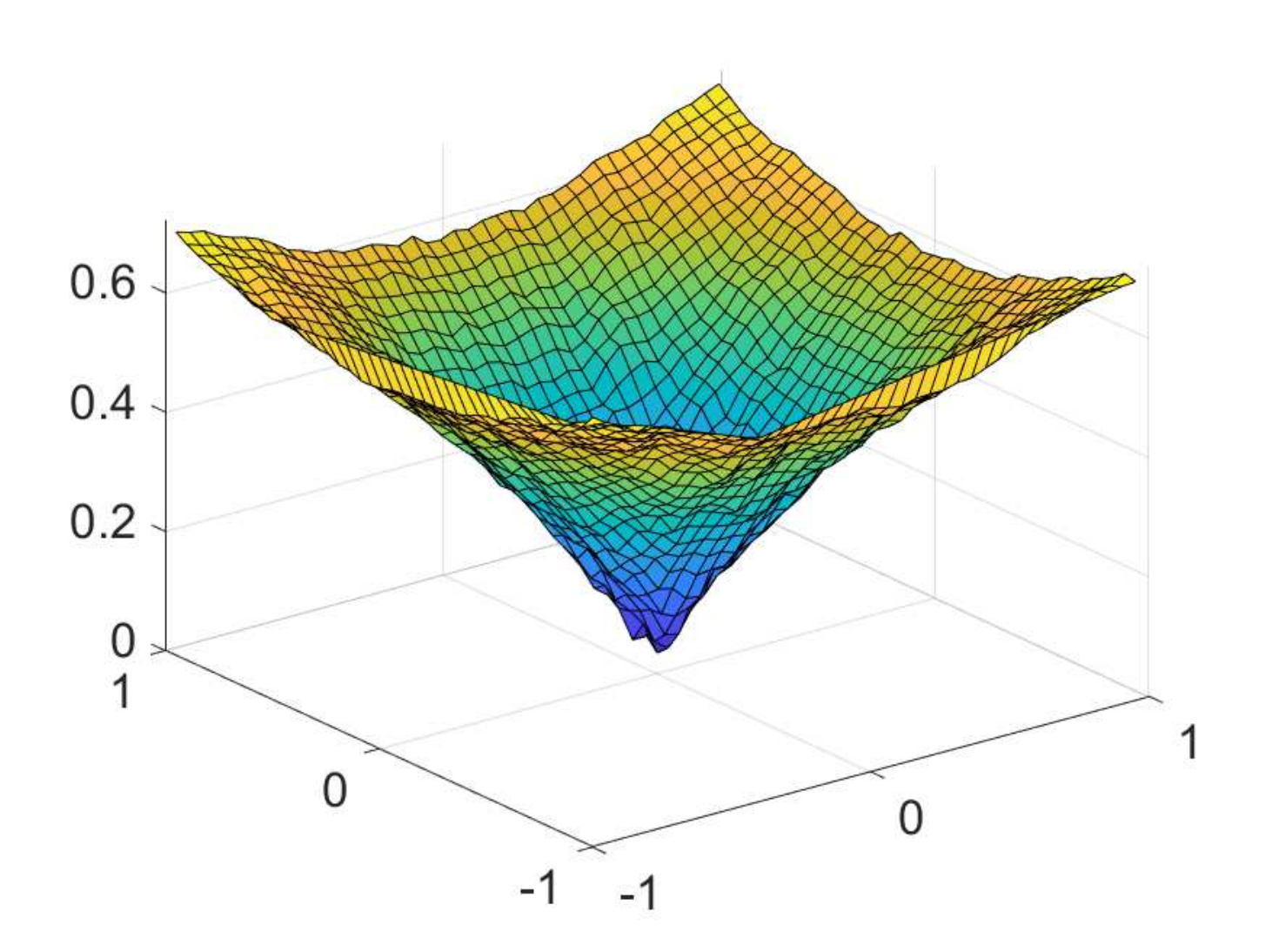}
	\includegraphics[scale = 0.27]{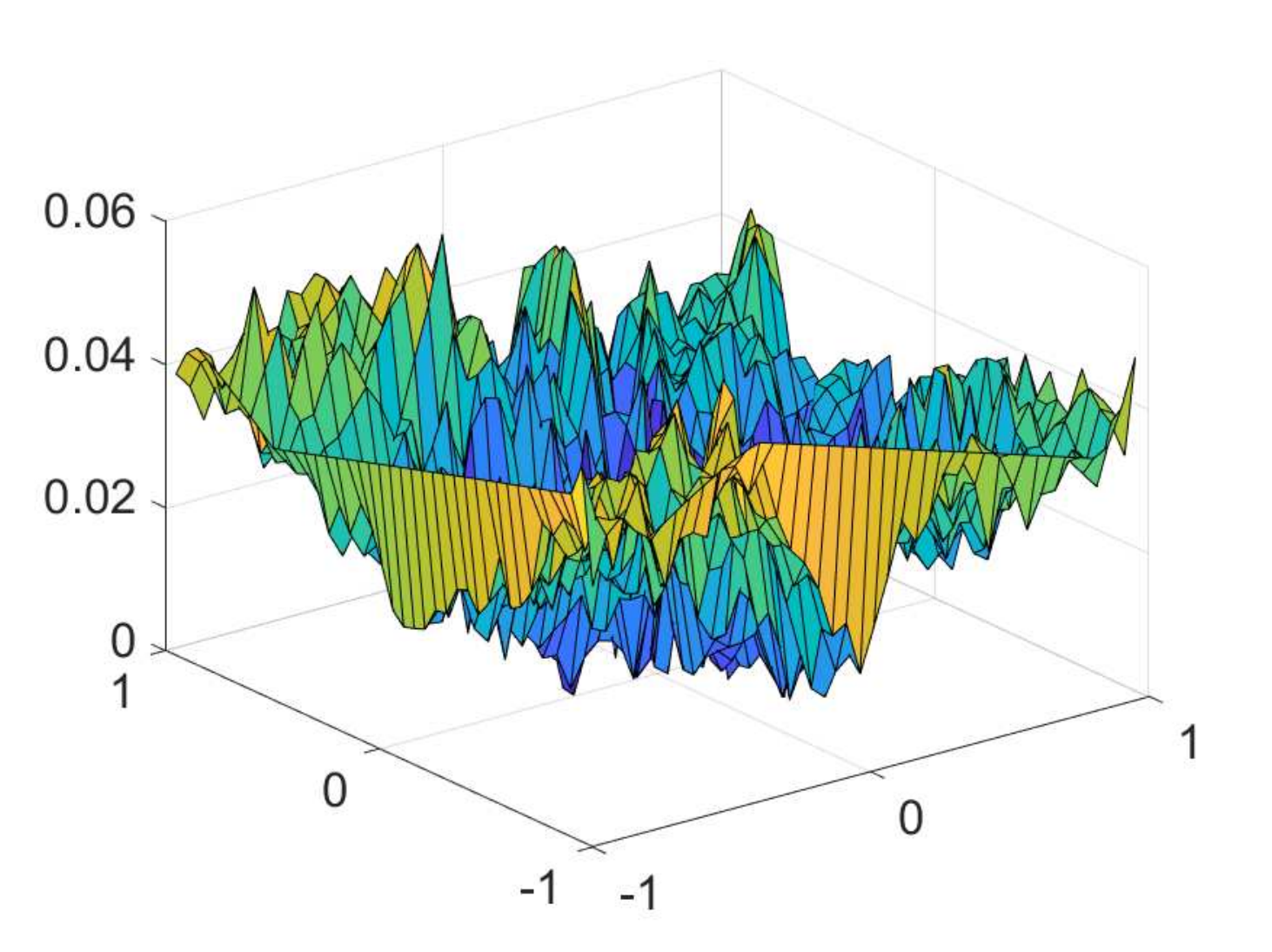}
	\caption{Example \ref{case2}. Value functions generated in a Random  Unstructured Grid formed by $3200$ points. Left: Exact solution. Center: Solution 
obtained by VI and Shepard approximation. Right: Absolute error of exact solution and value function obtained by Shepard approximation Value Iteration.}
\end{figure}

% \begin{table}[htbp]
% \label{gradient_unstructured}
%     \centering
%     \begin{tabular}{cccc}
%     \hline
%          Points & CPU time & $\theta$ & $\mathcal{E}(V_{\bar{\theta}})$\\ \hline
%         $200$ & $0.94$& $1.38$ & $0.3516$\\
%         $400$ & $9.49$& $1.43$ & $0.2391$\\
%         $800$ & $44.02$ & $1.11$ & $0.2021$\\
%         $1600$ & $5.34$e$+3$ & $2.08$ & $0.1628$ \\ \hline
%     \end{tabular}
%     \caption{\ale{Test 1. Case 2. Results using gradient method.}}
% \end{table}}
}
\end{example}

\begin{example}{Grid driven by the dynamics}\label{case3}

\noindent
{\rm In this example we test our novel grid proposed in Section \ref{Section4}. We set $\Delta t = h$ in the HJB equation. To generate the trajectories which will be our grid points, we set $(\overline \Delta t, \bar L, \bar M)=\{(0.1,4,16),(0.05,8,16),(0.025,16,16) \}$ in \eqref{eq:mesh}. Figure \ref{unstructured_case3} shows some examples of meshes generated in this case, with randomly selected initial conditions. These meshes follow the pattern of problem dynamics.

\begin{figure}[htbp]
\label{unstructured_case3}
	\centering 
	\includegraphics[scale = 0.27]{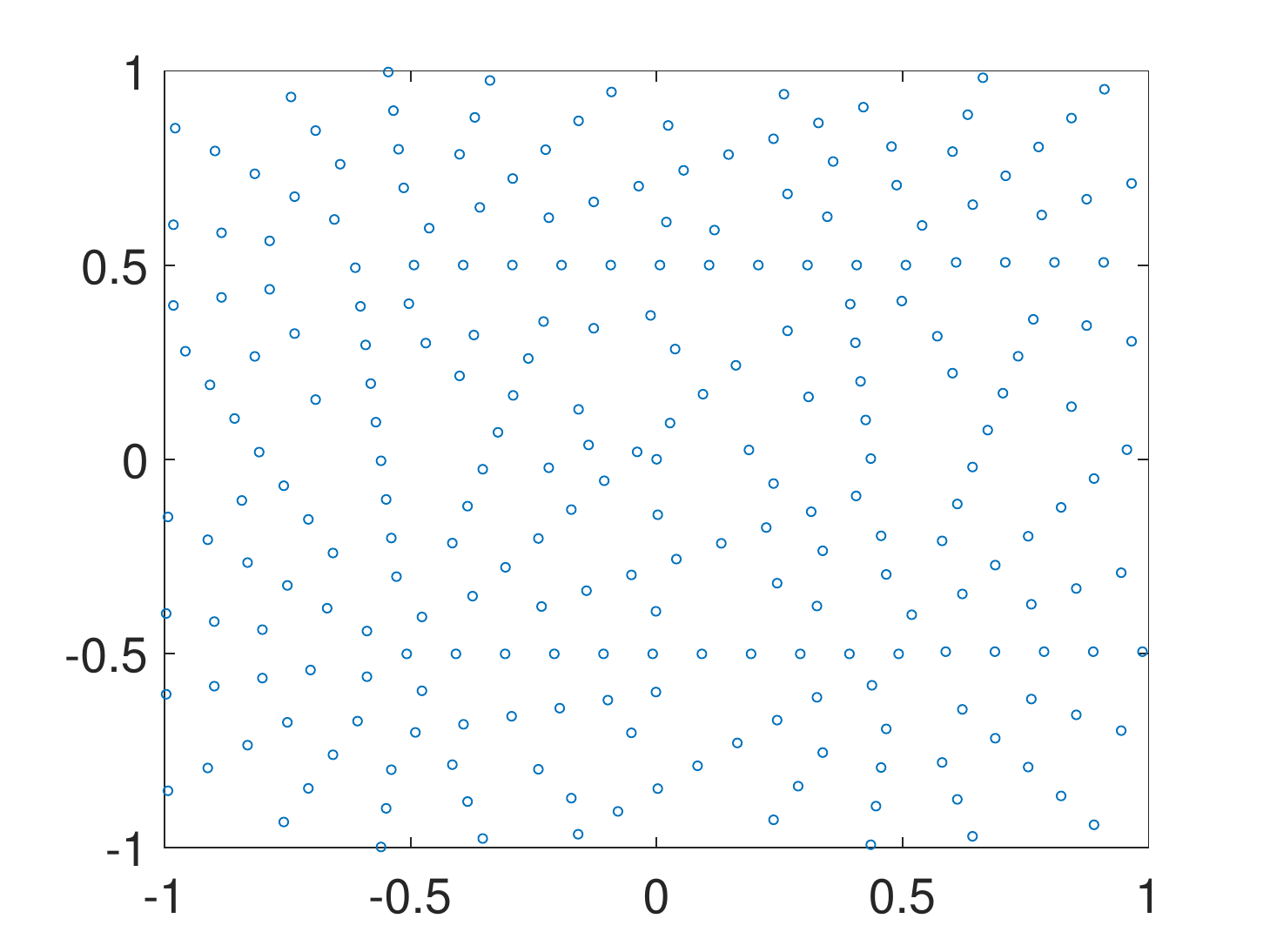}
	\includegraphics[scale = 0.27]{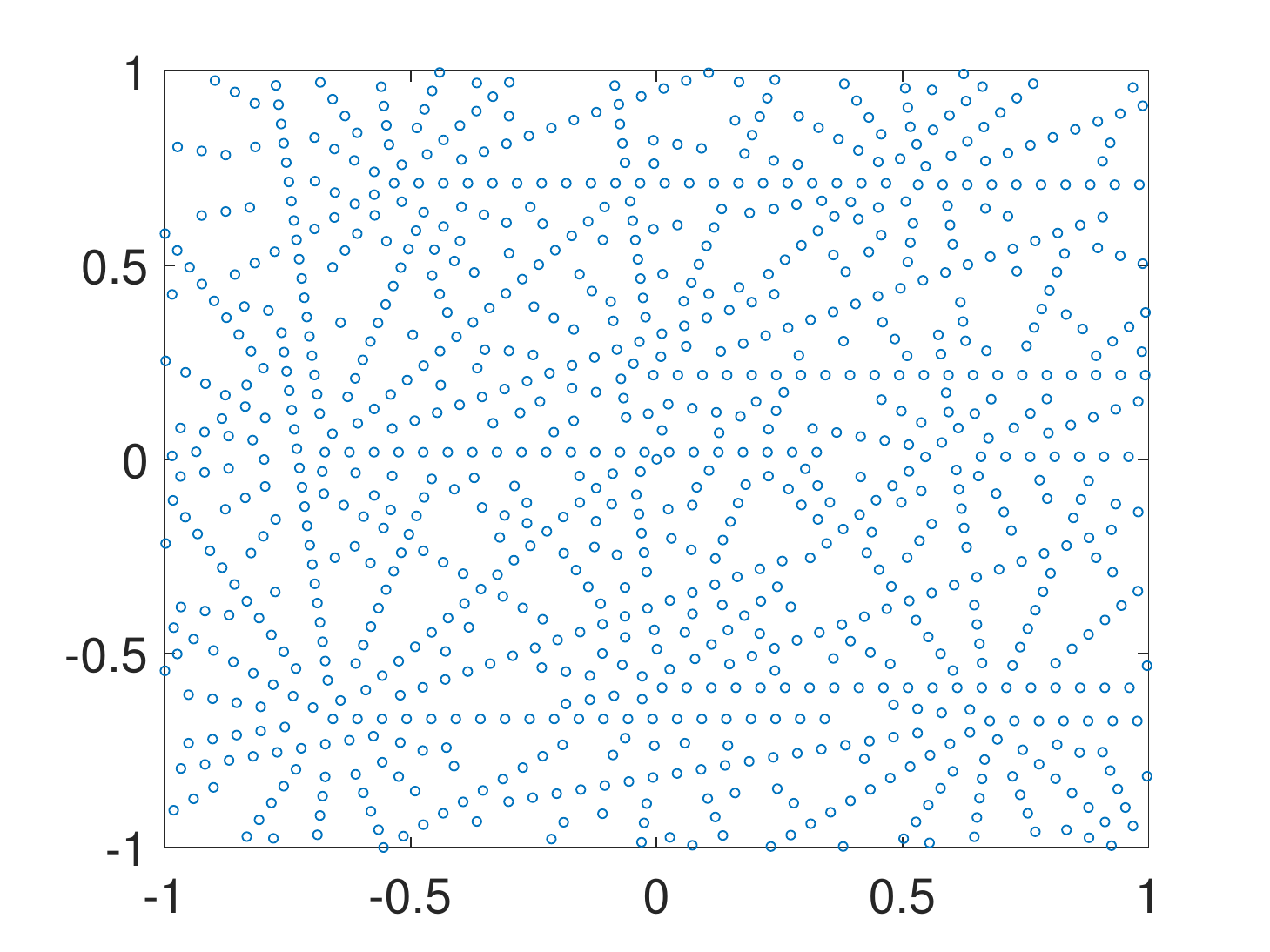}
	\includegraphics[scale = 0.27]{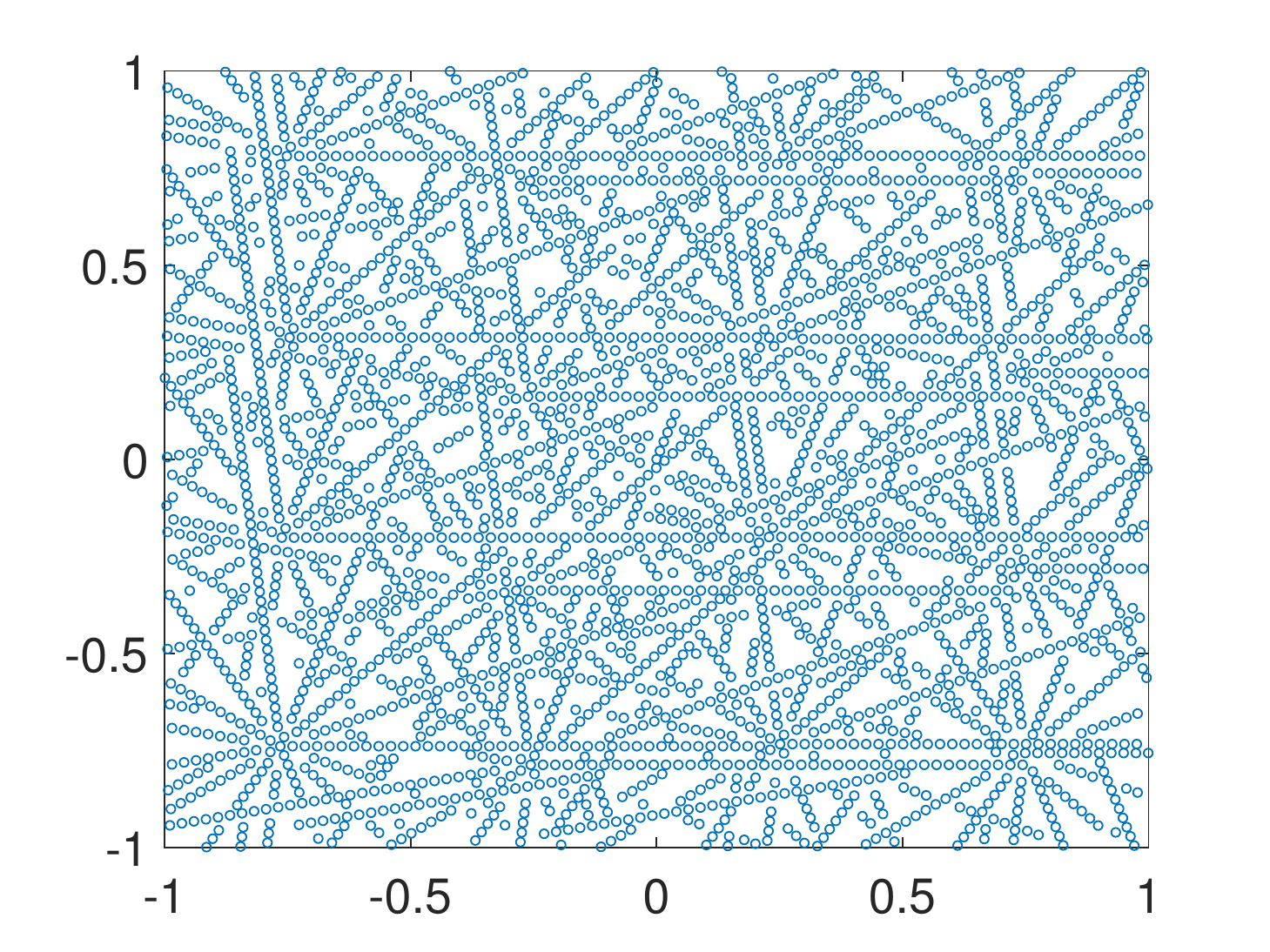}
	\caption{Example \ref{case3}. Meshes generated by the dynamics. Left: $246$ points and fill distance $0.1436$. Middle: $909$ points and fill distance $0.0882$. Right: $3457$ points and fill distance $0.0439$.}
\end{figure}

We set, again, $\mathcal{P}=[1,3]$. The behavior of $R(\theta)$ is shown in Figure \ref{unstruc2} when the grid has $3469$ points on average. The minimum value is achieved for $\bar{\theta} = 1.7$. The middle panel of Figure \ref{unstruc2} shows a plot of $\mathcal{E}(V_{\theta})$ for a different number of points. The right picture shows the error behavior in $\mathcal{E}(V_{\bar{\theta}})$ and $\mathcal{E}(V_{\theta^{*}})$ with the reduction of $h$. It decreases according to Theorem \ref{theo:us}. We stress that all of these quantities are computed averaging $5$ simulations.

\begin{figure}[htbp]
\label{unstruc2}
	\centering 
	\includegraphics[scale = 0.27]{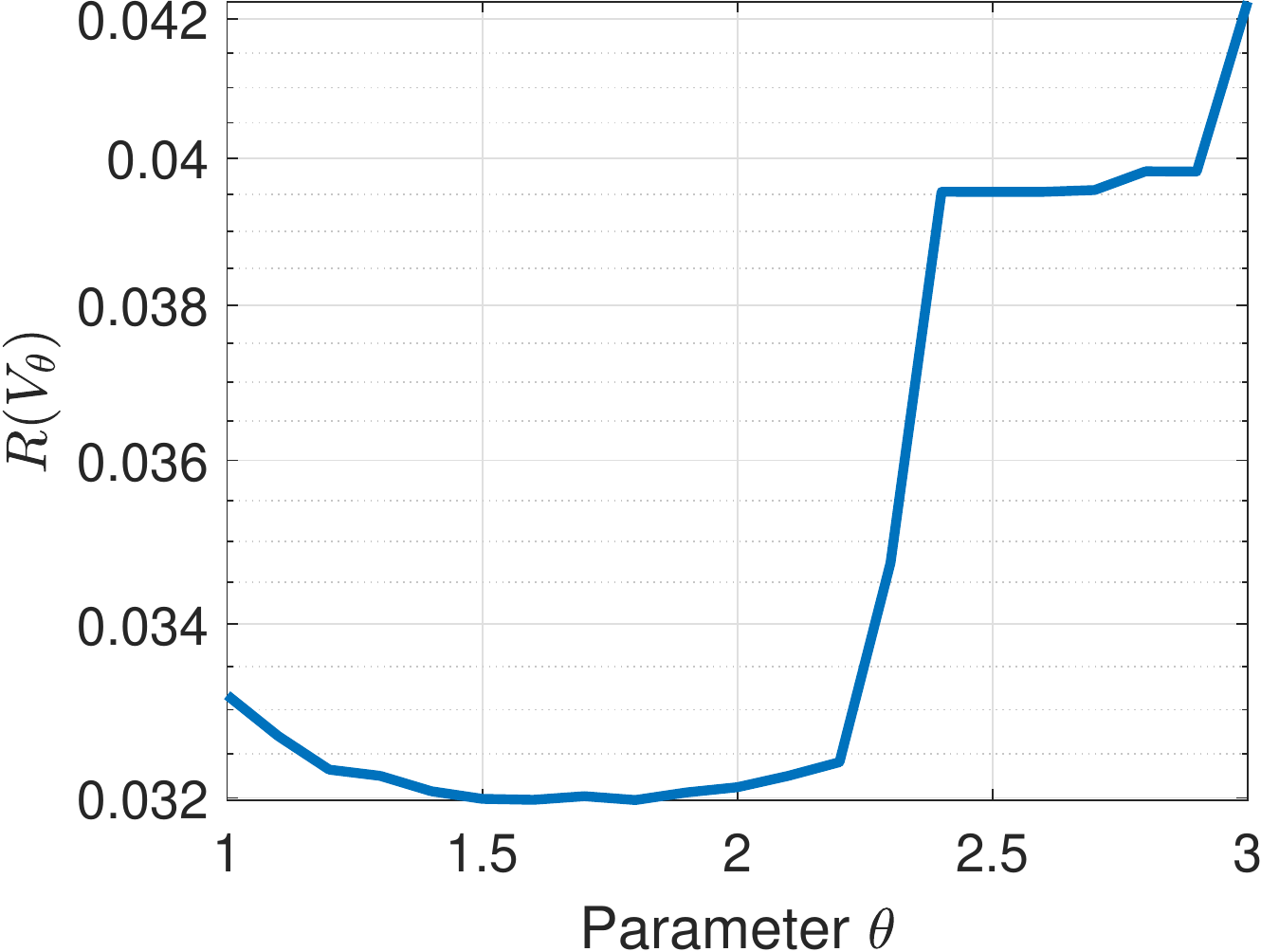}
	\includegraphics[scale = 0.27]{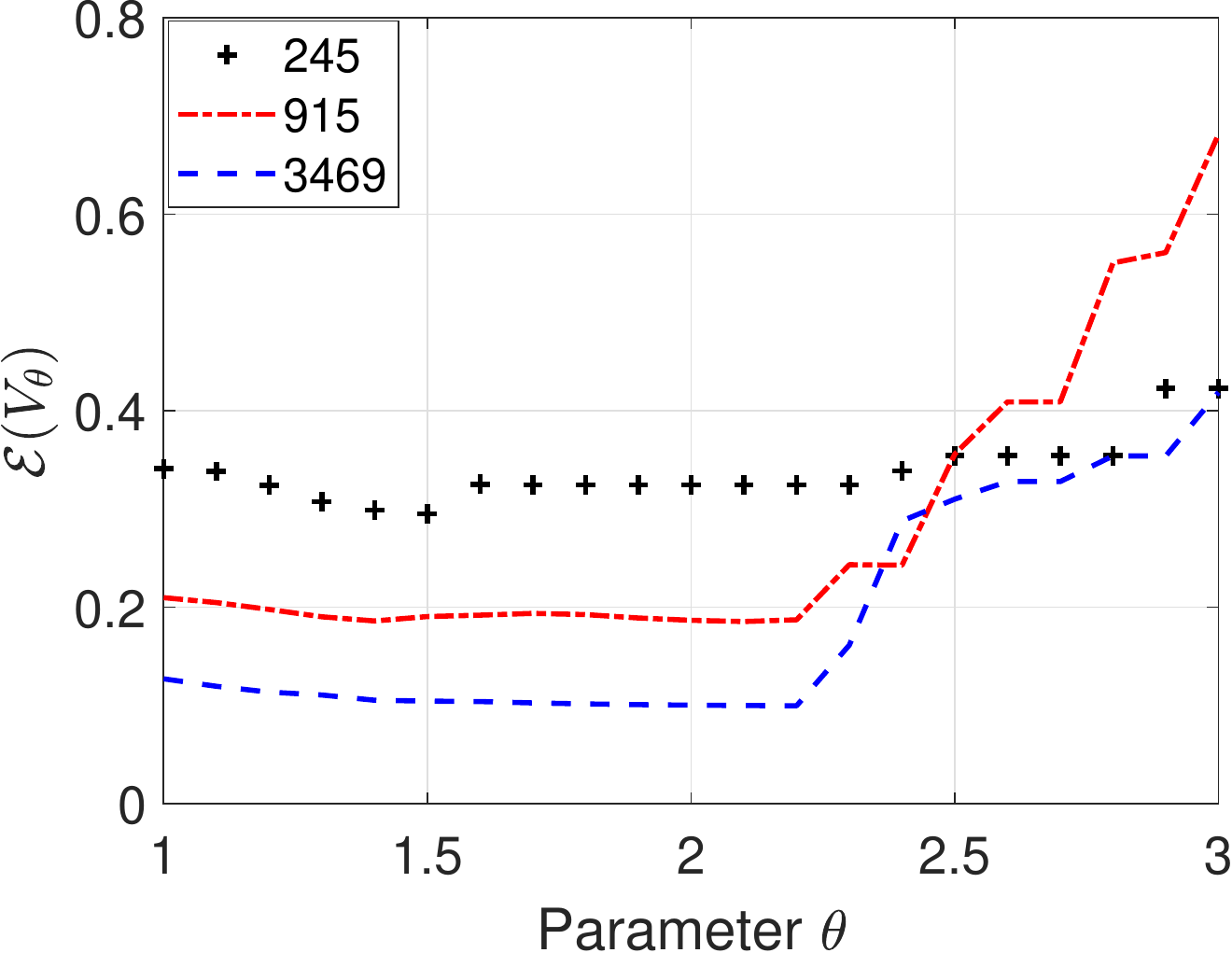}
	\includegraphics[scale = 0.27]{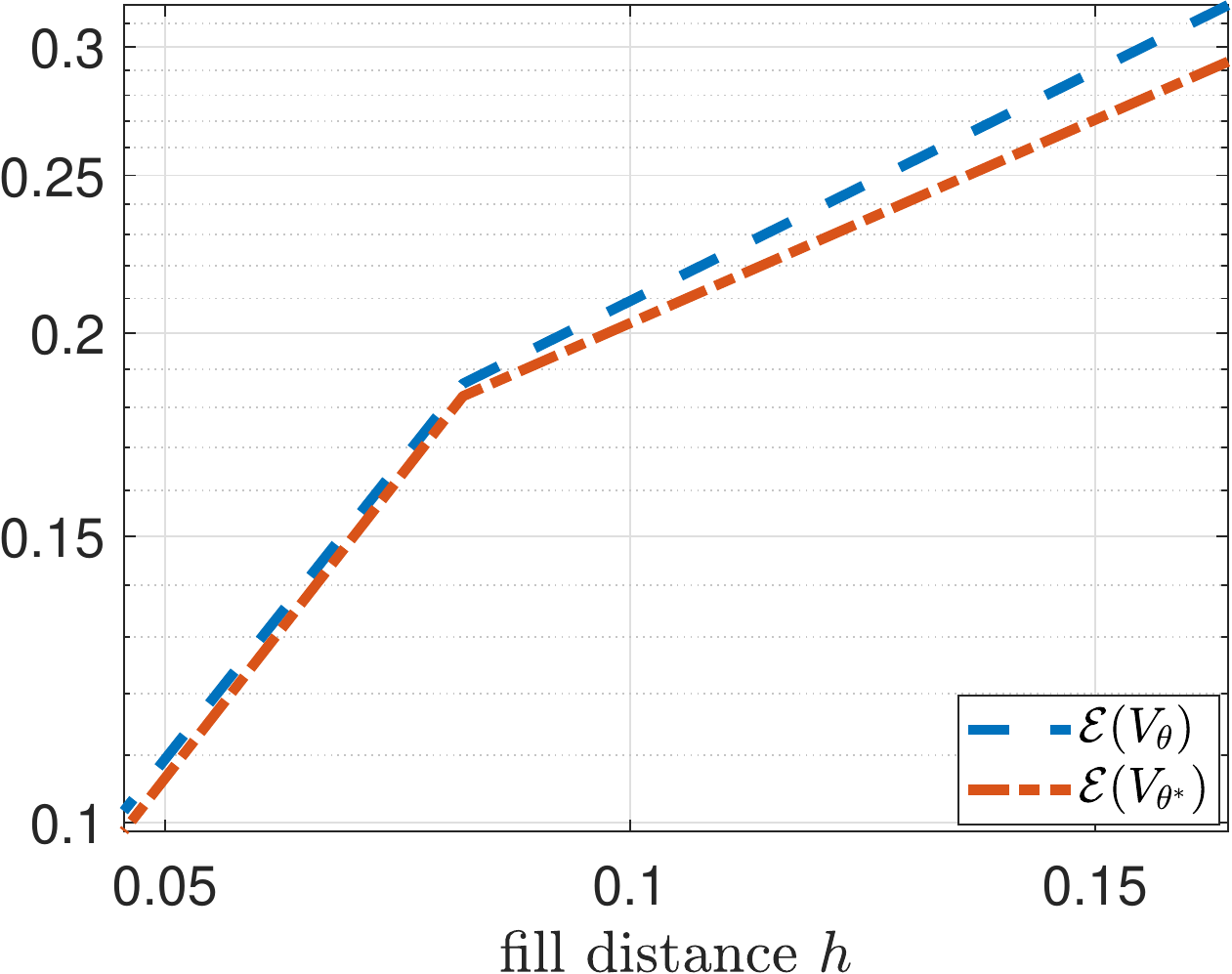}
	\caption{Example \ref{case3}. Left: Average residual to case with $3469$ points. Middle: $\mathcal{E}(V_{\theta})$. Right: $\mathcal{E}(V_{\bar{\theta}})$ and $\mathcal{E}(V_{\theta^{*}})$ variation with $h$.}
\end{figure}

Table \ref{table3_eikonal} summarizes the results of Example \ref{case3} as discussed in Example \ref{case2}. All the considerations are very similar to the previous example but now we have a different grid which will help us to deal with higher dimensional problems as it will be presented in the next sections. We stress that the values of the error indicator in the sixth and seventh column of Table \ref{table3_eikonal} are very close to each other. Once again, this is very interesting since we are using an a-posteriori criteria for the computation of the shape parameter.
% The behavior of variables $h$, number of points and average computational time is similar to previous cases. We also point that values of $\bar{\theta}$ and $\theta^{*}$ are relatively close to each other in each case considered. The reduction in quantities $\frac{h}{\bar{\theta}}$ and $\frac{h}{\theta^{*}}$ and reduction in values of $\mathcal{E}(V_{\bar{\theta}})$ and $\mathcal{E}(V_{\theta^{*}})$ is again in accordance with theory.

%\begin{itemize}
%	\item Meshes generated using the dynamics of the problem. %$16$ controls, $\Delta t = 0.05$ and the  $\Delta t$ used to %select points were respectvely $0.1$, $0.05$ and $0.025$.
%	\item Left: $4$ initial conditions. Center: $8$ initial %conditions. Right: $16$ initial conditions. All selected using %\textbf{k-means} algorithm.
%\end{itemize}

\begin{table}[htbp]
\label{table3_eikonal}
\centering
	\begin{tabular}{lllllllll}
		\hline
		$h$	& Points & CPU time &$\theta$ & $\theta^{*}$   &$\mathcal{E}(V_{\bar{\theta}})$ & $\mathcal{E}(V_{\theta^{*}})$ \\
% 		&  $\frac{h}{\theta}$ & $\frac{h}{\theta^{*}}$\\
\hline
		$0.1642$	& $245$ & $8.5$ & $1.58$ & $1.82$  &$0.3182$ & $0.2949$ \\
% 		&  $0.1040  $ & $0.0902 $\\
		$0.0820$	& $915$ &$55.6$ &  $1.66$ & $1.76$ &$0.1861$ &$0.1855$\\ 
% 		&   $0.0494$ & $0.0466 $\\
		$0.0455$	& $3469$ &  $654$ &  $1.7$ & $1.82$ &$0.1016$ &$0.0997$\\
% 		& $0.0268$ & $0.0250$\\ 
\hline
	\end{tabular}  
	\caption{Example \ref{case3}. Numerical results with a grid driven by the dynamics.}
\end{table}

%\todo[inline]{TO UPDATE}
Finally, Table \ref{gradient_equidistant} presents the results of Algorithm \ref{algo1} using a gradient descent method with $\varepsilon = 10^{-6}$ in step $11$ of Algorithm \ref{algo1}. If we compare the results with Table \ref{table3_eikonal} we can see that this method is computationally slower and the accuracy has the same of order of the comparison method. Thus, we will only show the performance of our method where the minimization is computed by comparison. We do not provide results with smaller $h$ in Table \ref{gradient_equidistant} because it is clear that it will be way slower than the comparison method.

\begin{table}[htbp]
\label{gradient_equidistant}
    \centering
    \begin{tabular}{ccccc}
    \hline
      h&   Points & CPU time & $\bar{\theta}$ & $\mathcal{E}(V_{\bar{\theta}})$\\ \hline
    $0.1364$ &   $248$ & $9.7$& $1.62$ & $0.278$\\
    $0.0865$    & $924$ & $152$& $1.72$ & $0.1859$\\
        %$1681$ & $561$ & $0.53$ & $0.0649$\\
        %$6561$ & $2.17$e$+4$ & $0.518$ & $0.0536$\\ 
        \hline
    \end{tabular}
    \caption{Example \ref{case3}. Results using gradient method.}
\end{table}}
% \end{example}
\end{example}

\begin{example}{\it On the feedback reconstruction.}\label{case4}

{\rm 
\noindent
The approximated value functions computed by means of Algorithm \ref{algo1} with different grids allows us to obtain the optimal trajectories and optimal controls for any initial conditions. In Figure \ref{feedback_eikonal}, we present an example of optimal controls and trajectories computed for $x=(0.7, -0.7)$  using value functions obtained in Example \ref{case2} and Example \ref{case3}. For completeness, we also show the results considering the value function obtained by traditional value iteration algorithm using linear interpolation and Shepard method on a regular grid. The latter comes from \cite{JS15}.

\begin{figure}[htbp]
\label{feedback_eikonal}
	\centering 

	\includegraphics[scale = 0.4]{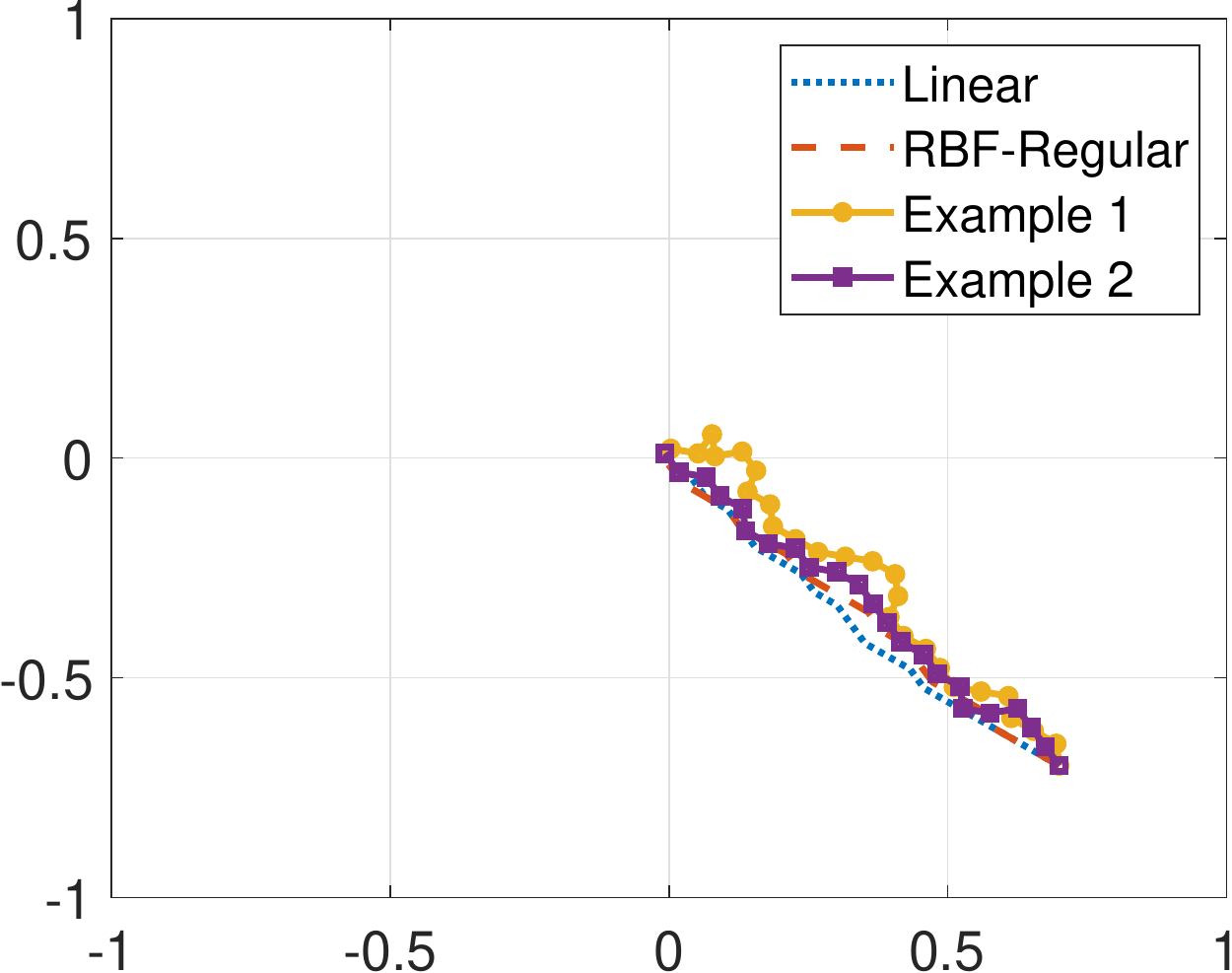}
	\includegraphics[scale = 0.4]{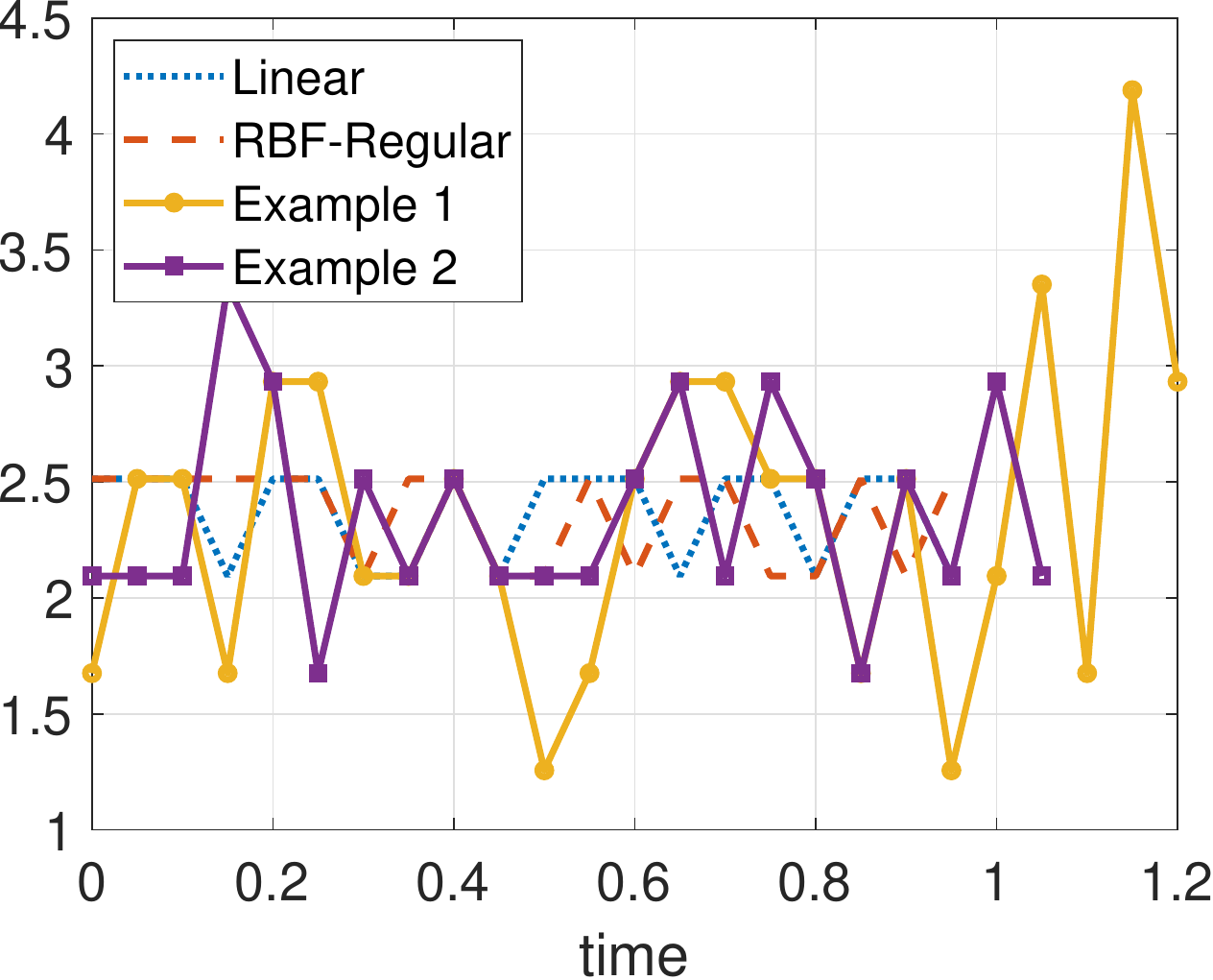}
	\caption{Example \ref{case4}. Optimal trajectories (left) and optimal controls (right) for   $x=(0.7, -0.7)$.}
\end{figure}

All trajectories reach the target $\mathcal{T}$ but with different costs. The value of the cost functional, for a given initial condition, is equal when dealing with a structured grid with linear interpolation and Shepard's approximation. The value of cost functional with value function from Example \ref{case3} is always smaller or equal than the value of cost functional from Example \ref{case2}, for both initial conditions.

In Table \ref{tab:feed} we provide the evaluation of che cost functional for different initial conditions and different methods. It is interesting to see that linear and RBF interpolation coincide on a equidistributed grid and that that the use of RBF with a grid driven by the dynamics lead to lower cost functional values with respect to the a random mesh.
\begin{table}[htbp]
\label{tab:feed}
\centering
	\begin{tabular}{ccccc}
		\hline
$x$	&	Linear	& \textit{RBF-Regular} & \textit{Example 1}  & \textit{Example 2}  \\ \hline
$(-0.7, -0.7)$	&	$0.6664$	& $0.6664$ & $0.7315$ & $0.7006$ \\
$(0.7, 0.7)$ &		$0.6664$	& $0.6664$ & $0.7847$ & $0.6839$  \\
$(-0.7, 0.7)$ &		$0.6664$	& $0.6664$ & $0.7458$ & $0.6839$  \\
$(0.7, -0.7)$ &		$0.6664$	& $0.6664$ & $0.7458$ & $0.7006$  \\
		\hline
	\end{tabular}
	\caption{Example \ref{case4}. Evaluation of the cost functional for different methods and initial conditions $x$.}
\end{table}
}
\end{example}
% \textbf{$x=(-0.7, 0.3)$}

% \begin{figure}[htbp]
% 	\centering 

% 	\caption{Test 1. something}
% \end{figure}

% \begin{table}
% 	\begin{tabular}{lllll}
% 		\hline
% 		Linear	& Regular-Shepard & Scattered-Shepard  & Dynamic-Shepard  &\\ \hline
% 		$0.5646$	& $0.5646$ & $0.6084$ & $0.5870$ & 
% 		\\ \hline
% 	\end{tabular}
% 	\caption{Test 1. Cost functional evaluated in different trajectories for point $(-0.7, 0.3)$.}
% \end{table}

%\subsection{PDE control}

\subsection*{Test 2: Bilinear advection Equation}

	The second test deals with the control of a two-dimensional advection equation with constant velocity $v\in\R$:
\begin{equation}
\label{transport}
 \begin{cases}  \tilde{y}_{t}(\xi,t)  +v\nabla_\xi \tilde{y}(\xi,t) = u(t)\tilde{y}(\xi,t) & (\xi,t) \in \Omega \times [0,T]\\ 			
 \tilde{y}(\xi,t)=0 &  \xi \in \partial \Omega \times [0,T]\\
 \tilde{y}(\xi,0)=\tilde{y}_0(\xi) & \xi \in \Omega. 
\end{cases}
\end{equation}
%with space domain $\Omega = [0,5]^2$, $v=1$ and final time $T=2.5$.
%and initial condition $y(x_1,x_2) = \sin(\pi x_1)\sin(\pi x_2) \chi_{[0,1]^2}$.
%with $u(t)\tilde{y}(x,t)$ being a control term of this equation. 
Equation \eqref{transport} can be written in form \eqref{dyn} using finite differences spatial discretization which leads to a system of ODEs:
\begin{equation}
    \label{transport_dyn}
    \begin{cases}
    \dot{y}(t) = A y(t) +  u(t)y(t) & t \in (0,\infty]\\
    y(0) = \tilde{y}_0(x) & x \in \Omega\\
    \end{cases}
\end{equation}
where $A \in \mathbb{R}^{d \times d}$ is the discretization of the gradient term, $y(t) \in \mathbb{R}^d$ and control $u(t)\in U$. Our goal is to steer the solution to $0$, minimizing the following cost functional:
\begin{equation}
	\label{cost_func_trans}
	\mathcal{J}_x(y, u) \equiv \int_0^{\infty} (\|y(s)\|_{2}^{2}  + \gamma |u(s)|^2) e^{-\lambda s}ds
\end{equation}
where $y(t)$ solves \eqref{transport_dyn}. The parameters in \eqref{transport} are $\Omega = [0,5]^2$, $v=1$ and $T=2.5$ is chosen large enough to simulate the infinite horizon. We also set $\gamma = 10^{-5}$ in \eqref{cost_func_trans}.

In this test we select initial conditions from a class of parametrized functions:
\begin{equation}\label{class}
\mathcal{C} \coloneqq \left\{k\sin(\pi x_1)\sin(\pi x_2)\chi_{[0,1]^2}; \text{ } k \in (0,1] \right\}.
\end{equation}
To generate the grid driven by the dynamics we have chosen $k=\{0.5,1\}$ in \eqref{class}, $11$ controls equidistributed in $U=[-2,0]$ and $\bar \Delta t = 0.1$. Thus, in \eqref{eq:mesh}, we set:
$\bar{\Delta t} = 0.1; 
\bar{M} = 11;
\bar{L}=2$. Equation \eqref{transport_dyn} is then solved for each initial condition and each control with $\Omega$ discretized with  $d=10201$ points, and the time interval $[0,2.5]$. % In order to run Algorithm \ref{algo1} and obtain the value function, then we use $21$ equidistributed controls to discretize $U$ together with a temporal step size $\Delta t = 0.05$.
After computing the grid generated by the dynamics, we run Algorithm \ref{algo1} with $\mathcal{P} = [0.4, 0.7]$ discretized with step size $0.05$, fixing $\Delta t = 0.05$ to discretize \eqref{transport_dyn} by an implicit Euler method and $21$ equidistributed controls. The residual $R(\theta)$ reaches is minimimum with $\bar{\theta }= 0.65$ as shown in the bottom-left panel of Figure \ref{initial_075sinsin_transp}. The CPU time to run our algorithm is $583.6$ seconds.

%\begin{figure}[htbp]
%	\label{residual_pde}
%	\centering
%		\includegraphics[scale = 0.3]{residual_transport_inf2.pdf}
%		\caption{Test 2. Residual.}
%\end{figure}

% \begin{figure}[htbp]
% 	\label{residual_pde_transp}
% 	\centering
% 		\includegraphics[scale = 0.3]{residual_transp_inf_final.pdf}
% 		\caption{Test 2. Residual}
% \end{figure}

To obtain the feedback control and optimal trajectories we have further discretized the set $U$ considering $81$ points. 
Thus, we have studied the control problem for different initial conditions selected from the set \eqref{class} using the value function already stored. In Figure \ref{initial_075sinsin_transp} we can compare in the top panels the uncontrolled solution, i.e., the transport of initial condition in $\Omega$ considering $u(t) = 0$, with controlled solution for $y(x,0)=0.75\sin(\pi x_1)\sin(\pi x_2)\chi_{[0,1]^2}$. The respective optimal control is shown in the bottom-right panel of Figure \ref{initial_075sinsin_transp}. Note that this initial condition does not belong to the grid.
%\todo[inline]{Redo Figure \ref{initial_075sinsin_transp}.}
% Similar results are in Figures \ref{initial_05sinsin_transp} and \ref{initial_075sinsin_transp} to initial conditions $y(x,0)=0.5\sin(\pi x_1)\sin(\pi x_2)\chi_{[0,1]^2}$ and $y(x,0)=0.75\sin(\pi x_1)\sin(\pi x_2)\chi_{[0,1]^2}$ respectively. 
\begin{figure}[htbp]
	\label{initial_075sinsin_transp}
	\centering 
	\includegraphics[scale = 0.35]{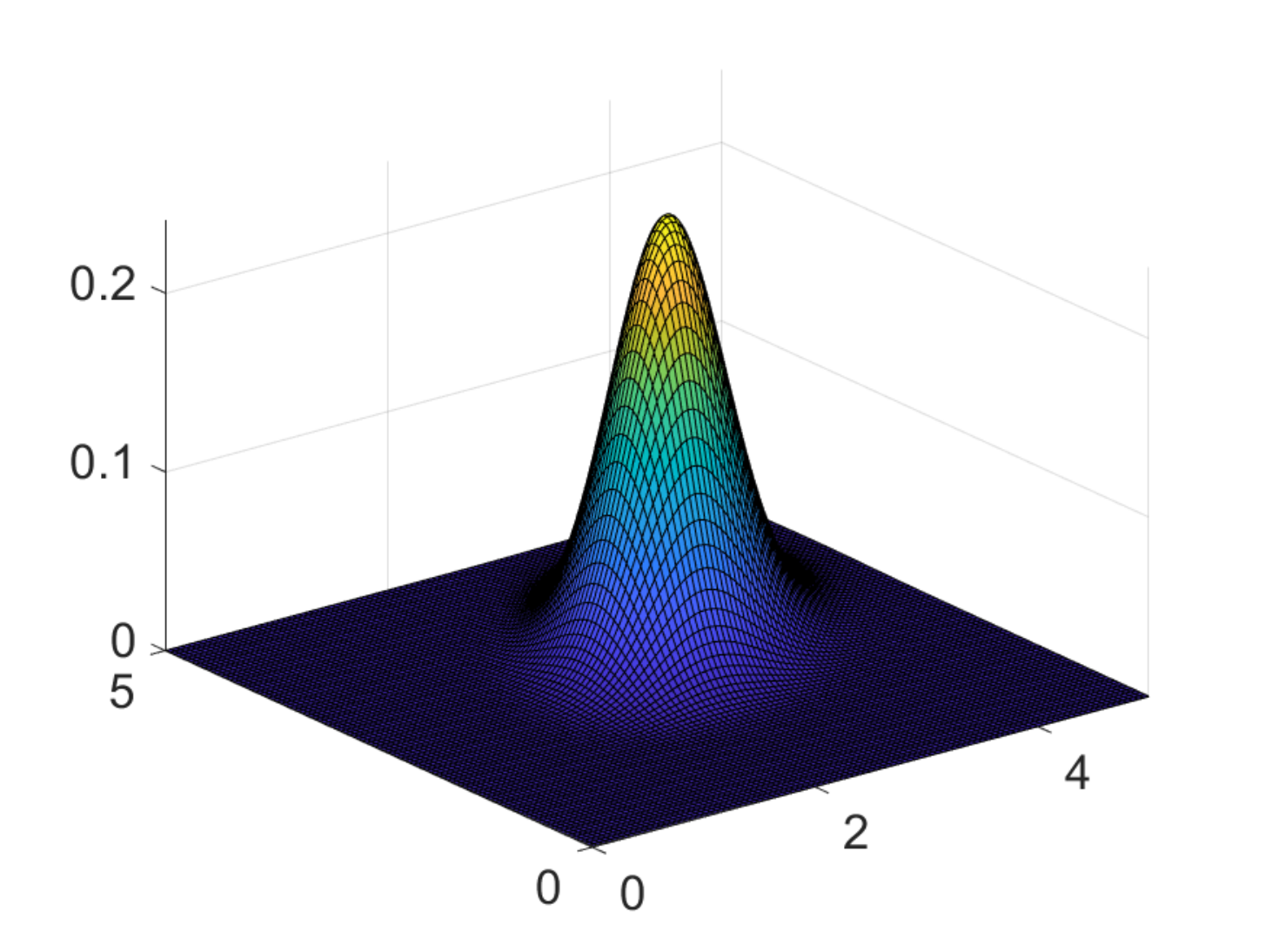}
	\includegraphics[scale = 0.35]{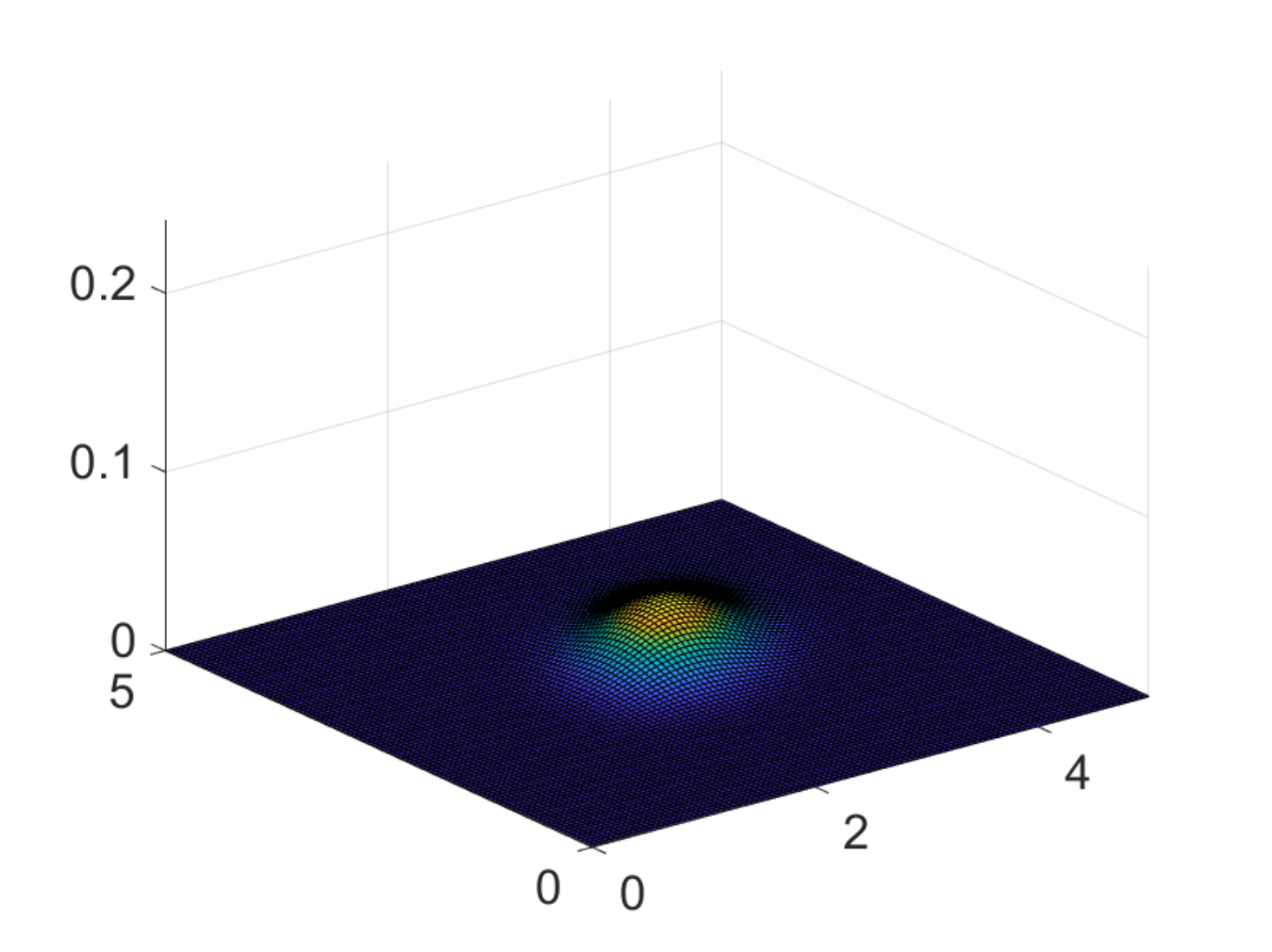}\\
	\includegraphics[scale = 0.35]{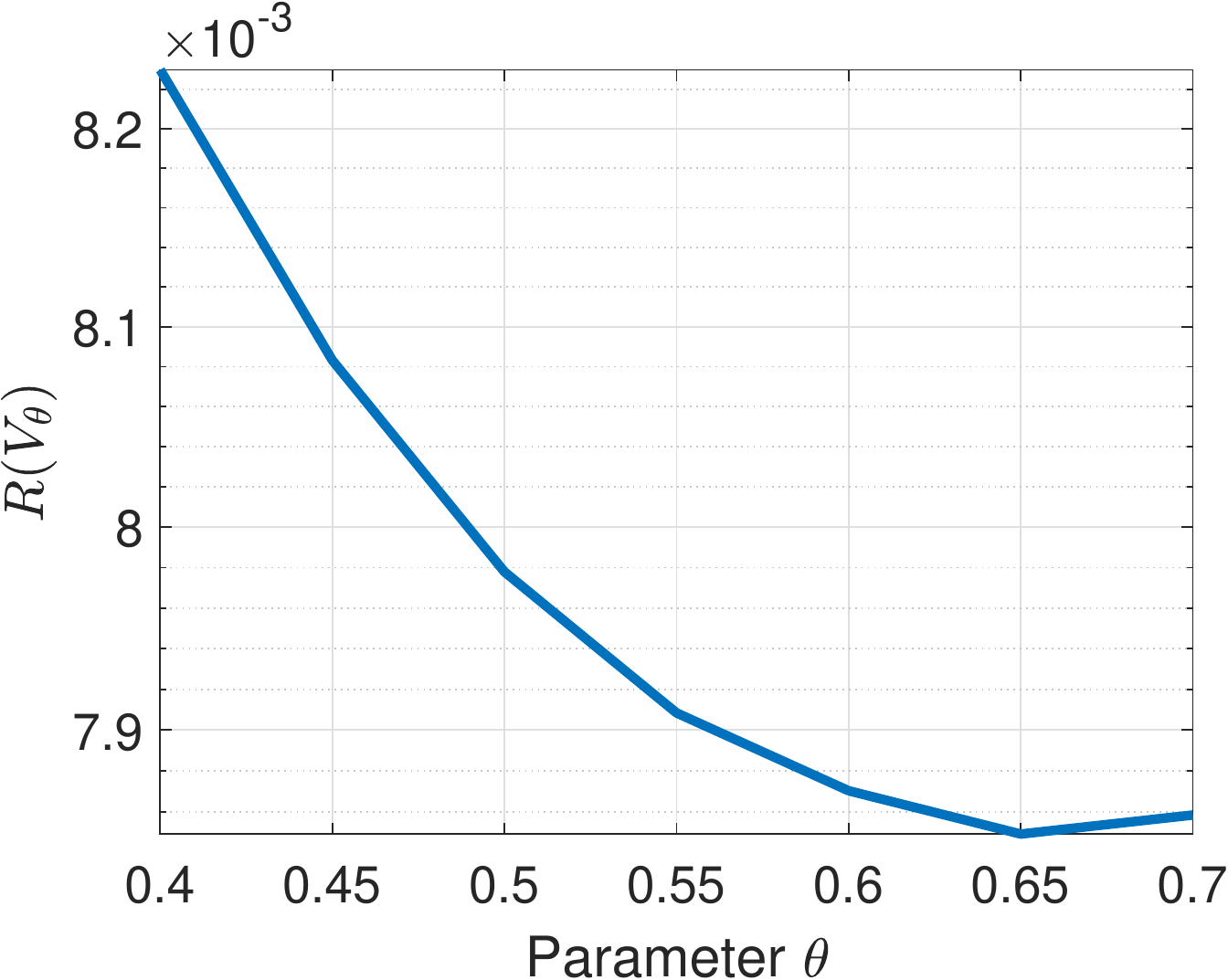}
	\includegraphics[scale = 0.35]{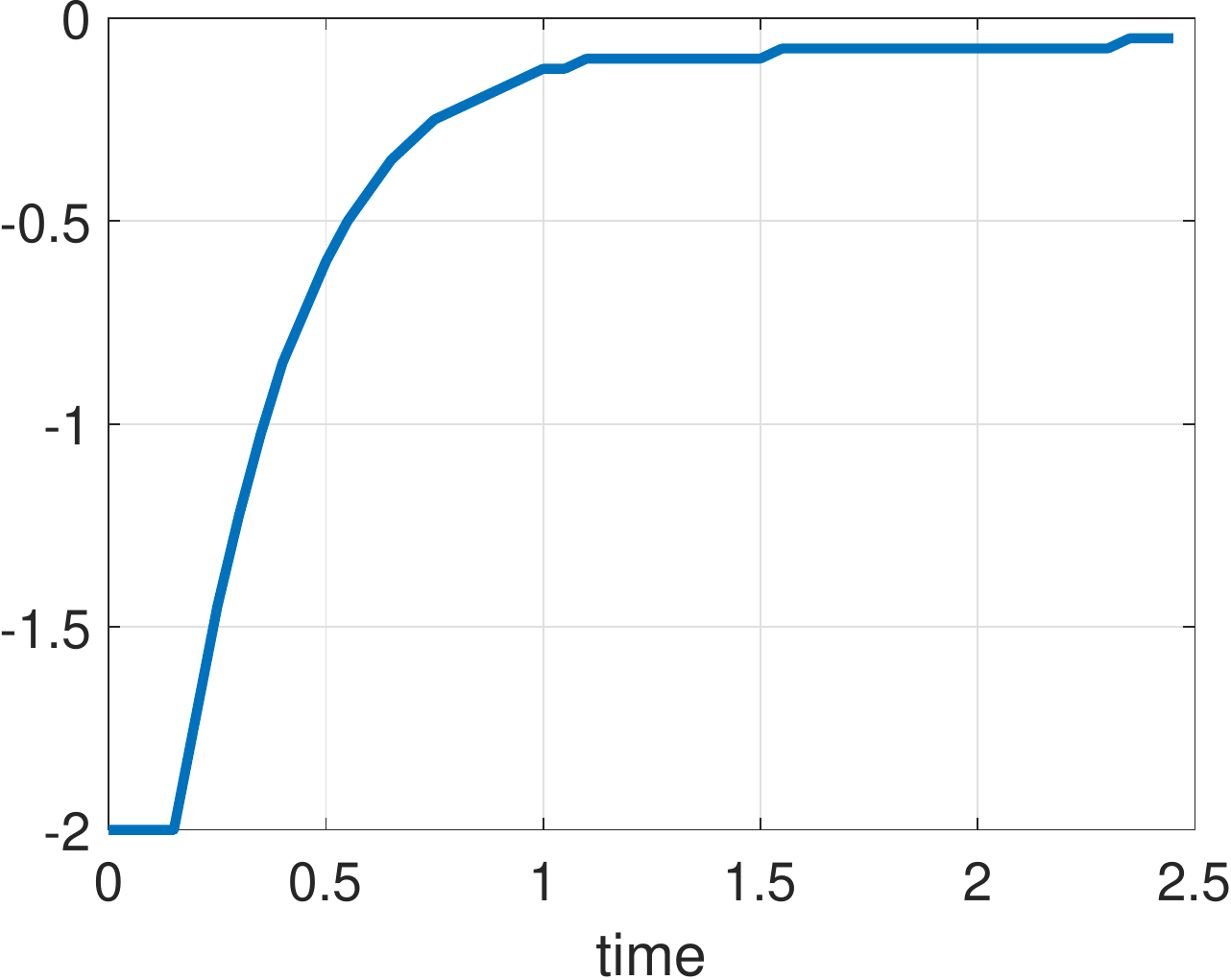}
	\caption{Test 2. Initial condition $y(x,0)=0.75sin(\pi x_1)sin(\pi x_2)\chi_{[0,1]^2}$. Top:  uncontrolled solution (left) and controlled solution (right). Bottom: residual (left) and optimal control (right).}
\end{figure}

We have also studied other initial conditions, with $k=\{0.5,1\}$ in \eqref{class}. The behavior of the solution is similar to Figure \ref{initial_075sinsin_transp}. We show a plot of the cost functionals in Figure \ref{transp_costfunct}, and we see that the controlled solution has always a lower cost functional than the uncontrolled solutions.
We are able to reach the desired configuration with the three initial conditions considered. 
% The evaluation of the cost functional for each initial condition is presented in Table \ref{costs_transp}. As expected, the costs of controlled solutions are always smaller than the costs of uncontrolled solutions.

% \begin{figure}[htbp]
% 	\label{initial_sinsin_transp}
% 	\centering 
% 	\includegraphics[scale = 0.27]{transp_uncontrolled_sinsin_inf.pdf}
% 	\includegraphics[scale = 0.27]{transp_controlled_sinsin_inf.pdf}
% 	\includegraphics[scale = 0.27]{transp_optimal_control_sinsin_inf.pdf}
% 	\caption{Test 2. Initial condition $y(x,0)=sin(\pi x_1)sin(\pi x_2)\chi_{[0,1]^2}$. Left: uncontrolled solution. Center controlled solution. Right: optimal control.}
% \end{figure}

% \begin{figure}[htbp]
% 	\label{runcost_sinsin}
% 	\centering
% 		\includegraphics[scale = 0.27]{runcost_sinsin_inf.pdf}
% 			\includegraphics[scale = 0.27]{runcost_05sinsin_inf.pdf}
% 			\includegraphics[scale = 0.27]{runcost_075sinsin_inf.pdf}
% 		\caption{Test 2. Running Cost with initial condition $y(x,0)=sin(\pi x_1)sin(\pi x_2)\chi_{[0,1]^2}$.\ale{??eliminabile tipo}}
% \end{figure}

\begin{figure}[htbp]
	\label{transp_costfunct}
	\centering
		
		\includegraphics[scale = 0.27]{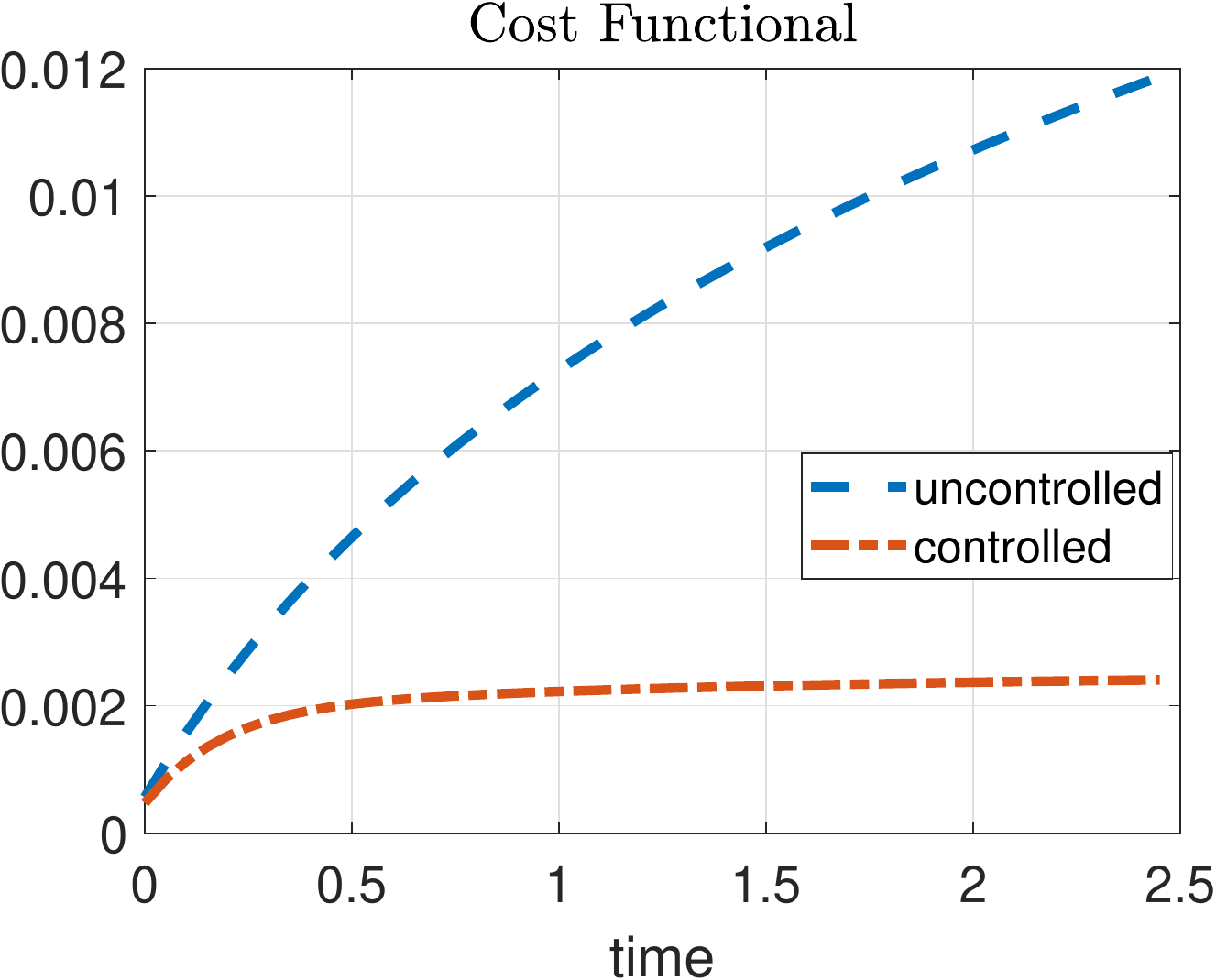}
		\includegraphics[scale = 0.27]{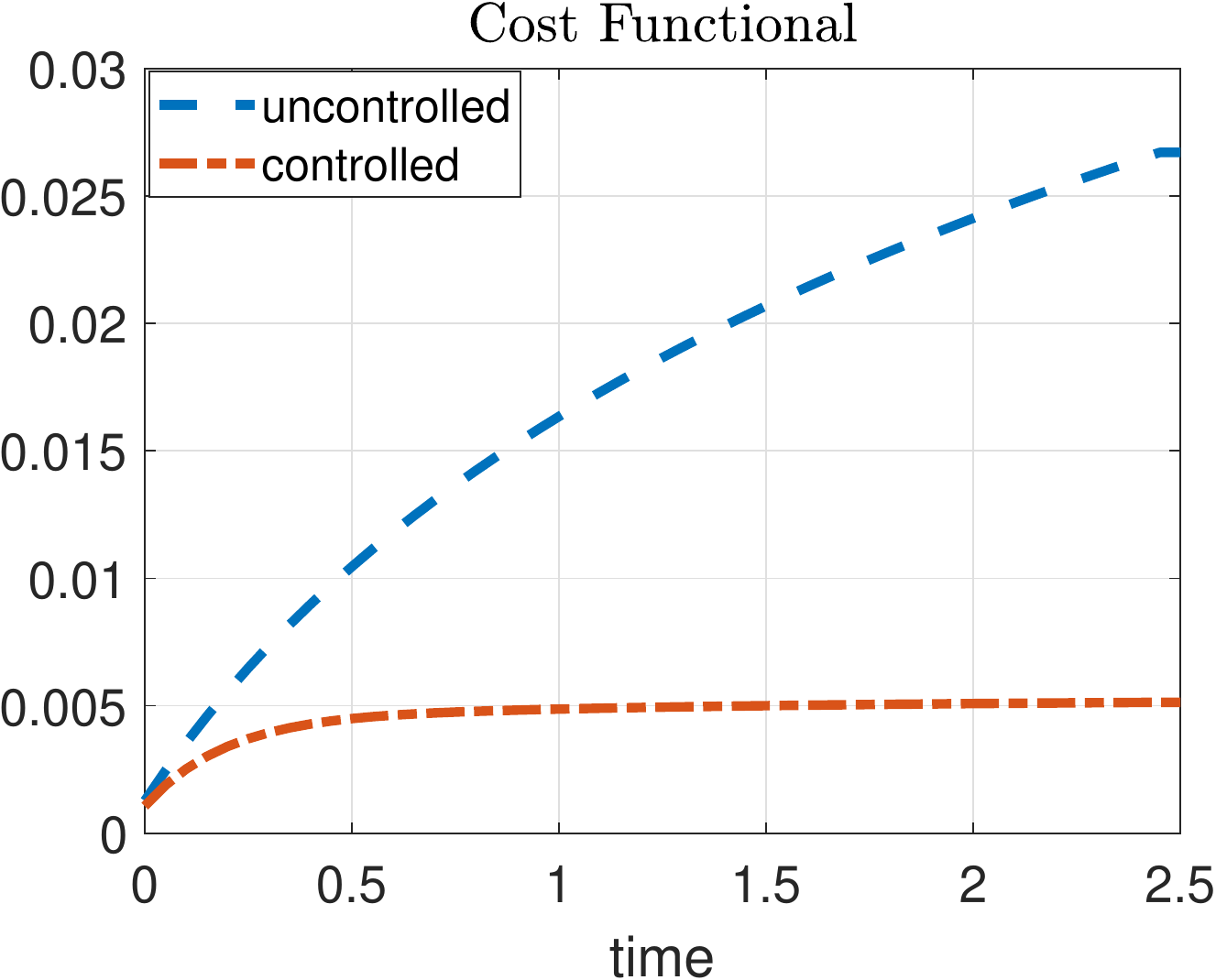}
		\includegraphics[scale = 0.27]{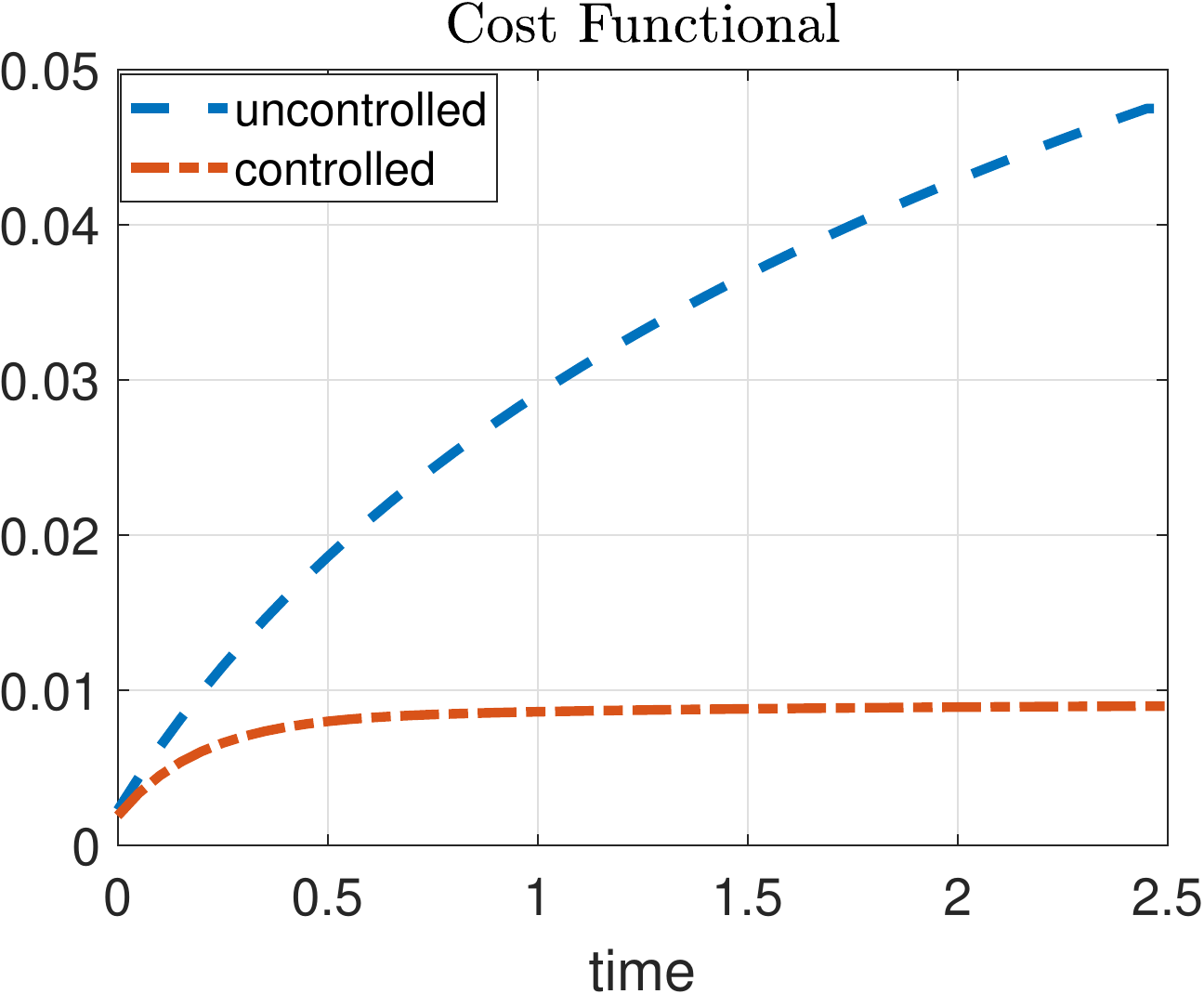}
		\caption{Test 2. Cost functional. Left: $y(x,0)=0.5 sin(\pi x_1)sin(\pi x_2)\chi_{[0,1]^2}$. Center: $0.75y(x,0)=sin(\pi x_1)sin(\pi x_2)\chi_{[0,1]^2}$. Right: $y(x,0)=sin(\pi x_1)sin(\pi x_2)\chi_{[0,1]^2}$.}
\end{figure}

\subsection*{Test 3: Nonlinear Heat Equation}
		This test deals with the control of a two-dimensional parabolic equation with polynomial  nonlinearities: 
\begin{equation}
\label{heat}
    \begin{cases}   \tilde{y}_{t}(\xi,t) = \alpha \Delta \tilde{y}(\xi,t) + \beta (\tilde{y}^2(\xi,t)-\tilde{y}^3(\xi,t))+ u(t)\tilde{y}_{0}(\xi) & (\xi,t) \in \Omega \times [0,\infty)\\ 
		\partial_{n} \tilde{y}(\xi,t)=0 & \xi \in \partial \Omega \times [0,\infty)\\
	\tilde{y}_{0}(\xi) = \tilde{y}(\xi,0) & \xi \in \Omega \\
\end{cases}
\end{equation}	
with $\Omega = [0,1]\times[0,1], \alpha =\frac{1}{100}, \beta = 6$. Finite differences in space for \eqref{heat} leads to 
\begin{equation}\label{eq:non}
    \begin{cases}
            \dot{y}(t) = A y(t) +  B u(t) + f(y(t)) & t \in (0,\infty]\\
    y(0) = \tilde{y}_0(\xi) & \xi \in \Omega
    \end{cases}
\end{equation}
where $A \in \mathbb{R}^{d \times d}$ is the discretization of the laplacian, $B \in \mathbb{R}^{d}$ where $B_i = y_{0}(\xi_i)$ for $i=1, \ldots, d$ and $\xi_i$ a node of the discretization of $\Omega$. Here the nonlinear term $f: \mathbb{R}^d \rightarrow \mathbb{R}^d$ is $f(y(t)) = y(t)^2 - y(t)^3$.

We want to minimize again \eqref{cost_func_trans} as in the previous test and consider the class \eqref{class} of initial conditions. Here, we build an unstructured mesh using as initial condition with $k=\{0.5,1\}$ in \eqref{class}. In \eqref{eq:mesh}, we set
$\bar{\Delta t} = 0.1, 
\bar{M} = 41,
\bar{L}=2$. The state space $\Omega = [0,1]^2$ was discretized in $31^2$ points which is the dimension of the discretized problem \eqref{eq:non}. The control space  $U=[-2,0]$ is discretized with $41$ points. The time domain is $T = [0,5]$ which was discretized in $51$ points. We set $\gamma = 10^{-4}$ in \eqref{cost_func_trans} and run Algorithm \ref{algo1} using $\mathcal{P} = [2, 2.4]$ with step size of $0.05$ and $\Delta t = 0.075$. The parameter that minimizes $R(\theta)$ is $\bar{\theta} = 2.25$ as shown in Figure \ref{residual_pde_heat}. %We can see that after the value the residual jumps to higher values.
The time needed to approximate the value function is approximately $20$ minutes.

%\begin{figure}[htbp]
%	\label{residual_pde}
%	\centering
%		\includegraphics[scale = 0.3]{residual_heat_inf.pdf}
%		\includegraphics[scale = 0.3]{residual_heat_inf_zoom.pdf}
%		\caption{Test 3. Uncontrolled solution at time Left: Residual calculated to all $\mathcal{P}$. Right: Residual restricted in order to show $\theta$ that minimizes $R(V_{\theta})$.}
%\end{figure}

\begin{figure}[htbp]
	\label{residual_pde_heat}
	\centering
		\includegraphics[scale = 0.3]{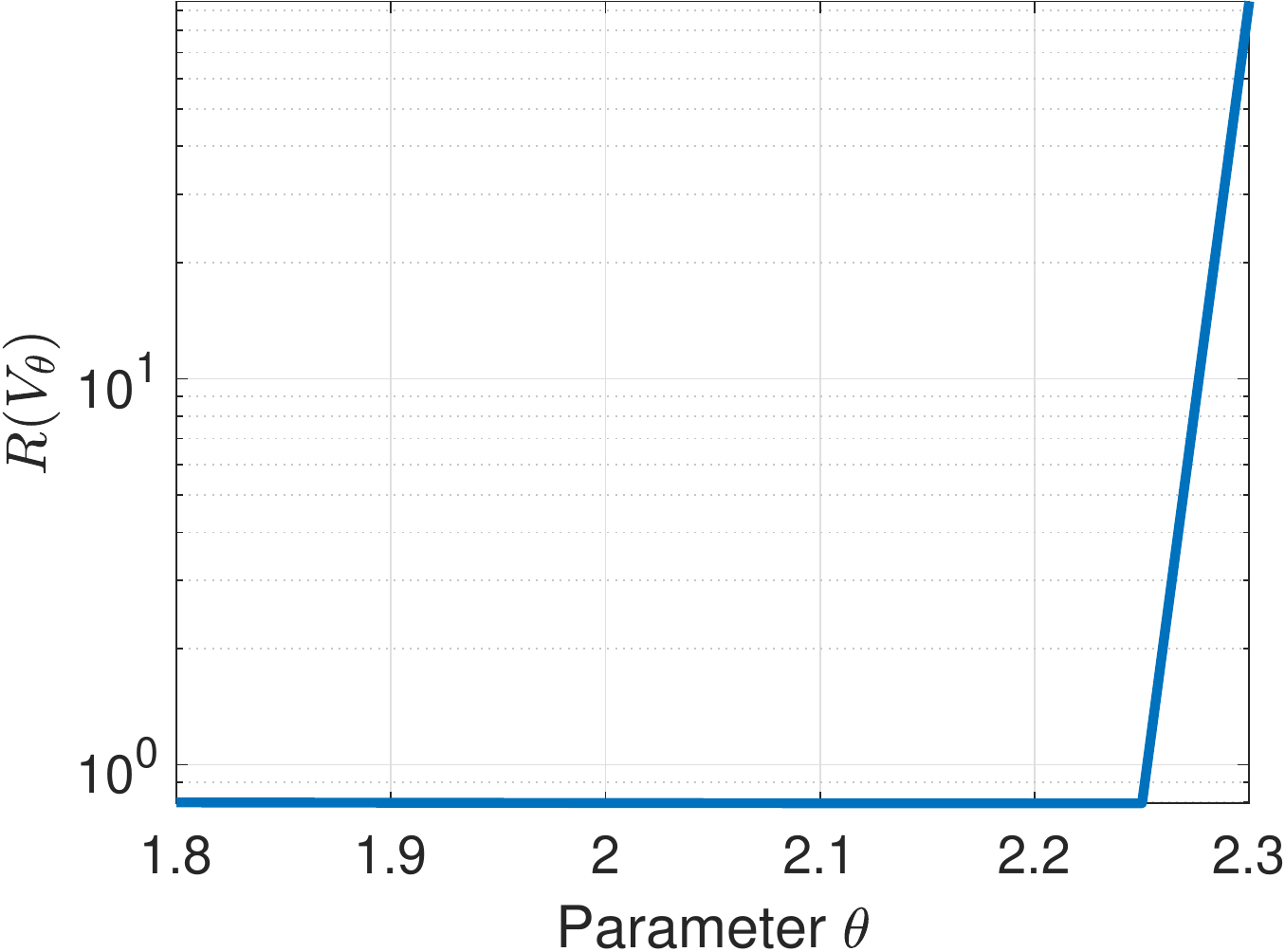}
		\includegraphics[scale = 0.3]{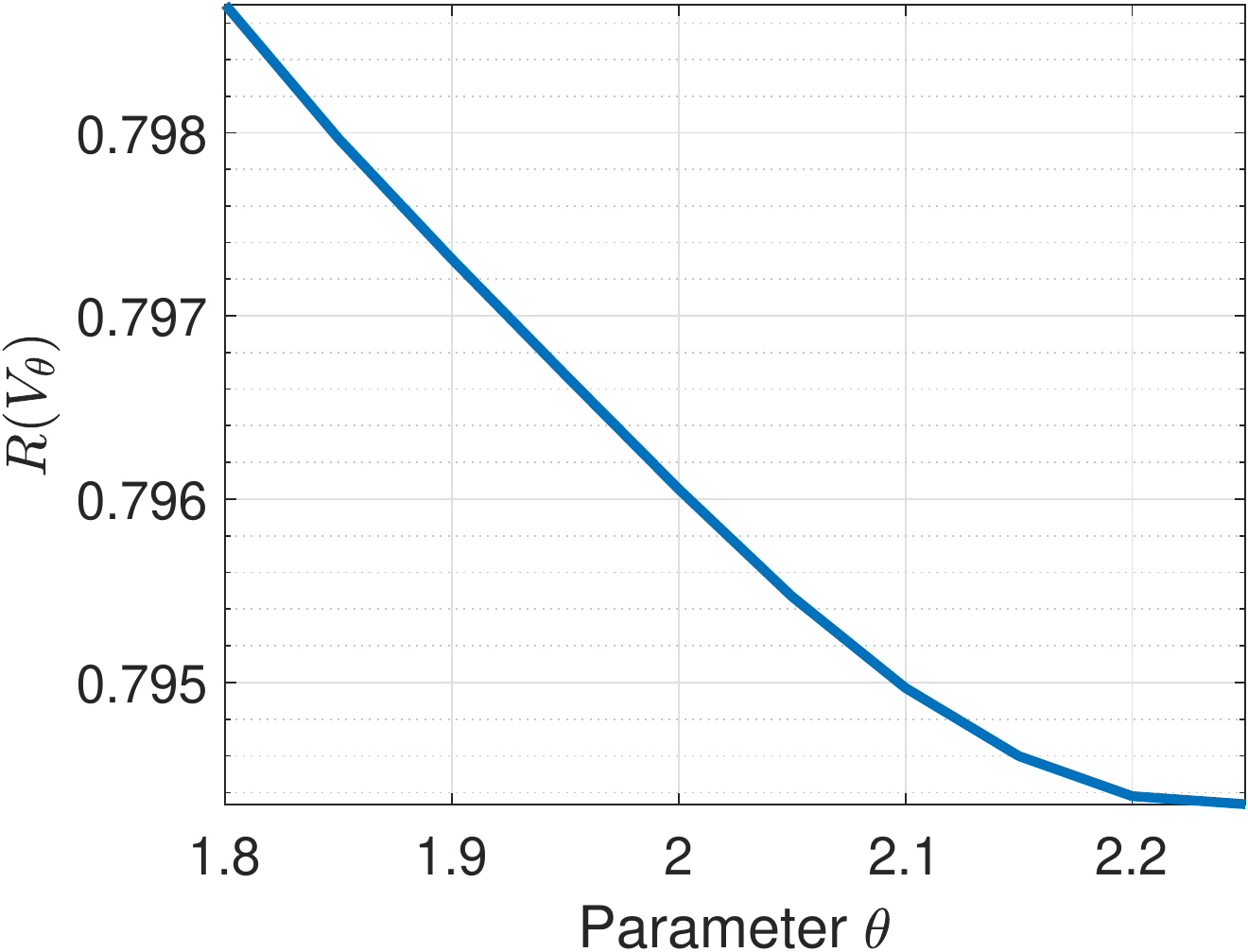}
		\caption{Test 3. Residual.}
\end{figure}

We present the controlled solutions for different initial conditions taken from $\mathcal{C}$ in \eqref{class}. First, we consider the initial condition $y(x,0)=0.75\sin(\pi x_1)\sin(\pi x_2)$ and the results are shown in Figure \ref{initial_sinsin}. As one can see in the left panel the solution reach the unstable equilibrium $y(t)=1$, whereas the controlled solution goes to $0$ as desired. The optimal control is then shown in the right panel of Figure \ref{initial_sinsin}. Note that the initial condition does not belong to the grid where we computed the value function.

\begin{figure}[htbp]
	\label{initial_sinsin}
	\centering 
	\includegraphics[scale = 0.25]{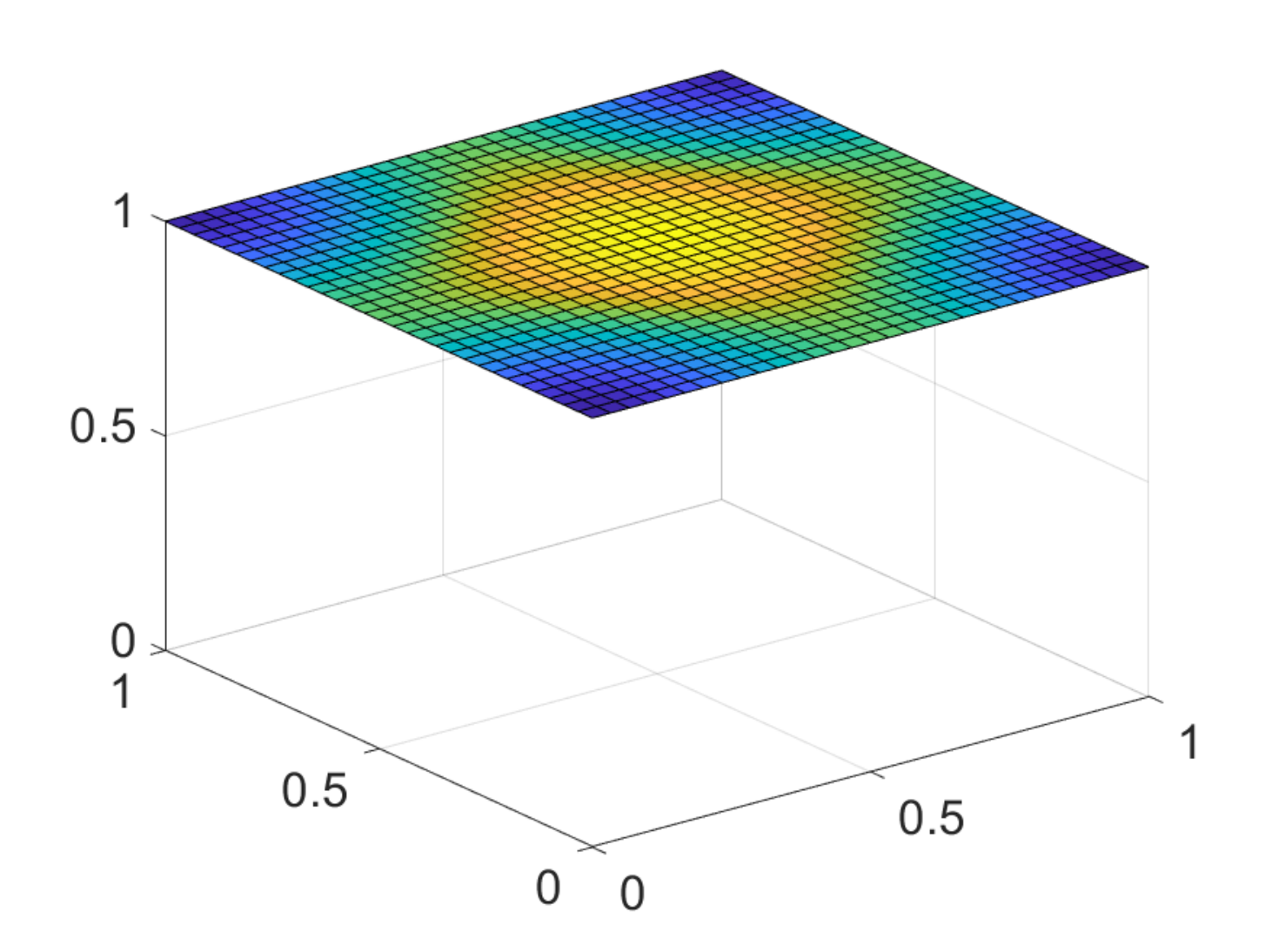}
	\includegraphics[scale = 0.25]{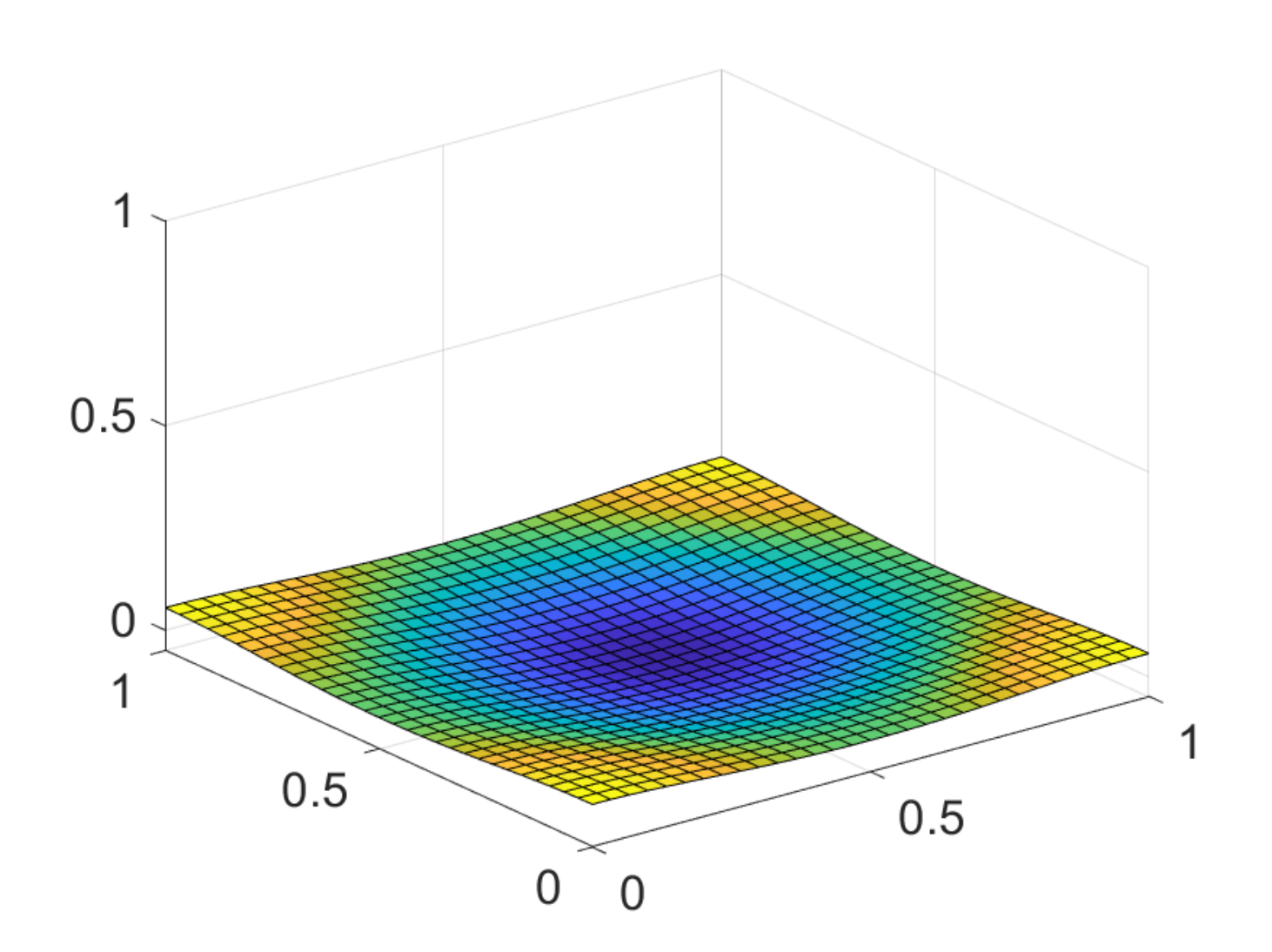}
	\includegraphics[scale = 0.25]{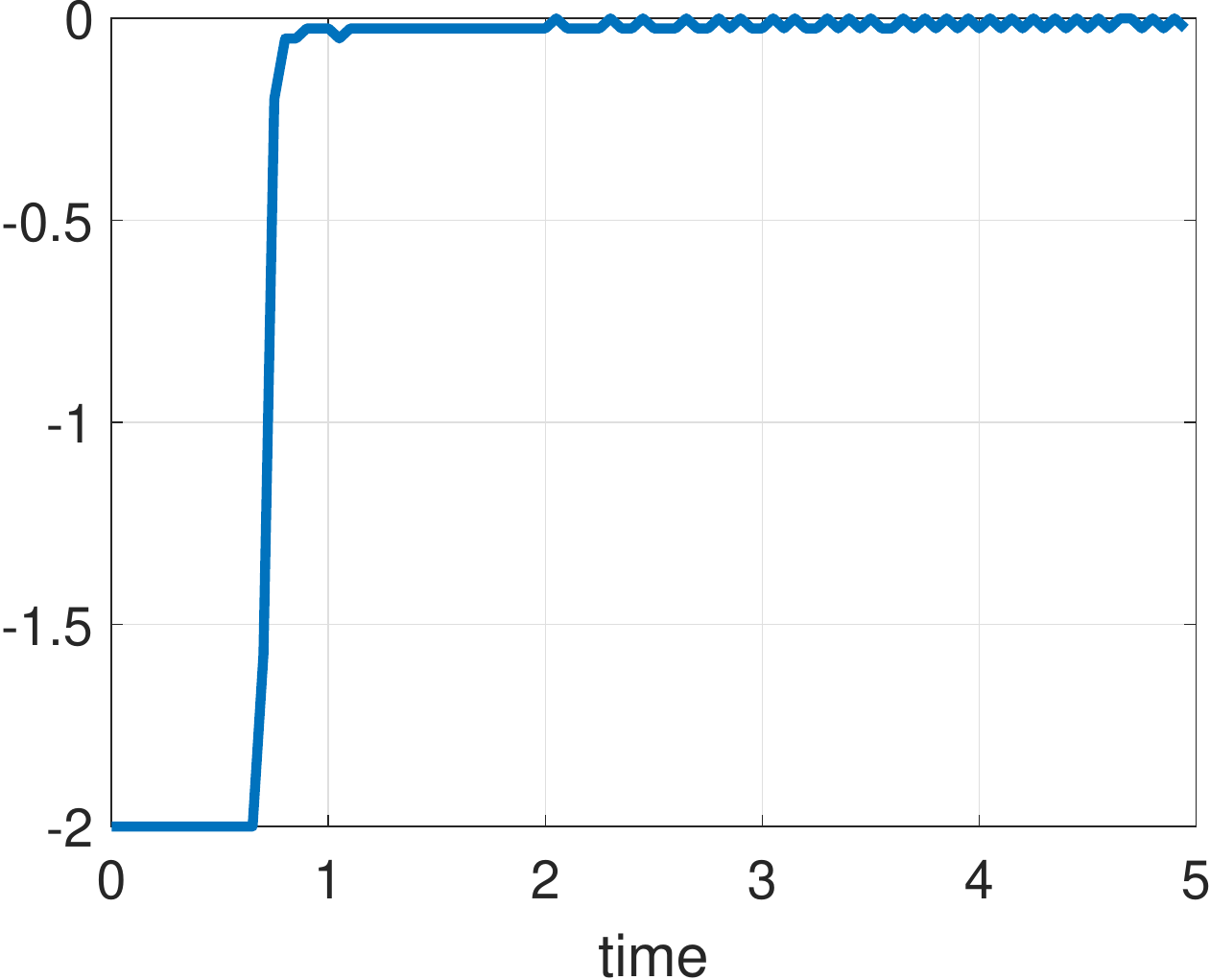}
		\caption{Test 3. Initial condition $y(x,0)=0.75 \sin(\pi x_1)\sin(\pi x_2)$.  Left: uncontrolled solution at time $t=5$. Middle: controlled solution at time $t=5$. Right: optimal control.}
\end{figure}

% \begin{figure}[htbp]
% 	\label{heat_runcost_sinsin}
% 	\centering
% 		\includegraphics[scale = 0.27]{runcost_heat_sinsin_inf.pdf}
% 		\includegraphics[scale = 0.27]{runcost_heat_05sinsin_inf.pdf}
% 		\includegraphics[scale = 0.27]{runcost_heat_075sinsin_inf.pdf}
% 		\caption{Test 3. Running Cost with initial condition $y(x,0)=sin(\pi x_1)sin(\pi x_2)$.\ale{?? eliminabile?}}
% \end{figure}

Figure \ref{heat_costfunct} presents the evaluation of the cost functional for the initial conditions considered. As expected, the cost of controlled solutions is always smaller than costs of uncontrolled solutions.

\begin{figure}[htbp]
	\label{heat_costfunct}
	\centering
	
		\includegraphics[scale = 0.27]{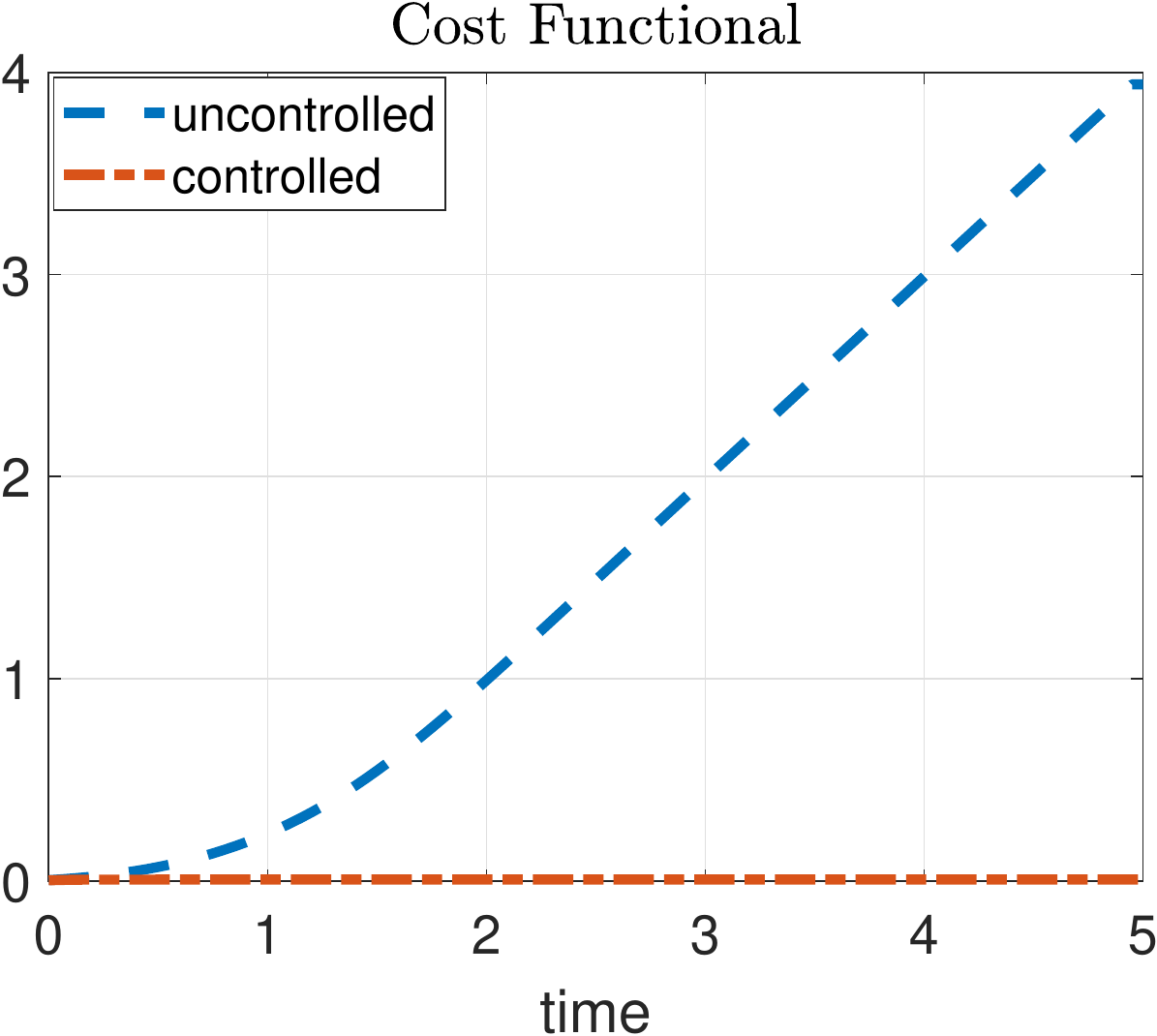}
		\includegraphics[scale = 0.27]{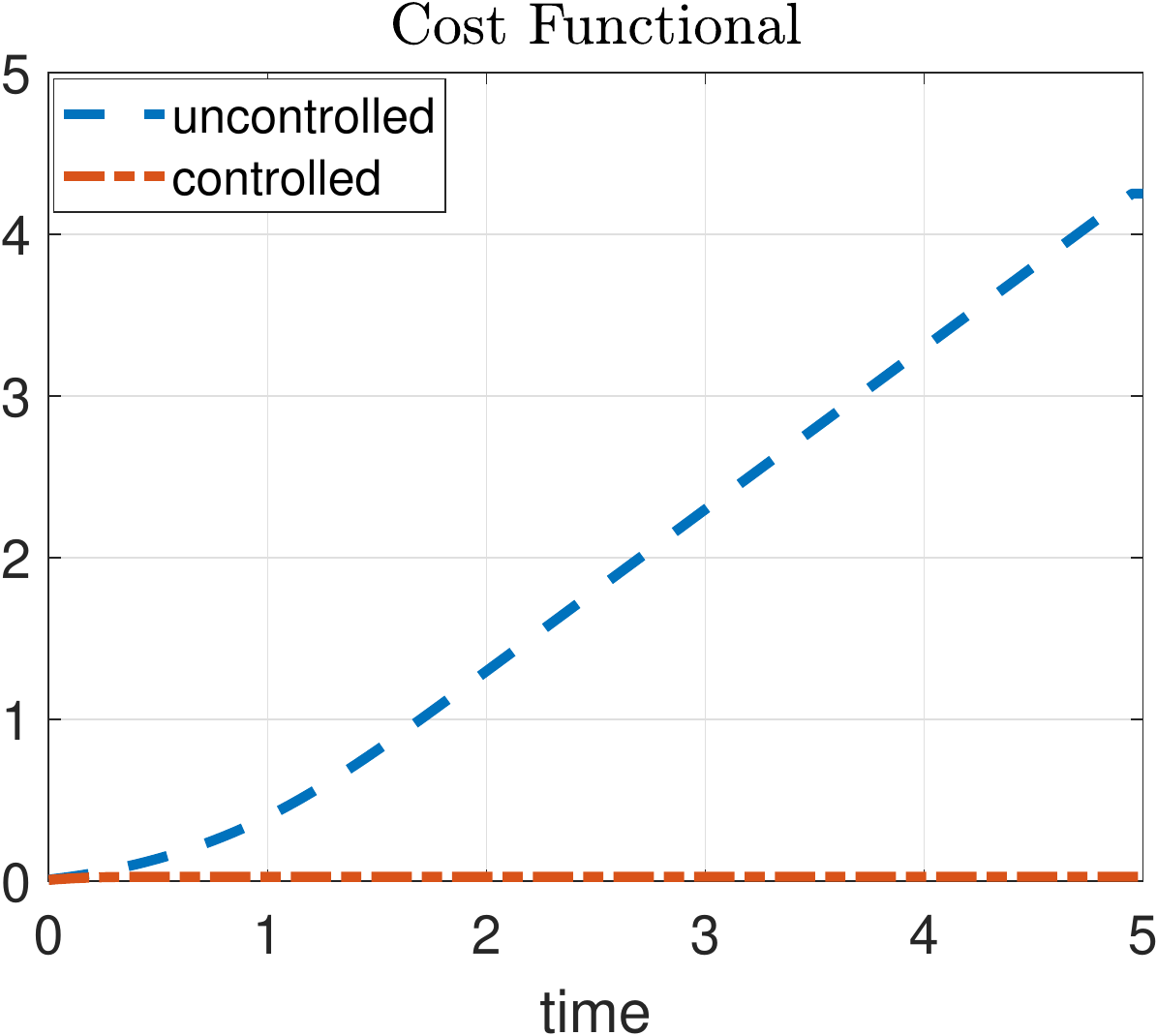}			\includegraphics[scale = 0.27]{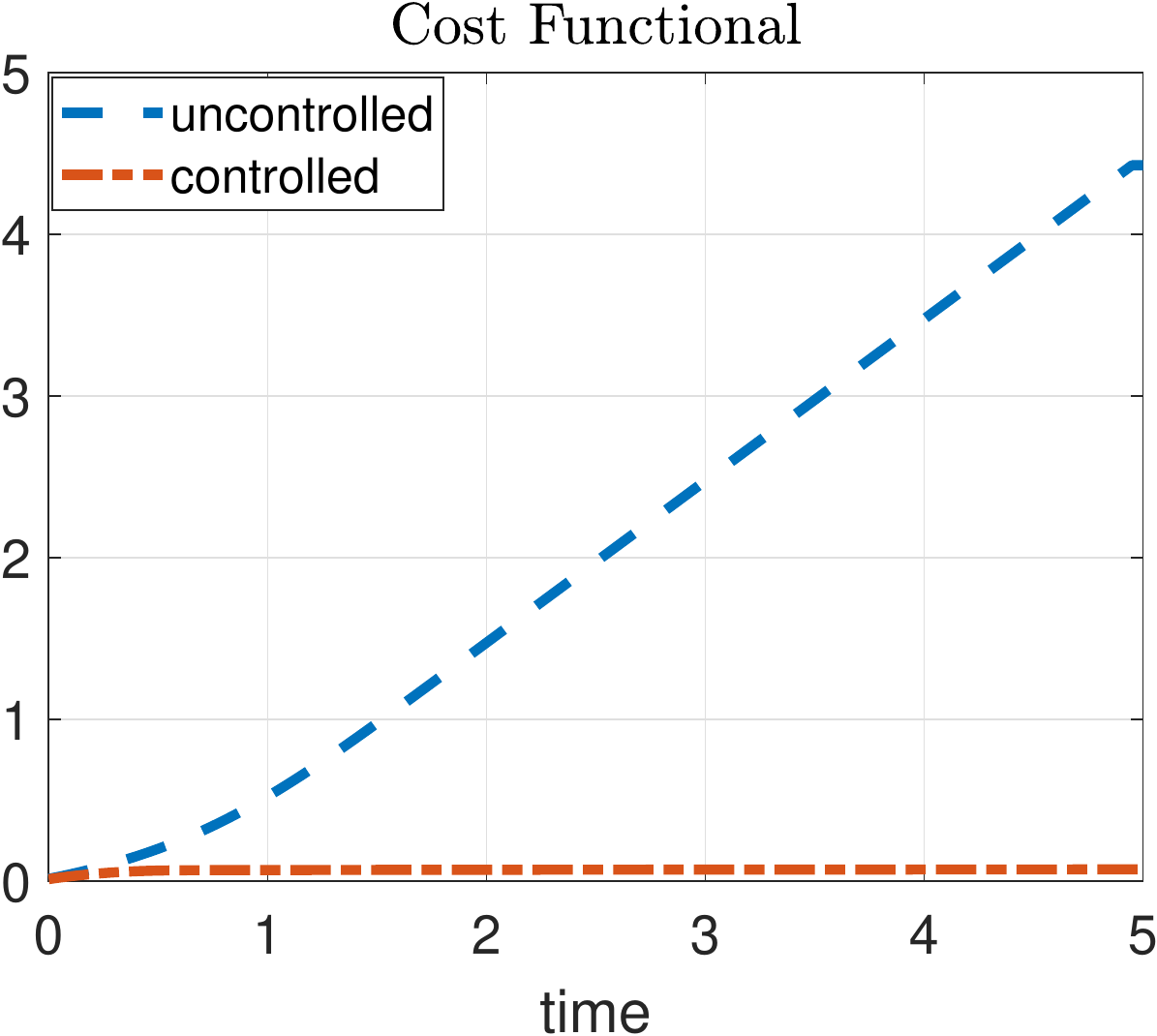}
		\caption{Test 3. Cost functional. Left: $y(x,0)=0.5sin(\pi x_1)sin(\pi x_2)$. Center: $0.75y(x,0)=sin(\pi x_1)sin(\pi x_2)$. Right: $y(x,0)=sin(\pi x_1)sin(\pi x_2)$}
\end{figure}

We now check the robustness of the method adding a noise term at each time instance. We keep using the same value function stored before and we computed the optimal trajectory for a initial condition with a small perturbation. Here we consider $y(x,0)=0.75 \sin(\pi x_1)\sin(\pi x_2) + \mathcal{N}(0, 0.025)$ where $\mathcal{N}(0, 0.025)$ is a normally distributed random variable with mean zero and standard deviation $0.025$. With these parameters we have a probability of $95,45\%$ of selecting a number in the range $[-0.05, 0.05]$ at each iteration. At each time iteration a new independent perturbation $\mathcal{N}(0, 0.025)$ has been added to the trajectory. Left picture of Figure \ref{initial_075sinsin_pert} presents the uncontrolled trajectory and the solution converges (somehow) to $y(x)=1$ with a perturbation. The middle panel shows the controlled solution and how it is close to $y(x)=0$, also with a perturbation. The left panel in the bottom line of Figure \ref{initial_075sinsin_pert} presents the optimal control. A comparison of the cost functional for the perturbed problem is show in the bottom right panel of Figure \ref{initial_075sinsin_pert}. 

% \begin{table}[htbp]
% 	\centering 
% 	\begin{tabular}{ll}
% 		\hline
% 		Uncontrolled	& Controlled \\ \hline
% 		$3.869$	& $0.0777$ \\ \hline
% 	\end{tabular}
% 	\caption{Test 3. Cost functional for initial condition $y(x,0)=sin(\pi x_1)sin(\pi x_2)$.}
% \end{table}

\begin{figure}[htbp]
	\label{initial_075sinsin_pert}
	\centering 
	\includegraphics[scale = 0.3]{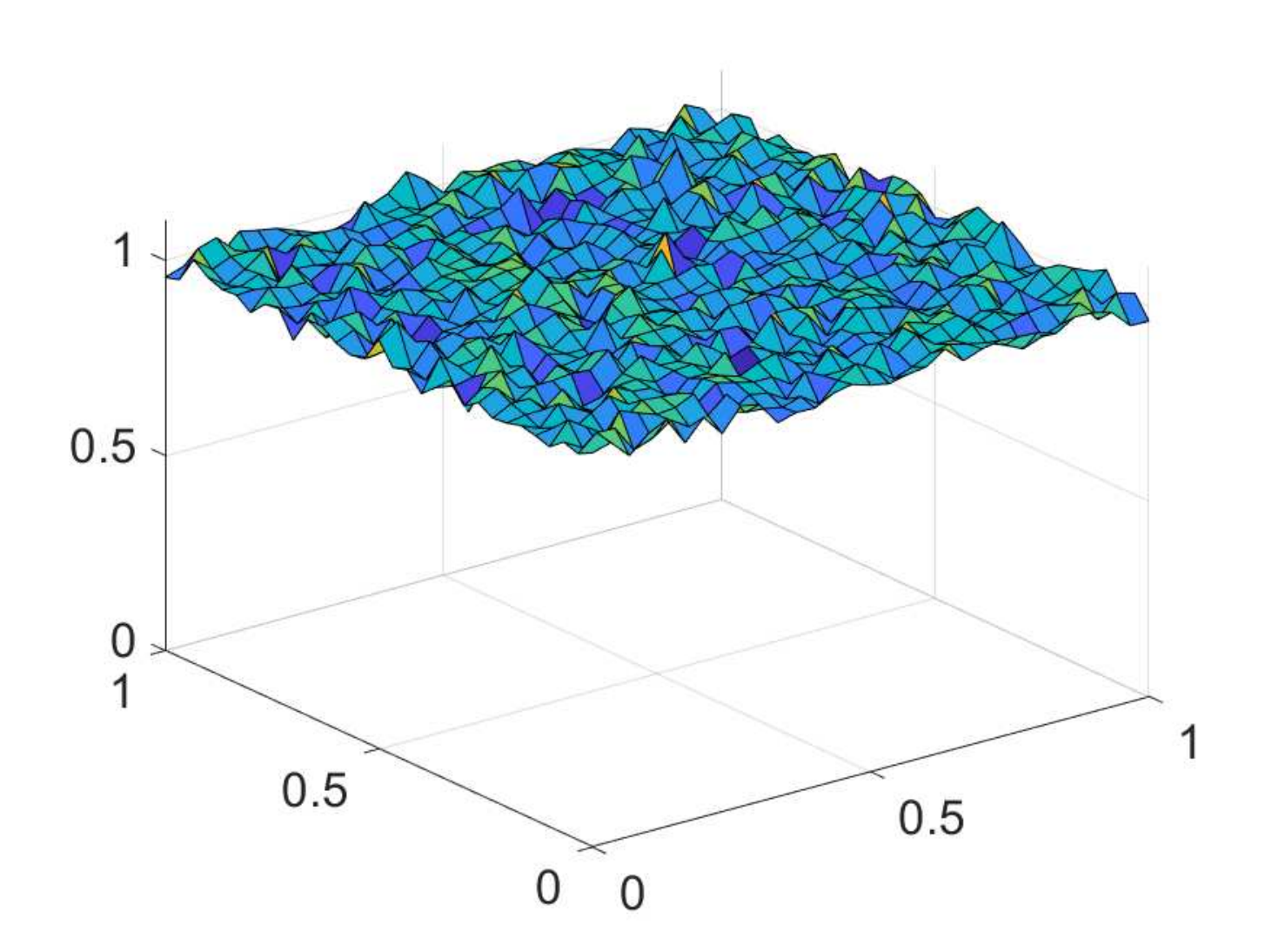}
	\includegraphics[scale = 0.3]{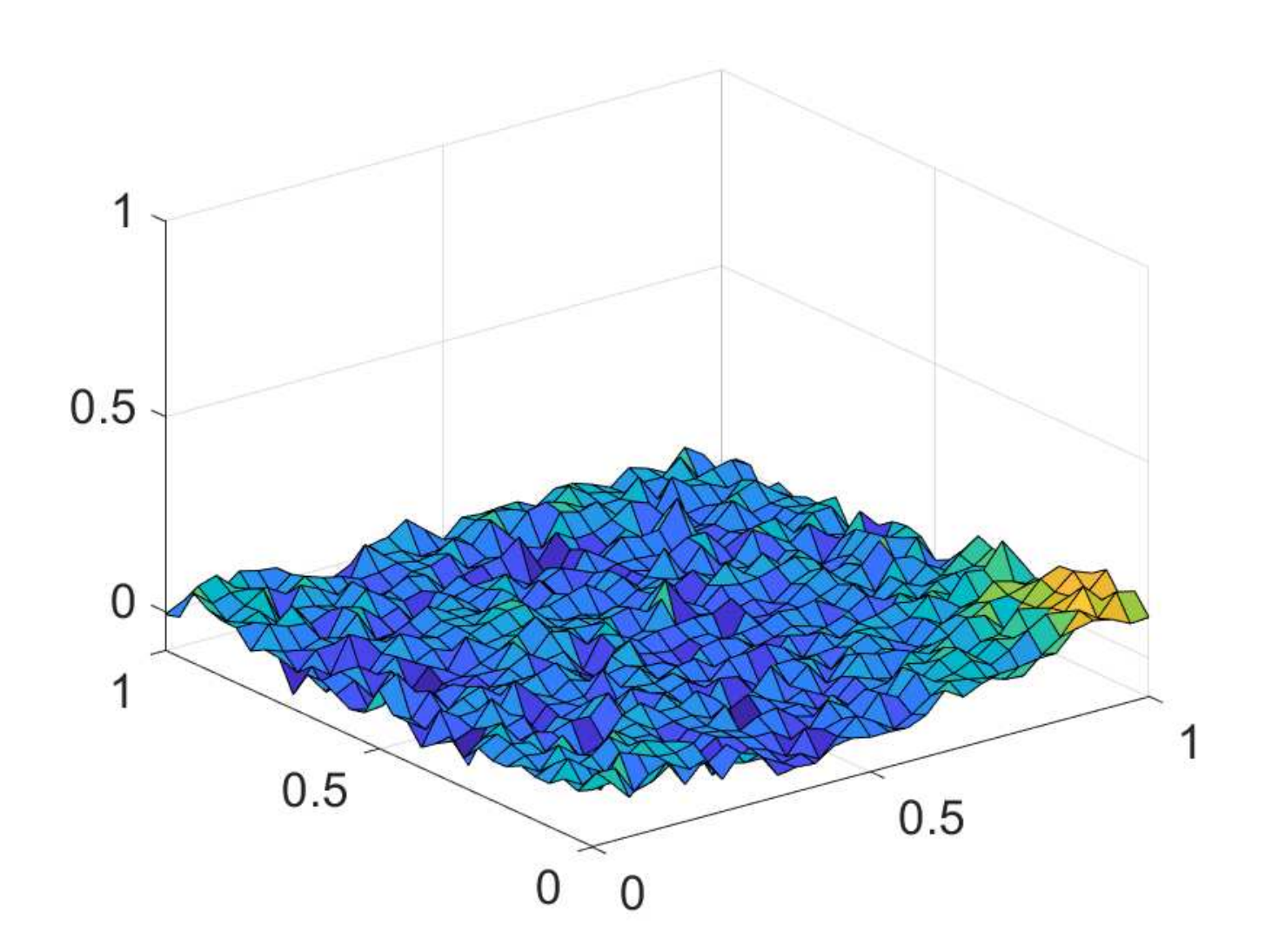}\\
	
	\includegraphics[scale = 0.3]{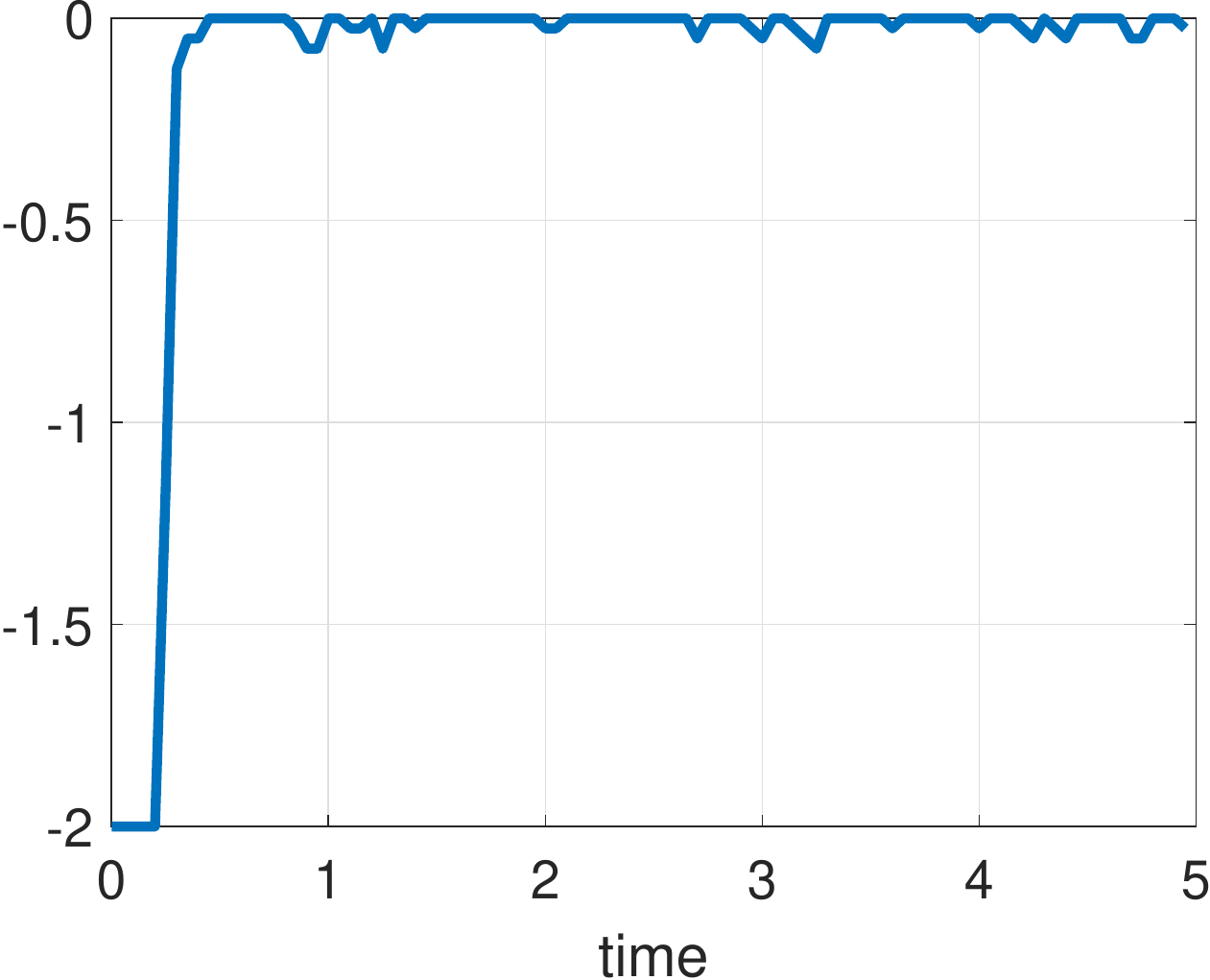}
	\includegraphics[scale = 0.3]{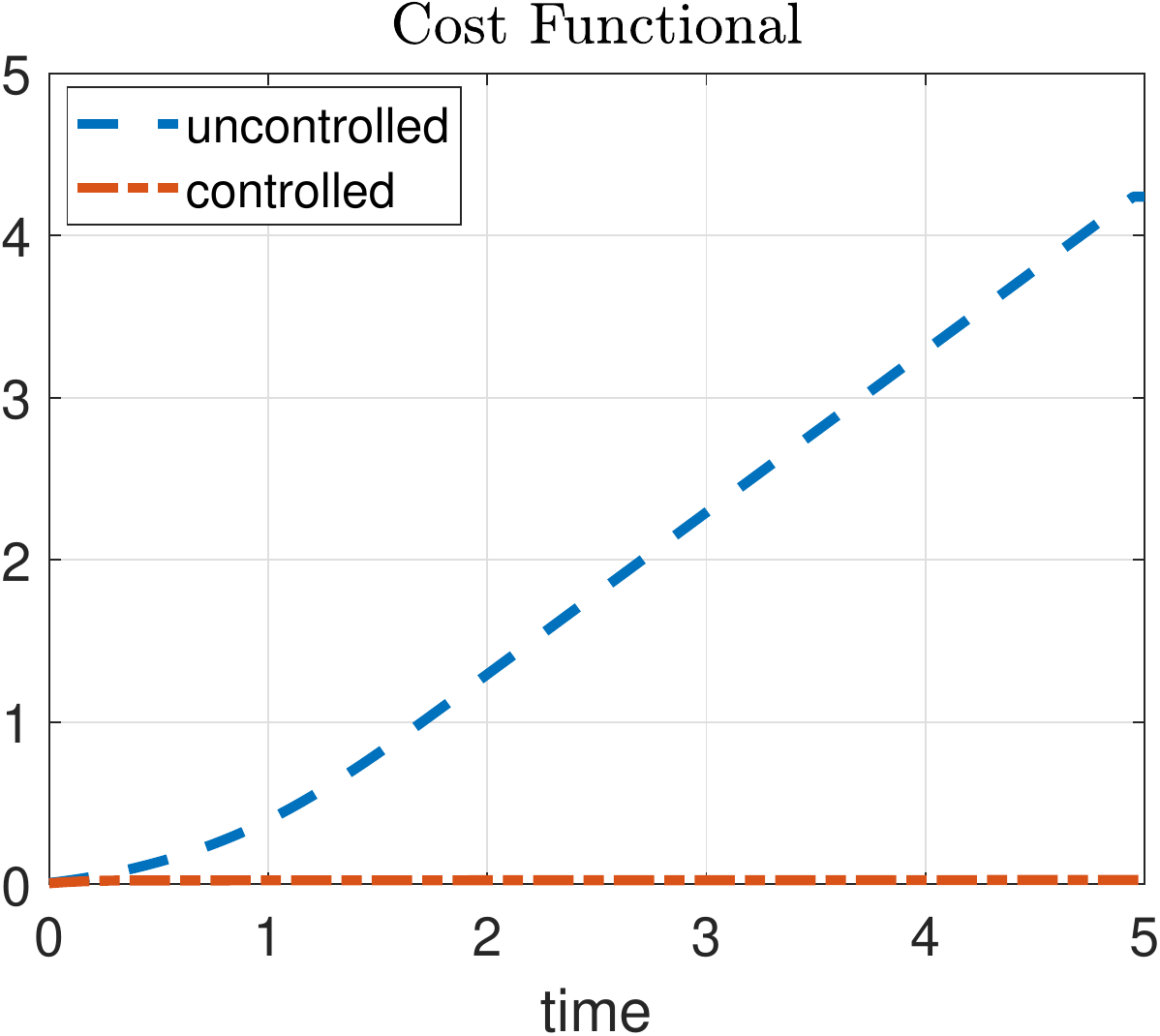}
		
		\caption{Test 3. Initial condition $y(x,0)=0.75sin(\pi x_1)sin(\pi x_2) + \mathcal{N}(0, 0.025)$.  Left: uncontrolled solution. Middle: controlled solution. Right: optimal control. Running Cost with initial condition $y(x,0)=0.75sin(\pi x_1)sin(\pi x_2) + \mathcal{N}(0, 0.025)$}
\end{figure}

Finally, to further show the effectiveness of our method we consider a non-smooth initial condition
that does not belong to \eqref{class}:
$$ y(x,0)=\max \{- (2|x-0.5|+1)(2|y-0.5|+1)+2,0\} $$ as shown in the top left panel of Figure \ref{initial_pyramidal}.
Here we consider a time interval $T=[0,8]$ discretized in $161$ points. 

% \begin{figure}[htbp]
% 	\label{initial_cond}
% 	\centering
% 		\includegraphics[scale = 0.3]{initial.pdf}
% 		\caption{Test 3. Initial condition $y(x,0)=\max \{- (2|x-0.5|+1)(2|y-0.5|+1)+2,0\}$.\ale{?? metti nella figura in basso}.}
% \end{figure}

\begin{figure}[htbp]
	\label{initial_pyramidal}
	\centering 
	\includegraphics[scale = 0.27]{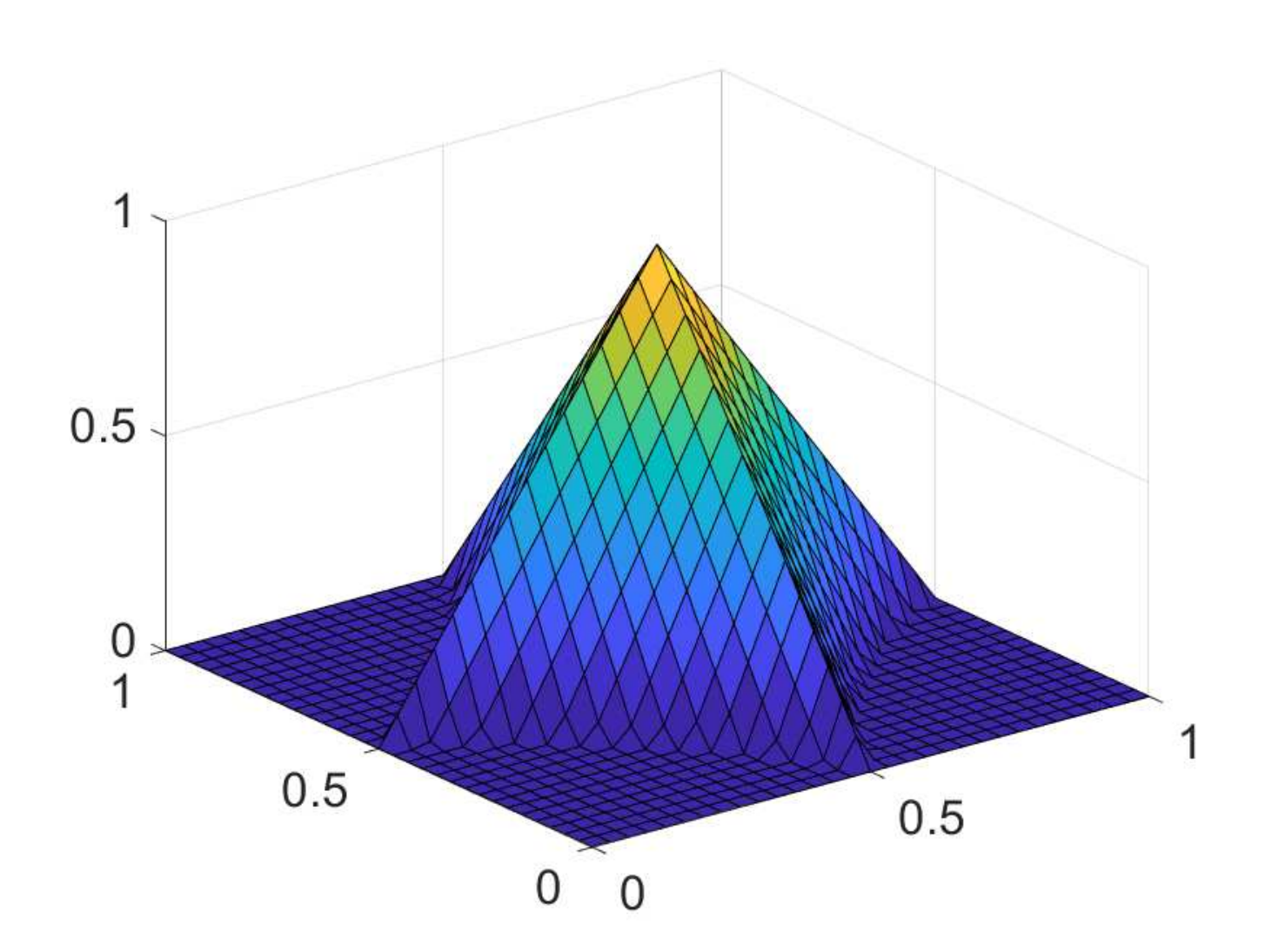}
	\includegraphics[scale = 0.27]{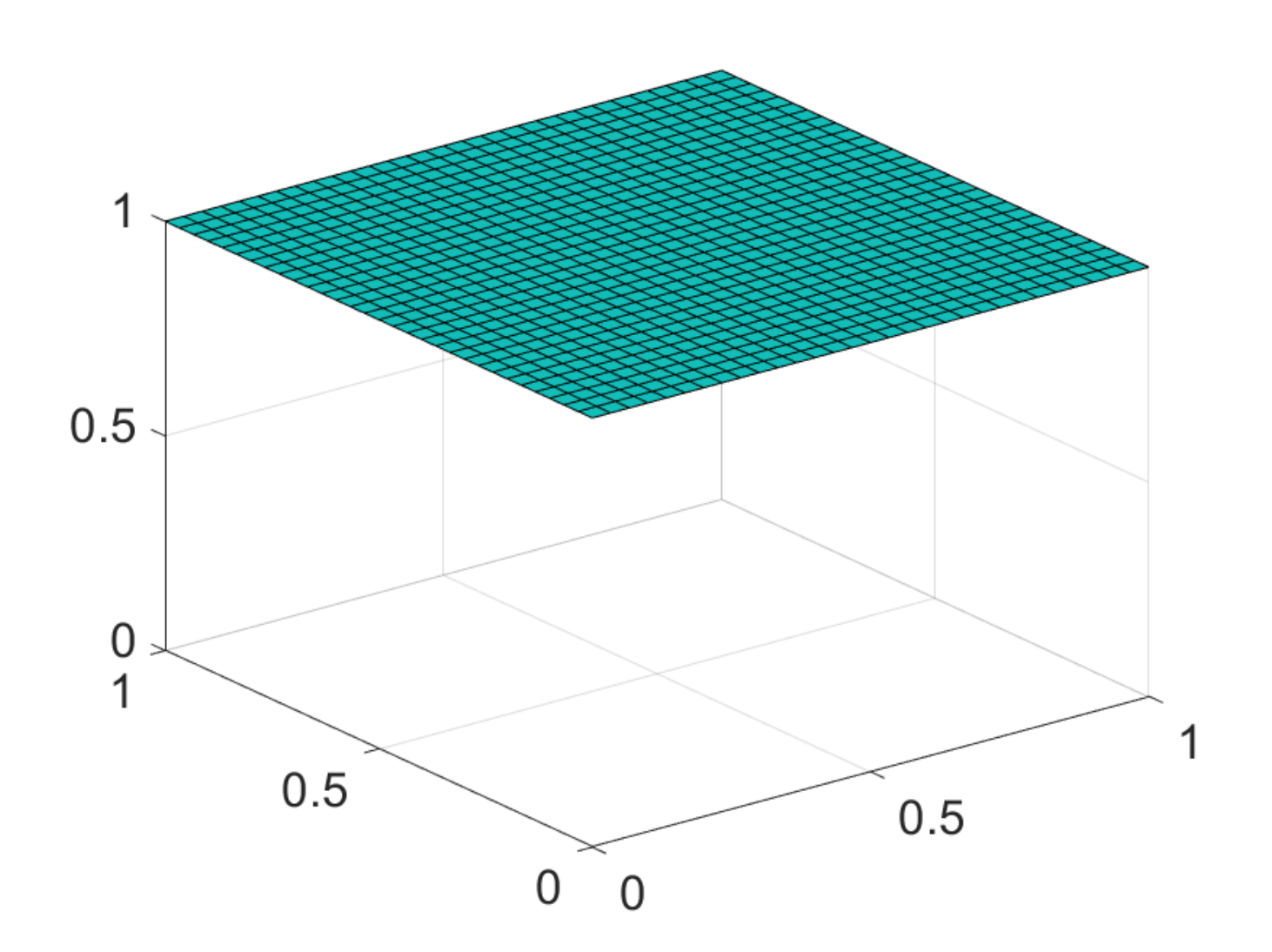}
	\includegraphics[scale = 0.27]{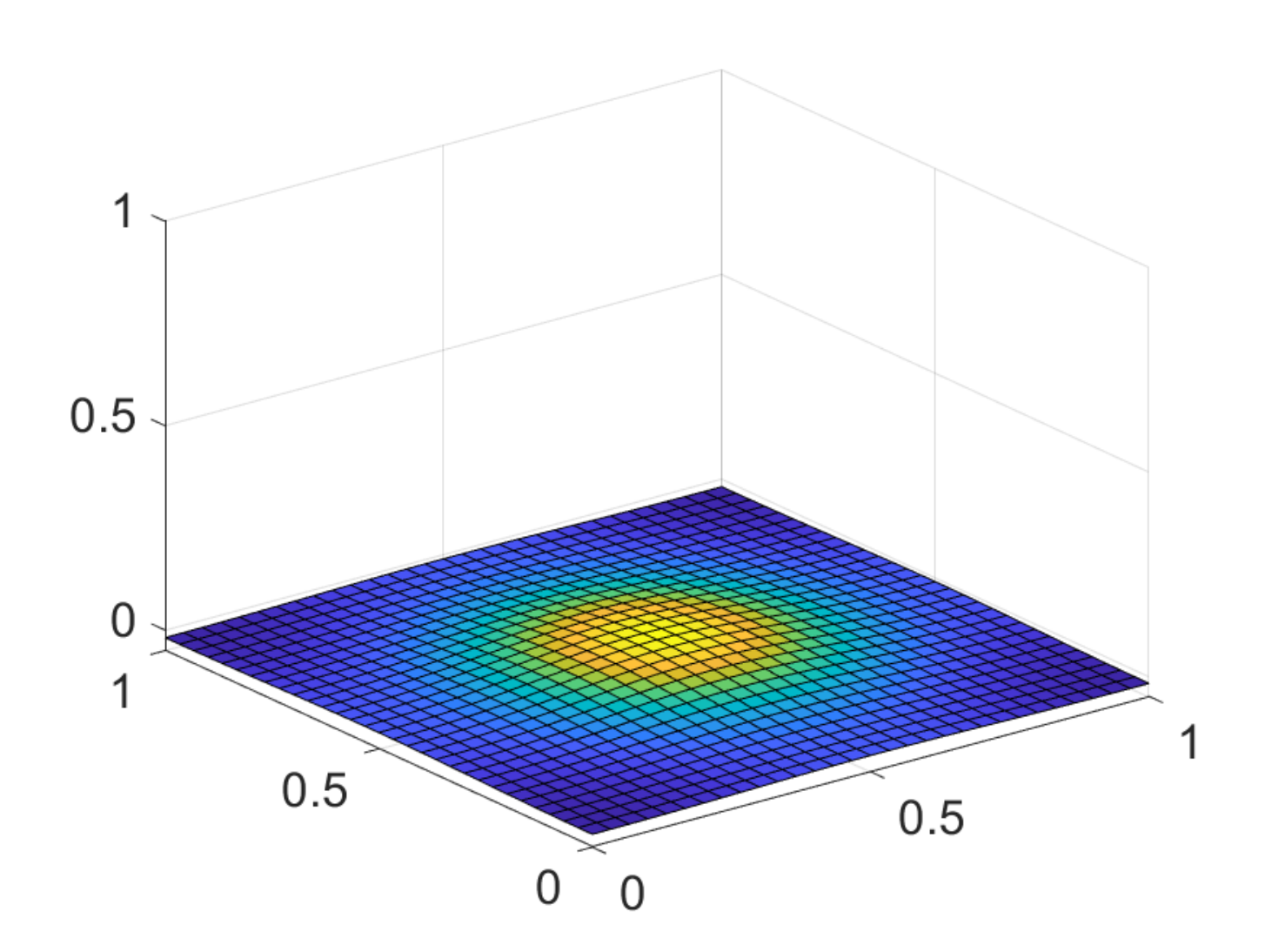}
		\caption{Test 3. Left: Initial condition $y(x,0)=\max \{- (2|x-0.5|+1)(2|y-0.5|+1)+2,0\}$. Middle: uncontrolled solution. Right: controlled solution.} 
\end{figure}

Figure \ref{initial_pyramidal} shows the uncontrolled trajectory, which converges to $y(x)=1$ and the controlled trajectory, which converges to the unstable equilibrium $y(x)=0$ as desired. Figure \ref{initial_pyramidal2} also presents the optimal control and the evaluation of the cost functional.

\begin{figure}[H]
	\label{initial_pyramidal2}
	\centering
	\includegraphics[scale = 0.3]{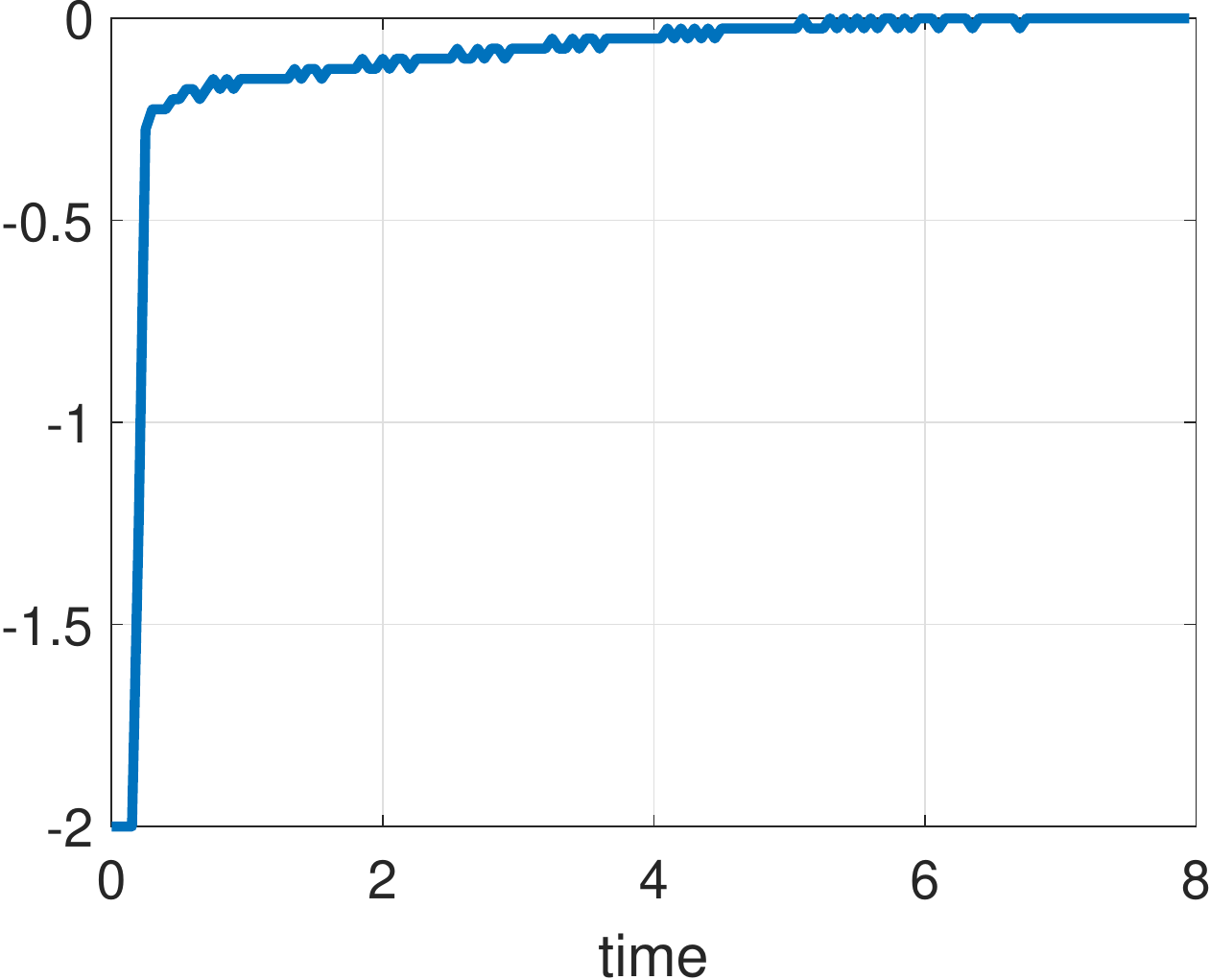}
	\includegraphics[scale = 0.3]{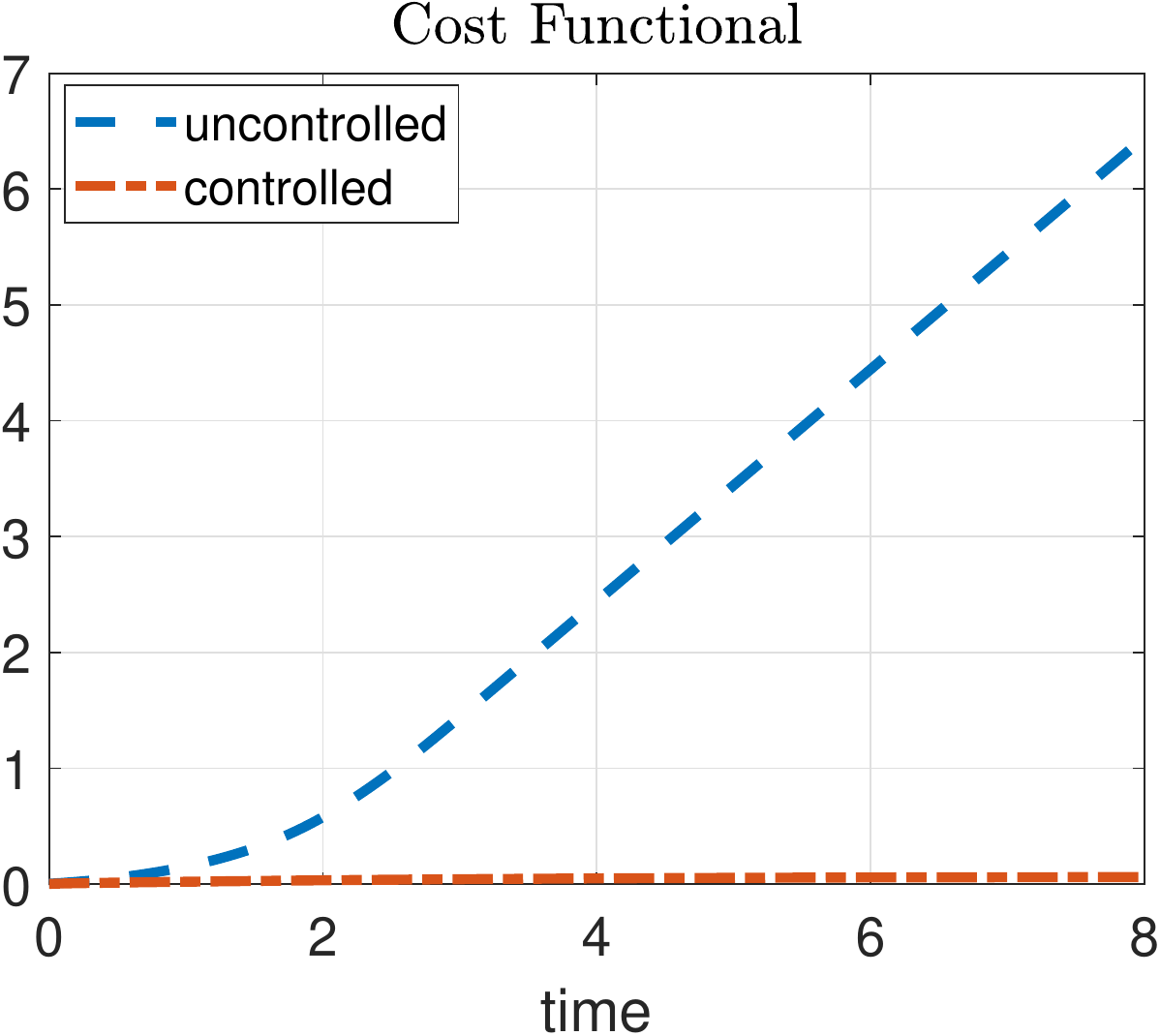}
		\caption{Test 3.  Left: optimal control and (right) running Cost with initial condition $y(x,0)=\max \{- (2|x-0.5|+1)(2|y-0.5|+1)+2,0\}$.}
\end{figure}

\section{Conclusions and future works} \label{Section6}
We have proposed a novel method to approximate stationary HJB equations using RBF and Shepard's approximation. The RBF usually presents a parameter which is here selected by means of a-posteriori criteria. We have also shown a new way to generate the grid which helps to deal with high dimensional problems localizing the problem along trajectories of interest. This method has the advantage to be able to reconstruct the feedback for (a class of) initial conditions without the need of update the mesh or recompute the HJB approximation as usually happens for the control of PDE with a DP approach. We have also provided an error estimate for the value function and proved, by numerical evidence, the effectiveness of our method.
The method could be extended to 
the semi-Lagrangian schemes where the interpolation is needed to reconstruct the characteristic of the problem. Furthermore, it will be our interest to couple this method with model order reduction to deal with more sophisticated PDE example, and possibly industrial applications.

\end{document}